\documentclass[reqno,11pt]{amsart}
\usepackage{amsmath,amssymb,latexsym,soul,cite,mathrsfs}
\usepackage{color,enumitem,graphicx}
\usepackage[colorlinks=true,urlcolor=blue,
citecolor=red,linkcolor=blue,linktocpage,pdfpagelabels,
bookmarksnumbered,bookmarksopen]{hyperref}
\usepackage[english]{babel}
\usepackage[left=2.6cm,right=2.6cm,top=2.9cm,bottom=2.9cm]{geometry}
\usepackage[hyperpageref]{backref}

\newtheorem{theorem}{Theorem}
\newtheorem{lemma}[theorem]{Lemma}
\newtheorem{corollary}[theorem]{Corollary}
\newtheorem{proposition}[theorem]{Proposition}
\newtheorem{remark}[theorem]{Remark}
\newtheorem{definition}[theorem]{Definition}

\newtheorem{theoremletter}{Theorem}

\newtheorem{lemmaletter}{Lemma}

\newenvironment{acknowledgement}{\noindent\textbf{Acknowledgments}}{}
\newenvironment{program}{\noindent{\textbf{Program.}}}{}
\newtheoremstyle{tttheorem}
{}                
{}                
{\slshape}        
{}                
{\bfseries}       
{'}               
{ }               
{}                
\theoremstyle{tttheorem}

\newcommand{\ud}{\mathrm{d}}

\newcommand*{\avint}{\mathop{\ooalign{$\int$\cr$-$}}}
\newcommand{\quotes}[1]{``#1''}

\title[Qualitative properties for solutions to  fourth order systems]{Qualitative properties for solutions to subcritical fourth order systems}  
\thanks{Research supported in part by Fulbright Commission in Brazil grant G-1-00001, CNPq grant 305726/2017-0, and the Coordena\c c\~ao de Aperfei\c coamento de Pessoal de N\'ivel Superior - Brasil (CAPES) grant 88882.440505/2019-01}

\author[J.H. Andrade]{Jo\~{a}o Henrique\ Andrade}
\author[J.M. do \'O]{Jo\~ao Marcos do \'O*}

\address[J.H. Andrade]{Department of Mathematics, 
	Federal University of Para\'{\i}ba
	\newline\indent 
	58051-900, Jo\~ao Pessoa-PB, Brazil}
\email{\href{mailto:andradejh@mat.ufpb.br}{andradejh@mat.ufpb.br}}

\address[J.M. do \'O]{Department of Mathematics,
	Federal University of Para\'{\i}ba
	\newline\indent 
	58051-900, Jo\~ao Pessoa-PB, Brazil}
\email{\href{mailto:jmbo@pq.cnpq.br}{jmbo@pq.cnpq.br}}

\thanks{* Corresponding author.}
\begin{document}
	
	\begin{abstract}
	We prove some qualitative properties for singular solutions to a class of strongly coupled  system involving a Gross--Pitaevskii-type nonlinearity. 
	Our main theorems are vectorial fourth order counterparts of the classical results of [J. Serrin, Acta Math. (1964)], [P.-L. Lions, J. Differential Equations (1980)], [P. Aviles, Comm. Math. Phys. (1987)], and [B. Gidas and J. Spruck, Comm. Pure Appl. Math. (1981)]. On the technical level, we use the moving sphere method to classify the limit blow-up solutions to our system.
	Besides, applying asymptotic analysis techniques, we show that these solutions are indeed the local models near the isolated singularity.
	We also introduce a new fourth order nonautonomous Pohozaev functional, whose monotonicity properties yield improvement for the asymptotics results due to [R. Soranzo, Potential Anal. (1997)].
		\end{abstract}
	
	\maketitle
	
	
	\begin{center}
		\footnotesize
		\tableofcontents
	\end{center}
	
	\section{Description of the results}\label{sec:introduction}
	
	In this work, we study the classification and the local behavior for nonnegative {\it $p$-map solutions}  $\mathcal{U}=(u_1,\dots,u_p):B_R^* \rightarrow \mathbb{R}^p$ 
	to the fourth order system,
	\begin{flalign}\tag{$\mathcal{S}_{p}$}\label{eq:subcriticalsystem}
	\Delta^{2} u_{i}=|\mathcal{U}|^{s-1}u_{i} \quad {\rm in} \quad B_R^*.
	\end{flalign}
	Here  $B_R^*:=B_R^{n}(0)\setminus\{0\}\subset\mathbb{R}^n$ is the {\it punctured ball} (resp. $B_{\infty}^*:=\mathbb{R}^n\setminus\{0\}$ is the {\it punctured space}), $n\geqslant5$, $|\mathcal{U}|^2=\sum_{i=1}^pu_i^2$, and $\Delta^{2}$ is the bi-Laplacian.
	System \eqref{eq:subcriticalsystem} is strongly coupled by the {\it Gross--Pitaevskii nonlinearity}  $f^s_i(\mathcal{U})=|\mathcal{U}|^{s-1}u_i$ with associated potential $F^s(\mathcal{U})=(f^s_1(\mathcal{U}),\dots,f^s_p(\mathcal{U}))$, where $s\in(1,2^{**}-1)$ with $2^{**}=2n/(n-4)$ the ({\it upper}) {\it critical Sobolev exponent}.
	
	By a {\it classical solution} to \eqref{eq:subcriticalsystem}, we 	mean a $p$-map $\mathcal{U}$ such that each component $u_i \in C^{4,\zeta}(B_R^*)$, for some $\zeta\in(0,1)$, and solves \eqref{eq:subcriticalsystem} in the classical sense (see Remark~\ref{rmk:regularity}).
	A solution may develop an isolated singularity at the origin, that is, some components may have a non-removable singularity when $x=0$. 
	More accurately, a solution $\mathcal{U}$ to \eqref{eq:subcriticalsystem} is said to be {\it singular}, if there exists $i\in I:=\{1,\dots,p\}$ such that the origin is a {\it non-removable singularity} for $u_{i}$. Otherwise, $\mathcal{U}$ is called {\it non-singular}, if the origin is a {\it removable singularity} for all components $u_i$, that is, $u_i$ can be extended continuously to the whole domain. 
	We also say that a $p$-map solution $\mathcal{U}$ is {\it nonnegative} ({\it strongly positive}) when $u_i\geqslant0$ ($u_i>0$) and $\mathcal{U}$ is {\it superharmonic} in case $-\Delta u_i>0$ for all $i\in I$. 
	Furthermore, when either $u_i>0$ or $u_i\equiv0$ for any $i\in I$, a solution $\mathcal{U}$ is called {\it weakly positive}. 
	Notice that, by the maximum principle, superharmonic solutions are weakly positive (see Remark~\ref{rmk:positiveness}).
	
	To make the exposition more comprehensible, we split our approach into three cases, namely $s\in(1,2_{**})$, $s=2_{**}$, and $s\in(2_{**},2^{**}-1)$, where $2_{**}={n}/{(n-4)}$ is the {\it lower Sobolev exponent} (or {\it Serrin exponent}).
	More precisely, $2_{**}$ is the greatest exponent for which all nonnegative solutions to \eqref{eq:subcriticalsystem} are trivial; this is a fourth order analog of the one found by J. Serrin \cite{MR0170096} in the context of second order quasilinear problems.
	
	Our first main result provides qualitative information for both non-singular and singular solutions to \eqref{eq:subcriticalsystem}. We classify the limit blow-up solutions, that is, $R=\infty$, which are interesting by themselves, since this type of solution can be useful in some related topics, such as in the theory of phase transition, in free boundary problems, and minimal hypersurface theory (see \cite{MR2757359} and the references therein).
	
	\begin{theorem}[Classification]\label{Thm3.1:classification}
		Let $R=\infty$ and $\mathcal{U}$ be a nonnegative solution to \eqref{eq:subcriticalsystem}. Assume that\\
		\noindent{\rm (i)} the origin is a removable singularity, then $\mathcal{U}\equiv0$.\\
		\noindent{\rm (ii)} the origin is a non-removable singularity. 
		\begin{itemize}
			\item[{\rm (a)}] If $s\in(1,2_{**}]$, then $\mathcal{U}\equiv0$; 
			\item[{\rm (b)}] If $s\in(2_{**},2^{**}-1)$, then there exists $\Lambda\in\mathbb{S}^{p-1}_{+}=\{ x \in \mathbb{S}^{p-1} : x_i\geqslant 0 \}$ such that 
			\begin{equation}\label{gidassprucklimitsolution}
			\mathcal{U}(x)=\Lambda K_0(n,s) ^{\frac{1}{s-1}}|x|^{-\frac{4}{s-1}},
			\end{equation}
		\end{itemize} 
		where $K_0(n,s)$ is defined by \eqref{autonomouscoefficients}.
	\end{theorem}
	
	On our second main result, we classify the local behavior for solutions to \eqref{eq:subcriticalsystem} in the punctured ball of radius $R<\infty$.
	More precisely, when $s\in(1,2_{**})$ the norm of the solution $\mathcal{U}$ grows like the fundamental solution to the bi-Laplacian, and the origin shall be a removable singularity.
	In contrast, when $s\in(2_{**},2^{**}-1)$, we show that the blow-up limit solutions in \eqref{gidassprucklimitsolution} are the asymptotics models of \eqref{eq:subcriticalsystem} near the isolated singularity.
	In addition, since $K_0(n,s)\equiv0$ when $s=2_{**}$, a completely new asymptotics with a slow growing logarithmic term is proved, which is a novelty even for the scalar case $p=1$.
	
	\begin{theorem}[Asymptotics]\label{Thm3.2:asymptotics}
		Let $R<\infty$ and $\mathcal{U}$ be a nonnegative superharmonic singular solution to \eqref{eq:subcriticalsystem}.
		Then, it follows
		\begin{equation*}
		|\mathcal{U}(x)|=(1+\mathcal{O}(|x|))|\overline{\mathcal{U}}(x)| \quad {\rm as} \quad x\rightarrow0,
		\end{equation*}
		where $|\overline{\mathcal{U}}(r)|=\avint_{\partial B_{1}} |\mathcal{U}(r \theta)|\ud\theta$ is the spherical average of $|\mathcal{U}|$. 
		Moreover, 
		\begin{itemize}
			\item[{\rm (a)}]  if $s\in(1,2_{**})$, then 
			\begin{equation*}
			|\mathcal{U}(x)|\simeq|x|^{4-n} \quad {\rm as} \quad x\rightarrow0;
			\end{equation*}
			\item[{\rm (b)}] if $s=2_{**}$, then 
			\begin{equation*}
			|\mathcal{U}(x)|=(1+o(1))\widehat{K}_{0}(n)^{\frac{n-4}{4}}|x|^{4-n}(-\ln|x|)^{\frac{4-n}{4}} \quad {\rm as} \quad x\rightarrow0,
			\end{equation*}
			where
			\begin{equation*}
			\widehat{K}_{0}(n)=\frac{(n-4)(n-2)(n+4)}{2};
			\end{equation*}
			\item[{\rm (c)}] if $s\in(2_{**},2^{**}-1)$, then 
			\begin{equation*}
			|\mathcal{U}(x)|=(1+o(1))K_0(n,s)^{\frac{1}{s-1}}|x|^{-\frac{4}{s-1}} \quad {\rm as} \quad x\rightarrow0.
			\end{equation*}
		\end{itemize}
	\end{theorem}
	
	Theorems~\ref{Thm3.1:classification} and \ref{Thm3.2:asymptotics} complete the asymptotic classification for system \eqref{eq:subcriticalsystem} in the sense of the celebrated works \cite{MR0170096,MR605060,MR615628,MR875297,MR982351}. 
	Some recent progress has been achieved on the critical case ($s=2^{**}-1$). 
	In this setting, System \eqref{eq:subcriticalsystem} is related to conformal geometry, being the vectorial extension of the constant $Q$-curvature equation. 
	Inspired by the works of L. A. Caffarelli, B. Gidas and J. Spruck \cite[Theorem~1.2]{MR982351}, and N. Korevaar, R. Mazzeo, F. Pacard and R. Schoen \cite[Theorem~1]{MR1666838} on the singular Yamabe equation, we present some qualitative results from \cite[Theorems~1 and 2]{arXiv:2002.12491} and \cite[Theorem~1']{arXiv:2003.03487}. 
	In this critical case, for geometrical reasons, \eqref{eq:subcriticalsystem} is often considered with a normalizing constant $c(n)>0$.
	
	\begin{theoremletter}\label{thm:andrade-do2020}
		Let $R=\infty, s=2^{**}-1$, and  $\mathcal{U}$ be a nonnegative solution to \eqref{eq:subcriticalsystem}. \\
		\noindent{\rm (i)} If the origin is a removable singularity. Then, $\mathcal{U}$ is weakly positive and radially symmetric about some $x_0\in\mathbb{R}^n$. Moreover, there exist $\Lambda \in\mathbb{S}^{p-1}_{+}=\{ x \in \mathbb{S}^{p-1} : x_i \geqslant 0 \}$ and a fourth order spherical solution $u_{x_0,\mu}$ $($see \eqref{eq:fourthbubble}$)$ such that
		\begin{equation*}
		\mathcal{U}=\Lambda u_{x_0,\mu}.
		\end{equation*}
		\noindent{\rm (ii)} 
		If the origin is a non-removable singularity. Then, $\mathcal{U}$ is strongly positive, radially symmetric about the origin and decreasing. Moreover, there exist $\Lambda^*\in\mathbb{S}^{p-1}_{+,*}=\{ x \in \mathbb{S}^{p-1} : x_i > 0 \}$ and an Emden--Fowler solution $u_{a,T}$  $($see \eqref{eq:fourthorderemdenfwoler}$)$ such that 
		\begin{equation*}
		\mathcal{U}=\Lambda^*u_{a,T}.
		\end{equation*}
		In addition, when $R<\infty$ and $\mathcal{U}$ is a strongly positive superharmonic solution to \eqref{eq:subcriticalsystem}. Then, either the origin is a removable singularity, or there exist a deformed Emden--Fowler solution $u_{a,T,0}$ and $\beta^*_1>1$ such that 
		\begin{equation*}
		|\mathcal{U}(x)|=(1+\mathcal{O}(|x|^{\beta^*_1}))u_{a,T,0}(|x|) \quad {\rm as} \quad x\rightarrow0.
		\end{equation*}
	\end{theoremletter}
	
	Notice that when $p=1$, \eqref{eq:subcriticalsystem} reduces to the following nonlinear fourth order equation,
	\begin{flalign}\tag{$\mathcal{S}_{1}$}\label{eq:subcriticalequation}
	\Delta^{2} u=u^{s} \quad {\rm in} \quad B_R^*.
	\end{flalign}
  	Now we present a holistic picture of the classification and local behavior for solutions to this equation.
	Namely, the next theorem summarizes some recent contributions due to C. S. Lin \cite[Theorem~1.3]{MR1611691}, Z. Guo, J. Wei and F. Zhou \cite[Theorem~1.2]{MR3632218}, R. Frank and T. K\"{o}nig \cite[Theorem~2]{MR3869387}, R. Soranzo \cite[Theorems~3 and 5]{MR1436822}, \cite[Theorem~2]{MR3869387}, H. Yang \cite[Theorem~1.1]{MR4123335}, T. Jin and J. Xiong \cite[Theorem~1.1]{arxiv:1901.01678} and J. Ratzkin \cite[Theorem~1]{arxiv.2001.07984}
	
	\begin{theoremletter}\label{thm:lin-soranzo-yang-frank-konig-ratzkin}
		Let $u$ be a nonnegative solution to \eqref{eq:subcriticalequation}. Assume that.\\ 
		\noindent{\rm Case (I):} $($punctured space$)$ $R=\infty$, and
		\begin{enumerate}
			\item[{\rm (i)}] the origin is a removable singularity. 
			\begin{itemize}
				\item[{\rm (a)}] If $s\in(1,2^{**}-1)$, then $u\equiv0$;
				\item[{\rm (b)}] If $s=2^{**}-1$, then there exist $x_0\in\mathbb{R}^n$ and $\mu>0$ such that $u$ is radially symmetric about $x_0$ and, up to a constant, is given by 
				\begin{equation}\label{eq:fourthbubble}
				u_{x_0,\mu}(x)=\left(\frac{2\mu}{1+\mu^{2}|x-x_0|^{2}}\right)^{\frac{n-4}{2}}.
				\end{equation}
				These are called the {\it $($fourth order$)$ spherical solutions} $($or bubbles$)$.
			\end{itemize}
			\item[{\rm (ii)}]  the origin is a non-removable singularity.
			\begin{itemize}
				\item[{\rm (a)}] If $s\in(1,2_{**}]$, then $u\equiv0$;
				\item[{\rm (b)}] If $s\in(2_{**},2^{**}-1)$, then   
				\begin{equation}\label{limitavilesffourthorder}
				u(x)=K_0(n,s)^{\frac{1}{s-1}}|x|^{-\frac{4}{s-1}};
				\end{equation}
				\item[{\rm (c)}] If $s=2^{**}-1$, then $u$ is radially symmetric about the origin. Moreover, there exist $a \in (0,a_0]$ and $T\in (0,T_a]$ such that
				\begin{equation}\label{eq:fourthorderemdenfwoler}
				u_{a,T}(x)=|x|^{\frac{4-n}{2}}v_{a}(\ln|x|+T).
				\end{equation}
				Here $a_0=[n(n-4)/(n^2-4)]^{n-4/8}$, $T_a\in\mathbb{R}$ is the fundamental period of the unique $T$-periodic bounded solution $v_a$ to the following fourth order Cauchy problem, 
				\begin{equation*}
				\begin{cases}
				v^{(4)}+K^*_2v^{(2)}+K^*_0v=c(n)v^{2^{**}-1}\\
				v(0)=a,\ v^{(1)}(0)=0,\ v^{(2)}(0)=b(a),\ v^{(3)}(0)=0,
				\end{cases}
				\end{equation*}
				where $K^*_2,K^*_0$ are constants depending only on $n$ and $b(a)$ is determined by $a \in (0,a_0]$. 
				We call both $u_{a,T}$ and $v_{a,T}$ $($fourth order$)$ Emden--Fowler $($or Delaunay-type$)$ solutions and $a\in(0,a_0)$ its Fowler parameter, chosen  to satisfy $a=\min_{t>0}v_a(t)$ $($Appendix~\ref{app:computations}$)$.
			\end{itemize}    
		\end{enumerate}
		\noindent{\rm Case (II):} $($punctured ball$)$ $R<\infty$, and the origin is a non-removable singularity. 
		Suppose that $u$ is superharmonic.  
		Then, $u(x)=(1+\mathcal{O}(|x|))\overline{u}(|x|)$ as $x\rightarrow0$, where $\overline{u}(r)=\avint_{\partial B_{1}} u(r \theta)\ud\theta$ is the spherical average of $u$. Moreover, 
		\begin{itemize}
			\item[$(a)$] if $s\in(1,2_{**})$, then $u(x)\simeq|x|^{4-n}$ as $x\rightarrow0$;
			\item[$(b)$] if $s=2_{**}$, then there exists $C_0(n)>0$ and $0<r_0<R$ such that
			\begin{equation}\label{upperboundaviles}
			|u(x)|\leqslant C_0(n)|x|^{4-n}(-\ln|x|)^{\frac{4-n}{4}} \quad {\rm for} \quad 0<|x|<r_0;
			\end{equation}
			\item[$(c)$] if $s\in(2_{**},2^{**}-1)$, then
			\begin{equation}\label{avilesasymptoticsscalar}
			u(x)=(1+o(1))K_0(n,s)^{\frac{1}{s-1}}|x|^{-\frac{4}{s-1}} \quad {\rm as} \quad x\rightarrow0;
			\end{equation}
			\item[$(d)$] if $s=2^{**}-1$, then there exists $u_{a,T}$ as in \eqref{eq:fourthorderemdenfwoler} such that
			\begin{equation*}
			u(x)=(1+o(1))u_{a,T}(|x|) \quad {\rm as} \quad x\rightarrow0.
			\end{equation*}    
			Furthermore, one can find a deformed Emden--Fowler solution $u_{a,T,0}$ and $\beta_1^*>1$ such that
			\begin{equation*}
			u(x)=(1+\mathcal{O}(|x|^{\beta_1^*}))u_{a,T,0}(|x|) \quad {\rm as} \quad x\rightarrow0.
			\end{equation*}  
		\end{itemize}
	\end{theoremletter}
	
	\begin{remark}
		In \eqref{upperboundaviles} $($see \cite[Theorem~5]{MR1436822}$)$ the upper bound estimate does not have an explicit constant. 
		We obtain a lower bound estimate and the \quotes{sharp constant} in the more general vectorial setting. 
		To this end, instead of using the techniques in \cite{MR605060}, we define a nonautonomous cylindrical transformation inspired by \cite[Section~3]{MR875297} $($see \eqref{nonautonomouscyltransform}$)$. 
		This new computation also allows us to compute $\widehat{K}_{0}(n)>0$ $($see  Lemma~\ref{lm:limitinglevels}$)$, and can be further adapted for a plethora of problems.
	\end{remark}
	
	\begin{remark}
		When $s\in(2_{**},2^{**}-1)$, it is well known that \eqref{limitavilesffourthorder}
		is an exact positive $($unstable$)$ singular solution to \eqref{eq:subcriticalsystem}, which provides an example of solution to \eqref{eq:subcriticalequation} that obeys the asymptotic behavior \eqref{avilesasymptoticsscalar} $($for another class of solutions, see \cite{MR3632218}$)$.
		In this case, the superharmonic condition can be dropped \cite{MR4123335}, which is indeed required to run the asymptotic analysis in the critical case $s=2^{**}-1$ \cite{arxiv:1901.01678}.
		More specifically, this hypothesis is essential to obtain the lower bound estimate near the origin.
	\end{remark}
	
	Next, we compare our preliminary results in the fourth order setting with their second order counterparts.
	In this direction, for $s\in(1,2^*-1]$ and $n\geqslant3$, where $2^{*}=2n/(n-2)$ and $2_*=n/(n-2)$ are respectively the {\it upper and lower critical Sobolev exponents}. 
	Let us consider nonnegative {\it $p$-map solutions}  $\mathcal{U}=(u_1,\dots,u_p):B_R^* \rightarrow \mathbb{R}^p$ to the following second order system analog to \eqref{eq:subcriticalsystem}, 
	\begin{equation}\label{eq:subcriticalsystemsecondorder}
	-\Delta u_{i}=|\mathcal{U}|^{s-1}u_{i} \quad {\rm in} \quad B_R^*.
	\end{equation}
	
	On this system, the results of O. Druet, E. Hebey and J. V\'etois \cite[Proposition~1.1]{MR2558186}, M. Ghergu, S. Kim  and H. Shahgholian \cite[Theorems~1.1--1.5]{MR4085120} and R. Caju, J. M. do \'O and A. Santos \cite[Theorem~1.2]{MR4002167} provided the following classification and asymptotics,
	
	\begin{theoremletter}\label{thm:druet-hebey-vetois-ghergu-kim-shahgohlian-caju-doo-santos}
		Let $\mathcal{U}$ be a nonnegative solution to \eqref{eq:subcriticalsystemsecondorder}. Assume that\\ 
		\noindent{\rm Case (I):} $($punctured space$)$ $R=\infty$, and
		\begin{enumerate}
			\item[{\rm (i)}] the origin is a removable singularity. 
			\begin{itemize}
				\item[{\rm (a)}] If $s\in(1,2^{*}-1)$, then $\mathcal{U}\equiv0$;
				\item[{\rm (b)}] If $s=2^{*}-1$, then there exist $\Lambda\in\mathbb{S}^{p-1}_{+}$, $x_0\in\mathbb{R}^n$ and $\mu>0$ such that $\mathcal{U}$ is radially symmetric about $x_0$ and, up to a constant, is given by 
				\begin{equation*}
				\mathcal{U}(x)=\Lambda\left(\frac{2\mu}{1+\mu^{2}|x-x_0|^{2}}\right)^{\frac{n-2}{2}}. 
				\end{equation*}
			\end{itemize}
			\item[{\rm (ii)}] the origin is a non-removable singularity. 
			\begin{itemize}
				\item[{\rm (a)}] If $s\in(1,2_{*}]$, then $\mathcal{U}\equiv0$;
				\item[{\rm (b)}] If $s\in(2_{*},2^{*}-1)$, then there exists $\Lambda^*\in\mathbb{S}^{p-1}_{+,*}$ such that 
				\begin{equation*}
				\mathcal{U}(x)=\Lambda^*\left[\frac{2(n-2)(s-2_{*})}{(s-1)^2}\right]^{\frac{1}{s-1}}|x|^{-\frac{2}{s-1}};
				\end{equation*}
				\item[{\rm (c)}] If $s=2^{*}-1$ and $\mathcal{U}$ is radially symmetric about the origin. Moreover, there exist  $\Lambda^*\in\mathbb{S}^{p-1}_{+,*}$, 
				$a \in (0,[(n-2)/n]^{(n-2)/4}]$ and $T\in (0,T_a]$ such that
				\begin{equation*}
				\mathcal{U}(x)=\Lambda^*|x|^{\frac{2-n}{2}}v_{a}(-\ln|x|+T).
				\end{equation*}
				Here $v_{a,T}$ is the unique $T$-periodic bounded solution to the following second order problem 
				\begin{equation}\label{eq:secondorderode}
				\begin{cases}
				v^{(2)}-\frac{(n-2)^2}{4}v+\frac{n(n-2)}{4}v^{2^{*}-1}=0\\
				v(0)=a,\ v^{(1)}(0)=0,
				\end{cases}
				\end{equation}
				where $T_a\in\mathbb{R}$ is the fundamental period of $v_a$; both $u_{a,T}$ and $v_{a,T}$ are called the $($second order$)$ Emden--Fowler $($or Delaunay type$)$ solutions.
			\end{itemize}    
		\end{enumerate}
		\noindent{\rm Case (II):} $($punctured ball$)$ $R<\infty$, and the origin is a non-removable singularity, then
		\begin{equation*}
		|\mathcal{U}(x)|=(1+\mathcal{O}(|x|))|\overline{\mathcal{U}}(x)| \quad {\rm as} \quad x\rightarrow0,
		\end{equation*}
		where $|\overline{\mathcal{U}}(r)|=\avint_{\partial B_{1}} |\mathcal{U}(r \theta)|\ud\theta$ is the spherical average of $|\mathcal{U}|$. Moreover, 
		\begin{itemize}
			\item[{\rm (a)}] if $s\in(1,2_{*}]$, then $|\mathcal{U}(x)|\simeq |x|^{2-n}$ as $x\rightarrow0$;
			\item[{\rm (b)}] if $s=2_{*}$, then 
			\begin{equation*}
			|\mathcal{U}(x)|=(1+o(1))\left(\frac{n-2}{\sqrt{2}}\right)^{n-2}|x|^{2-n}(-\ln|x|)^{\frac{2-n}{2}} \quad {\rm as} \quad x\rightarrow0;
			\end{equation*}
			\item[{\rm (c)}] if $s\in(2_{*},2^{*}-1)$, then
			\begin{equation*}
			|\mathcal{U}(x)|=(1+o(1))\left[\frac{2(n-2)(s-2_{*})}{(s-1)^2}\right]^{\frac{1}{s-1}}|x|^{-\frac{2}{s-1}} \quad {\rm as} \quad x\rightarrow0;
			\end{equation*}
			\item[{\rm (d)}] if $s=2^{*}-1$, then there exists a second order Emden--Fowler $u_{a,T}$ as in \eqref{eq:secondorderemdenfowler} such that
			\begin{equation*}
			|\mathcal{U}(x)|=(1+o(1))u_{a,T}(|x|) \quad {\rm as} \quad x\rightarrow0.
			\end{equation*}  
			Furthermore, one can find a deformed Emden--Fowler solution $u_{a,T,0}$ and $\beta_1^*>1$ such that
			\begin{equation*}
			|\mathcal{U}(x)|=(1+\mathcal{O}(|x|^{\beta_1^*}))u_{a,T,0}(|x|) \quad {\rm as} \quad x\rightarrow0.
			\end{equation*}  
		\end{itemize}
	\end{theoremletter}
	
All this analysis is motivated by some standard asymptotic classification results due to J. Serrin \cite[Theorem~11]{MR0170096}, P.-L. Lions \cite[Theorem~2]{MR605060}, P. Aviles \cite[Theorem~A]{MR875297}, B. Gidas and J. Spruck \cite[Theorems~1.1 and 1.2]{MR615628}, and L. A. Caffarelli et al. \cite[Theorems~1.1--1.3]{MR982351} with an improvement given by N. Korevaar et al. \cite[Theorem~1]{MR1666838},  which can be summarized as 

\begin{theoremletter}\label{thm:serrin-lions-aviles-caffarelli-gidas-spruck}
	Let $u$ be a nonnegative solution to \eqref{eq:subcriticalsystemsecondorder} with $p=1$. Assume that\\ 
	\noindent{\rm Case (I):} $($punctured space$)$ $R=\infty$, and
	\begin{enumerate}
		\item[{\rm (i)}] the origin is a removable singularity. 
		\begin{itemize}
			\item[{\rm (a)}] If $s\in(1,2^{*}-1)$, then $u\equiv0$;
			\item[{\rm (b)}] If $s=2^{*}-1$, then there exist $x_0\in\mathbb{R}^n$ and $\mu>0$ such that $u$ is radially symmetric about $x_0$ and, up to a constant, is given by 
			\begin{equation*}
			u_{x_0,\mu}(x)=\left(\frac{2\mu}{1+\mu^{2}|x-x_0|^{2}}\right)^{\frac{n-2}{2}};
			\end{equation*}
			those solutions are called the $($second order$)$ spherical solutions.
		\end{itemize}
		\item[{\rm (ii)}] the origin is a non-removable singularity. 
		\begin{itemize}
			\item[{\rm (a)}] If $s\in(1,2_{*}]$, then $u\equiv0$;
			\item[{\rm (b)}] If $s\in(2_{*},2^{*}-1)$, then
			\begin{equation*}
			u(x)=\left[\frac{2(n-2)(s-2_{*})}{(s-1)^2}\right]^{\frac{1}{s-1}}|x|^{-\frac{2}{s-1}};
			\end{equation*}
			\item[{\rm (c)}] If $s=2^{*}-1$ and $u$ is radially symmetric about the origin. Moreover, there exist $a \in (0,[(n-2)/n]^{(n-2)/4}]$ and $T\in (0,T_a]$ such that
			\begin{equation}\label{eq:secondorderemdenfowler}
			u_{a,T}(x)=|x|^{\frac{2-n}{2}}v_{a}(\ln|x|+T),
			\end{equation}
			where $v_a$ is a solution to \eqref{eq:secondorderode}. We call both $u_{a,T}$ and $v_{a,T}$ $($second order$)$ Emden--Fowler $($or Delaunay-type$)$ solutions
		\end{itemize}    
	\end{enumerate}
	\noindent{\rm Case (II):} $($punctured ball$)$ $R<\infty$, and the origin is a non-removable singularity, it follows that $u(x)=(1+\mathcal{O}(|x|))\overline{u}(|x|)$ as $x\rightarrow0$,
	where $\overline{u}(r)=\avint_{\partial B_{1}} u(r \theta)\ud\theta$ is the spherical average of $u$. Moreover, 
	\begin{itemize}
		\item[{\rm (a)}] $(${\it Serrin--Lions case}$)$ if $s\in(1,2_{*}-1]$, then $u(x)\simeq |x|^{2-n}$ as $x\rightarrow0$;
		\item[{\rm (b)}] $(${\it Aviles case}$)$ if $s=2_{*}-1$, then
		\begin{equation}\label{avilesasymptotics}
		u(x)=(1+o(1))\left(\frac{n-2}{\sqrt{2}}\right)^{n-2}|x|^{2-n}(-\ln|x|)^{\frac{2-n}{2}} \quad {\rm as} \quad x\rightarrow0;
		\end{equation}
		\item[{\rm (c)}] $(${\it Gidas--Spruck case}$)$ if $s\in(2_{*},2^{*}-1)$, then
		\begin{equation*}
		u(x)=(1+o(1))\left[\frac{2(n-2)(s-2_{*})}{(s-1)^2}\right]^{\frac{1}{s-1}}|x|^{-\frac{2}{s-1}} \quad {\rm as} \quad x\rightarrow0;
		\end{equation*}
		\item[{\rm (d)}] $(${\it Caffarelli--Gidas--Spruck case}$)$ if $s=2^{*}-1$, then there exists a second order Emden--Fowler solution $u_{a,T}$ as in \eqref{eq:secondorderemdenfowler} such that
		\begin{equation*}
		u(x)=(1+o(1))u_{a,T}(|x|) \quad {\rm as} \quad x\rightarrow0;
		\end{equation*}
		\item[{\rm (e)}] $(${\it Korevaar--Mazzeo--Pacard--Schoen case}$)$ Furthermore, one can find a deformed Emden--Fowler solution $u_{a,T,0}$ and $\beta_1^*>1$ such that
		\begin{equation}\label{korevaaretal}
		u(x)=(1+\mathcal{O}(|x|^{\beta_1^*}))u_{a,T,0}(|x|) \quad {\rm as} \quad x\rightarrow0.
		\end{equation}    
	\end{itemize}
\end{theoremletter}

\begin{remark}
	To analyze  the lower critical second order case $s=2_*$, in \cite{MR875297} it was introduced a new type of cylindrical coordinates $($see also  Appendix~\ref{app:computations}$)$, which leads to the following nonautonomous PDE on the cylinder $\mathcal{C}_T:=(-\ln R,\infty)\times \mathbb{S}^{n-1}$,
	\begin{equation}\label{eq:nonautonomoussecondorder}
	w^{(2)}+(n-2)\left(1-\frac{1}{t}\right)w^{(1)}-\frac{n-2}{2t}\left(n-2-\frac{n}{2t}\right)w+\Delta_{\theta}w+\frac{1}{t}w^{2_{*}}=0,
	\end{equation}
	where $w(t)=|x|^{2-n}(-\ln|x|)^{\frac{2-n}{2}}u(|x|)$ and $t=-\ln|x|$.
\end{remark}

In conclusion, we can summarize Theorems~\ref{Thm3.1:classification}, \ref{Thm3.2:asymptotics} and \ref{thm:andrade-do2020} in one research program, which is contained in the thesis \cite{andradethesis}.
\vspace{0.2cm}

\begin{program}\label{Thm:completeclassificationfourthorder}
	{\it 
	Let $\mathcal{U}$ be a nonnegative solution to \eqref{eq:subcriticalsystem}. Assume that\\ 
	\noindent{\rm Case (I):} $($punctured space$)$ $R=\infty$, and
	\begin{enumerate}
		\item[{\rm (i)}] the origin is a removable singularity. 
		\begin{itemize}
			\item[{\rm (a)}] If If $s\in(1,2^{**}-1)$, then $\mathcal{U}\equiv0$;
			\item[{\rm (b)}] If $s=2^{**}-1$, then there exist $\Lambda\in\mathbb{S}^{p-1}_{+}$, $x_0\in\mathbb{R}^n$ and $\mu>0$ such that $u$ is radially symmetric about $x_0$ and, up to a constant, is given by $\mathcal{U}(x)=\Lambda u_{x_0,\mu}(x)$, where $u_{x_0,\mu}$ is the $($fourth order$)$ spherical solution defined by \eqref{eq:fourthbubble}.
		\end{itemize}
		\item[{\rm (ii)}]  the origin is a non-removable singularity. 
		\begin{itemize}
			\item[{\rm (a)}] If $s\in(1,2_{**}]$, then $\mathcal{U}\equiv0$;
			\item[{\rm (b)}] If $s\in(2_{**},2^{**}-1)$, then there exists $\Lambda^*\in\mathbb{S}^{p-1}_{+,*}$ such that
			\begin{equation*}
			\mathcal{U}(x)=\Lambda^* K_0(n,s)^{\frac{1}{s-1}}|x|^{-\frac{4}{s-1}};
			\end{equation*}
			\item[{\rm (c)}] If $s=2^{**}-1$, then $u$ is radially symmetric about the origin. Moreover, there exist $\Lambda^*\in\mathbb{S}^{p-1}_{+,*}$, $a \in (0,a_0]$ and $T\in (0,T_a]$ such that $\mathcal{U}(x)=\Lambda^*u_{a,T}(x)$, where $u_{a,T}$ is the $($fourth order$)$ Emden--Fowler solution defined by \eqref{eq:fourthorderemdenfwoler}.
		\end{itemize}    
	\end{enumerate}
	\noindent{\rm Case (II):} $($punctured ball$)$ $R<\infty$, and the origin is a non-removable singularity. Suppose that $\mathcal{U}$ is superharmonic. Then, 
	\begin{equation*}
	|\mathcal{U}(x)|=(1+\mathcal{O}(|x|))|\overline{\mathcal{U}}(x)| \quad {\rm as} \quad x\rightarrow0,
	\end{equation*}
	where $|\overline{\mathcal{U}}(r)|=\avint_{\partial B_{1}} |\mathcal{U}(r \theta)|\ud\theta$ is the spherical average of $|\mathcal{U}|$. 
	Moreover, 
	\begin{itemize}
		\item[{\rm (a)}] if $s\in(1,2_{**})$, then 
		\begin{equation*}
		|\mathcal{U}(x)|\simeq |x|^{4-n} \quad {\rm as} \quad x\rightarrow0;
		\end{equation*}
		\item[{\rm (b)}] if $s=2_{**}$, then 
		\begin{equation*}
		|\mathcal{U}(x)|=(1+o(1))\left[\frac{(n-4)(n-2)(n+4)}{2}\right]^{\frac{n-4}{4}}|x|^{4-n}(-\ln|x|)^{\frac{4-n}{4}} \quad {\rm as} \quad x\rightarrow0;
		\end{equation*}
		\item[{\rm (c)}] if $s\in(2_{**},2^{**}-1)$, then
		\begin{equation*}
		|\mathcal{U}(x)|=(1+o(1))K_0(n,s)^{\frac{1}{s-1}}|x|^{-\frac{4}{s-1}} \quad {\rm as} \quad x\rightarrow0;
		\end{equation*}
		\item[{\rm (d)}] if $s=2^{**}-1$, then there exists $u_{a,T}$ as in \eqref{eq:fourthorderemdenfwoler} such that
		\begin{equation*}
		|\mathcal{U}(x)|=(1+o(1))u_{a,T}(|x|) \quad {\rm as} \quad x\rightarrow0.
		\end{equation*}    
		Furthermore, one can find a deformed Emden--Fowler solution $u_{a,T,0}$ and $\beta_1^*>1$ such that
		\begin{equation*}
		|\mathcal{U}(x)|=(1+\mathcal{O}(|x|^{\beta_1^*}))u_{a,T,0}(|x|) \quad {\rm as} \quad x\rightarrow0.
		\end{equation*}
	\end{itemize}
	}
\end{program}

	The main difference between the asymptotic analysis for the critical and subcritical regimes occurs because of the change on the monotonicity properties of the Pohozaev functional, which in this case works as a Lyapunov function, classifying the type of stability for solutions to \eqref{eq:subcriticalsystem} around a blow-up (shrink-down) limit solution. 
	This method is inspired by Fleming's tangent cone analysis for minimal hypersurfaces \cite{MR157263,MR3190428}).
	In the critical case, since the Pohozaev functional becomes constant, limit solutions are stable,  whereas, in the subcritical case, they are asymptotically stable. 
	This discrepancy is caused by the sign-changing behavior of the bi-Laplacian coefficients in cylindrical coordinates, which are suitable for this problem (see Remark~\ref{rmk:signoncoefficients}).  
	
	On the supercritical case $s\in(2^{**}-1,\infty)$, there are also several similar classification and asymptotics results in the scalar case \cite{MR2209261,MR2679617,MR2677898,MR1134481}.
	Recently, some parts of Theorem~\ref{thm:druet-hebey-vetois-ghergu-kim-shahgohlian-caju-doo-santos} were extended to the case of coupled systems with nonlinearities at the boundary (see \cite{li-bao,MR2288595,MR2745200}). We speculate that a result in this direction shall be true for system \eqref{eq:subcriticalsystem}.
	We also quote results like Theorem \ref{Thm:completeclassificationfourthorder} for higher order systems \cite{MR1432813,MR2534120,MR1679783}, for integral elliptic systems \cite{MR2200258,MR2131045,MR3918618,MR3562307}, and for fully nonlinear problems \cite{MR2214582,MR2247857,MR2737708}.
	One can also find a more geometric motivation to study \eqref{eq:subcriticalequation}. In the critical case ($s=2^{**}-1$), \eqref{eq:subcriticalequation} is equivalent to the constant $Q$-curvature equation (for more details, see \cite{MR3618119,MR3518237,MR1710786}). 
	In contrast, $s\in(1,2^{**}-1)$, \eqref{eq:subcriticalequation} is the constant $Q$-curvature problem on an underlying manifold that can be factored as $\mathbb{S}^1\times \mathbb{S}^{n-1}$ \cite{arXiv:2002.05939}.
	
	Besides their applications in conformal geometry, strongly coupled fourth order systems also appear in several parts of mathematical physics. 
	For instance, in hydrodynamics, for modeling the behavior of deep-water and Rogue waves in the ocean \cite{dysthe,lo-mei}, and in the Hartree--Fock theory for Bose--Einstein double condensates \cite{MR2040621,PhysRevLett.78.3594}. 
	Additionally, for $p=1$ the second order system \eqref{eq:subcriticalsystemsecondorder} becomes the  Lane--Emden--Fowler equation \cite{lane,emden,fowler}, which models the density of mass distribution for polytropic spherical stars in hydrostatic equilibrium \cite{MR0092663}.
	
	The strategy to prove Theorem~\ref{Thm3.1:classification} (i) is to use the Kelvin transform and a blow-up argument, based on moving spheres technique.
	For (ii), we use the scaling invariance of Pohozaev functional to study the behavior of limit solutions to \eqref{eq:subcriticalsystem} for blow-ups and shrink-downs limits. 
	The proof of Theorem~\ref{Thm3.2:asymptotics} is divided into two parts. We prove the asymptotic symmetry of singular solutions to \eqref{eq:subcriticalsystem} in the punctured ball. 
	Then, we use some ODE analysis and the monotonicity properties of the Pohozaev functional to study the asymptotic behavior for solutions on the cylinder.
	
	On the technical level, the study of system \eqref{eq:subcriticalsystem} has several difficulties.
	For instance, the fourth order operator implies a lack of strong maximum principle for solutions to \eqref{eq:subcriticalsystem}. 
	To deal with this, we use Green identity to convert \eqref{eq:subcriticalsystem} into an integral equation system, for which a type of maximum principle is available (see \cite{arxiv:1901.01678,MR2055032,MR4013228}).
	To deal with the nonlinear effects imposed by the coupling term on the right-hand side of \eqref{eq:subcriticalsystem}, we use some arguments from \cite{MR4085120,MR2558186,MR2603801}.
	
	Here is our plan for the rest of the paper. 
	In Section~\ref{sec:preliminaries}, we introduce some basic notation, an integral representation for \eqref{eq:subcriticalsystem}, the Kelvin transform, and both the autonomous and nonautonomous cylindrical transformations. 
	In Section~\ref{sec:pohozaevinvariant}, we define the associated Pohozaev functionals, and we prove their (asymptotic) monotonicity properties.
	In section~\ref{sec:blowlimitcase}, we use sliding techniques and a blow-up method to prove Theorem~\ref{Thm3.1:classification}.
	In section~\ref{sec:proof2}, we use the monotonicity formulas and some asymptotic analysis to prove Theorem~\ref{Thm3.2:asymptotics}.
	
	\section{Preliminaries}\label{sec:preliminaries}
	This section aims to introduce some necessary background definitions and results for developing the sliding methods and the asymptotic analysis that will be later used in this manuscript. 
	
	\subsection{Basic notations}
	Throughout this text, we adopt some notations, and that will be described as follows.
	Let $x_0\in\mathbb{R}\cup\{\pm\infty\}$, $u,\widetilde{u},f\in C_c^{\infty}(B_R^*)$ (resp. $v,\widetilde{v}\in C_c^{\infty}(\mathcal{C}_T)$) be positive functions and $\mathcal{U},\widetilde{\mathcal{U}}\in C_c^{\infty}(B_R^*,\mathbb{R}^p)$ ($\mathcal{V},\widetilde{\mathcal{V}}\in C_c^{\infty}(\mathcal{C}_T,\mathbb{R}^p)$) nonnegative $p$-maps, we denote
	\begin{itemize}
		\item $u=\mathcal{O}(f)$ as $x\rightarrow x_0$, if $\limsup_{x\rightarrow x_0}(u/f)(x)<\infty$;
		\item $u=o(f)$ as $x\rightarrow x_0$, if $\lim_{x\rightarrow x_0}(u/f)(x)=0$;
		\item $u\simeq\widetilde{u}$ as $x\rightarrow x_0$, if $u=\mathcal{O}(\widetilde{u})$ and $\widetilde{u}=\mathcal{O}(u)$ for $x_0\in\mathbb{R}\cup\{\pm\infty\}$;
		\item $\partial_{j}u={\partial u}/{\partial{x_j}}$ for $j\in\mathbb{N}$;
		\item $\partial_{r}^{(j)}u={\partial^{{j}} u}/{\partial r^j}$ for $j\in\mathbb{N}$ are the higher order radial derivatives;
		\item $\partial_{t}^{(j)}v={\partial^{j} v}/{\partial t^j}$ for $j\in\mathbb{N}$;
		\item $v^{(j)}={\ud^{j} v}/{\ud t^{j}}$ for $j\in\mathbb{N}$;
		\item $D^{(j)}\mathcal{U}=(D^{(j)}u_1,\dots,D^{(j)}u_p)$;
		\item $\mathcal{V}^{(j)}=(v_1^{(j)},\dots,v_p^{(j)})$;
		\item $\nabla_{\theta}$ (or $\nabla_{\sigma}$) is the tangential gradient;
		\item $\partial_{\nu}u={\partial u}/{\partial\nu}$ is the normal derivative;
		\item $\Delta_{\theta}=\Delta_{\mathbb{S}^{n-1}}$ ($\Delta_{\sigma}=\Delta_{\mathbb{S}^{n-1}}$) is the Laplace--Beltrami on the sphere $\mathbb{S}^{n-1}=\partial B_1^*$;
		\item $\langle\mathcal{U},\widetilde{\mathcal{U}}\rangle=\sum_{i=1}^pu_i\widetilde{u}_i$ is the inner product of $p$-maps.
		\item $\langle D^{(j)}\mathcal{U},D^{(j)}\widetilde{\mathcal{U}}\rangle=\sum_{i=1}^pD^{(j)}u_iD^{(j)}\widetilde{u}_i$ for $j\in\mathbb{N}$ is the inner product of higher order derivatives of $p$-maps.
	\end{itemize}
	
	Here and subsequently, we always deal with nonnegative solutions $\mathcal{U}$ to \eqref{eq:subcriticalsystem}, that is, $u_i \geqslant 0$ for all $i\in I$, where we recall the notation $I=\{1,\dots,p\}$. Let us split the index set $I$ into $I_0=\{i : u_i\equiv0\}$ and $I_{+}=\{i : u_i>0\}$. 
	Then, we present a standard definition on the positiveness of solution to elliptic systems.
	More precisely, let us divide solutions to \eqref{eq:subcriticalsystem} into two types:
	\begin{definition}
		Let $\mathcal{U}$ be a nonnegative solution to \eqref{eq:subcriticalsystem}. We call $\mathcal{U}$ {\it strongly positive} if $I_+=I$. 
		On the other hand, when $I_0\neq\emptyset$, we say that $\mathcal{U}$ is {\it weakly positive}. 
	\end{definition}
	
	\begin{remark}\label{rmk:positiveness} 
		Nonnegative solutions are always weakly positive, provided that the maximum principle holds.
		To this end, we either prove or assume that each component is superharmonic. 
		In the singular case, it is even possible to show that solutions are strongly positive $($see, for instance, \cite[Corollary~47]{arXiv:2002.12491}$)$, which we are not assuming here a priori.
	\end{remark}
	
	\begin{definition}
		Let $\Omega=B_R^*$ be the punctured ball $($space$)$ with $R<\infty$ $($$R=\infty$$)$, and $\mathcal{U}$ be a non-singular $($singular$)$ solution to \eqref{eq:subcriticalsystem}. We say that $\mathcal{U}$ is a {\it weak solution} to \eqref{eq:subcriticalsystem}, if it belongs to $\mathcal{D}^{2,2}(\Omega,\mathbb{R}^p)$ and solves \eqref{eq:subcriticalsystem} in the weak sense, {\it i.e.}, for all nonnegative $\Phi\in C^{\infty}_c(\Omega,\mathbb{R}^p)$, one has
		\begin{equation*}
		\displaystyle\int_{\mathbb{R}^n}\langle\Delta \mathcal{U},\Delta \Phi\rangle \ud x=\displaystyle\int_{\mathbb{R}^n}\langle|\mathcal{U}|^{s-1}\mathcal{U},\Phi\rangle \ud x.
		\end{equation*}
		Here $\mathcal{D}^{2,2}(\Omega,\mathbb{R}^p)$ is the classical Beppo--Levi space, completion of the space of compactly supported smooth $p$-maps, denoted by $C^{\infty}_{c}(\Omega,\mathbb{R}^p)$ under the Dirichlet norm $\|\mathcal{U}\|_{\mathcal{D}^{2,2}(\Omega,\mathbb{R}^p)}^2=\sum_{i=1}^p\|\Delta u_i\|_{L^{2}(\Omega)}^2$.
	\end{definition}
	
	\begin{remark}\label{rmk:regularity}
		Assuming that the component solutions are smooth away from the origin does not impose any restrictions, since we are dealing with the subcritical regime. Indeed, by classical elliptic regularity theory and bootstrap methods, one can prove that any weak non-singular $($singular$)$ solution to \eqref{eq:subcriticalsystem} is also a classical non-singular $($singular$)$ solution.
		Besides, observe that for the case of unbounded domains $(R=\infty)$, some decay estimates are necessary to prove that solutions have finite $L^p$-norm.
	\end{remark}
	
	\subsection{Integral representation formulas}
	Now we use a Green identity to transform the fourth order differential system  \eqref{eq:subcriticalsystem} into an integral system.
	In this way, we can avoid using the classical form of the maximum principle, and a sliding method is available \cite{arxiv:1901.01678,MR3558255, MR2055032}, which will be used to classify solutions. 	
	Besides, in this setting is also possible to prove regularity through a barrier construction.
	
	For $n\geqslant3$, the following expression for the Green function of the Laplacian in the unit ball is well-known.
	\begin{equation*}
		G_{1}(x, y)=\frac{1}{(n-2) \omega_{n-1}}\left(|x-y|^{2-n}-\left|\frac{x}{|x|}- x|y|\right|^{2-n}\right),
	\end{equation*}
	where $\omega_{n-1}$ is the surface area of the Euclidean unit sphere. 
	In addition, for any $u \in C^{2}\left(B_{1}\right) \cap C\left(\bar{B}_{1}\right)$, the next decomposition holds,
	\begin{equation*}
		u(x)=-\int_{B_{1}} G_{1}(x, y)\Delta u(y) \mathrm{d} y+\int_{\partial B_{1}} H_{1}(x, y) u(y) \mathrm{d} \sigma_{y},
	\end{equation*}
	where
	\begin{equation*}
		H_{1}(x, y)=-{\partial_{v_{y}}} G_{1}(x, y)=\frac{1-|x|^{2}}{\omega_{n-1}|x-y|^{n}} \quad \mbox{for} \quad x \in B_{1} \quad \mbox{and} \quad y \in \partial B_{1},
	\end{equation*}
	with ${v_{y}}$ the outward normal vector at $y$.
	
	Similarly, in the fourth order case with $n\geqslant5$, for any $u\in C^{4}\left(B_{1}\right)\cap C^{2}\left(\bar{B}_{1}\right)$, it follows
	\begin{equation*}
		u(x)=\int_{B_{1}} G_{2}(x, y)\Delta^{2} u(y) \ud y+\int_{\partial B_{1}} H_{1}(x, y)u(y) \ud \sigma_{y}-\int_{\partial B_{1}} H_{2}(x, y)\Delta u(y) \ud \sigma_{y},
	\end{equation*}
	where
	\begin{equation*}
		G_{2}(x, y)=\int_{B_{1} \times B_{1}} G_{1}\left(x, y_{1}\right) G_{1}\left(y_{1}, y\right) \ud y_{1}
	\end{equation*}
	and
	\begin{equation*}
		H_{2}(x, y)=\int_{B_{1} \times B_{1}} G_{1}\left(x, y_{1}\right) H_{1}\left(y_1, y\right) \mathrm{d} y_{1}.
	\end{equation*}
	By a direct computation, we have
	\begin{equation}\label{greenfunction}
		G_{2}(x, y)=C(n, 2)|x-y|^{4-n}-A(x, y),
	\end{equation}
	where $C(n, 2)=\frac{\Gamma(n-4)}{2^{4} \pi^{n / 2} \Gamma(2)}$, $A:B_1\times B_1\rightarrow\mathbb{R}$ is a smooth map and $H_{i}(x, y) \geqslant 0$ for $i=1,2$.
	
	In the next lemma, we provide 
	basic integrability for singular solutions to \eqref{eq:subcriticalsystem} when $R<\infty$ and $s>1$, which yields a weaker formulation for \eqref{eq:subcriticalsystem}. Here, the proof is similar to the one in \cite[Lemma~3.1]{MR4123335}. Nevertheless, we included the proof for the sake of completeness.
	
	\begin{lemma}\label{lm:integrability}
		Let $R=1$, $s\in(1,\infty)$, and $\mathcal{U}$ be a nonnegative superharmonic singular solution to \eqref{eq:subcriticalsystem}.Then, $\mathcal{U} \in L^{s}\left(B_{1},\mathbb{R}^{p}\right)$. 
		In particular, if $s\in(2_{**},\infty)$, then $\mathcal{U}$ is a
		distribution solution to \eqref{eq:subcriticalsystem}, that is, for all nonnegative $\Phi\in C^{\infty}_c(B_1,\mathbb{R}^{p})$, one has
		\begin{equation}\label{distributional}
		\int_{B_{1}} \langle\mathcal{U},\Delta^2 {\Phi}\rangle \ud x=\int_{B_{1}}|\mathcal{U}|^{s-1}\langle\mathcal{U}, {\Phi}\rangle  \ud x \quad.
		\end{equation}
	\end{lemma}
	
	\begin{proof}
		For any $0<\varepsilon \ll 1$, let us consider $\eta_{\varepsilon}\in C^{\infty}\left(\mathbb{R}^{n}\right)$ with $0\leqslant\eta_{\varepsilon}\leqslant1$ satisfying
		\begin{equation}\label{cutoff}
		\eta_{\varepsilon}(x)=
		\begin{cases}
		0, & \mbox{if} \ |x| \leqslant \varepsilon\\
		1, & \mbox{if} \ |x| \geqslant 2 \varepsilon,
		\end{cases}
		\end{equation}
		and $|D^{(j)} \eta_{\varepsilon}(x)| \leqslant C \varepsilon^{-j}$ for $j\geqslant1$.
		Define $\xi_{\varepsilon}=\left(\eta_{\varepsilon}\right)^{q}$, where  $q=s\gamma(s)$. Multiplying \eqref{eq:subcriticalsystem} by $\xi_{\varepsilon}$, and integrating by parts in $B_{r}$ with $r\in(1/2,1)$, we obtain
		\begin{equation*}
		\int_{B_{r}} |\mathcal{U}|^{s-1}u_i \xi_{\varepsilon}\ud x =\int_{\partial B_{r}} {\partial_\nu} \Delta u_i\ud\sigma_r+\int_{B_{r}} u_i \Delta^{2} \xi_{\varepsilon}\ud x \quad \mbox{for all} \quad i\in I.
		\end{equation*}
		On the other hand, there exists $C>0$ such that
		\begin{equation*}
		\left|\Delta^{2} \xi_{\varepsilon}\right| \leqslant C \varepsilon^{-4}\left(\eta_{\varepsilon}\right)^{q-4} \chi_{\{\varepsilon \leqslant|x| \leqslant 2 \varepsilon\}}=C \varepsilon^{-4}\left(\xi_{\varepsilon}\right)^{1/s} \chi_{\{\varepsilon \leqslant|x| \leqslant 2 \varepsilon\}},
		\end{equation*}
		which, by H\"{o}lder's inequality, gives us
		\begin{align*}
		\left|\int_{B_{r}} u_i \Delta^{2} \xi_{\varepsilon}\ud x\right|
		&\leqslant C \varepsilon^{-4} \int_{\{\varepsilon \leqslant|x| \leqslant 2 \varepsilon\}} u_i\xi_{\varepsilon}^{1/s}\ud x \\
		& \leqslant C \varepsilon^{-4} \varepsilon^{n(1-1/s)}\left(\int_{\{\varepsilon \leqslant|x| \leqslant 2 \varepsilon\}} |\mathcal{U}|^{s-1}u_i \xi_{\varepsilon}\ud x\right)^{1/s} \\
		& \leqslant C\left(\int_{\{\varepsilon \leqslant|x| \leqslant 2 \varepsilon\}} |\mathcal{U}|^{s-1}u_i \xi_{\varepsilon}\ud x\right)^{1/s}.
		\end{align*}
		Thus, it follows
		\begin{equation*}
		\int_{B_{r}} |\mathcal{U}|^{s-1}u_i \xi_{\varepsilon}\ud x \leqslant \int_{\partial B_{r}} {\partial_\nu} \Delta u_i\ud\sigma_r+C\left(\int_{\{\varepsilon \leqslant|x| \leqslant 2 \varepsilon\}} |\mathcal{U}|^{s-1}u_i \xi_{\varepsilon}\ud x\right)^{1/s},
		\end{equation*}
		which provides a constant $C>0$ (independent of $\varepsilon$) such that
		\begin{equation*}
		\int_{B_{r}}|\mathcal{U}|^{s-1}u_i \xi_{\varepsilon}\ud x \leqslant C.
		\end{equation*}
		Now letting $\varepsilon \rightarrow 0$, since $u_i\leqslant|\mathcal{U}|$, we conclude that $u_i \in L^{s}\left(B_{r}\right)$ for all $i\in I$ and the integrability follows.
		
		We are left to show that $\mathcal{U}$ is a distribution solution to \eqref{eq:subcriticalsystem}, that is, we need to establish \eqref{distributional}. For any nonnegative $\Phi\in C^{\infty}_c(B_1,\mathbb{R}^{p})$, we multiply \eqref{eq:subcriticalsystem} by $\widetilde{\Phi}=\eta_{\varepsilon} \Phi$, where $\eta_{\varepsilon}$ is given by \eqref{cutoff}. Then, using that $|\mathcal{U}|\in L^{s}\left(B_{r}\right)$ and integrating by parts twice, we get
		\begin{equation}\label{yang1}
		\int_{B_{1}} \langle \mathcal{U},\Delta^{2}\left(\eta_{\varepsilon}\Phi\right)\rangle\ud x=\int_{B_{1}} \langle|\mathcal{U}|^{s-1}\mathcal{U},\eta_{\varepsilon}\Phi\rangle\ud x.
		\end{equation}
		By a direct computation, we find that $\Delta^{2}\left(\eta_{\varepsilon} \phi_i\right)=\eta_{\varepsilon} \Delta^{2}\phi_i+\varsigma^{\varepsilon}_i$, where
		\begin{equation*}
		\varsigma^{\varepsilon}_i=4 \langle\nabla \eta_{\varepsilon},\nabla \Delta\phi_i\rangle+2\Delta \eta_{\varepsilon} \Delta\phi_i+4 \Delta\eta_{\varepsilon}\Delta\phi_i+4\langle\nabla \Delta \eta_{\varepsilon},\nabla \phi_i\rangle+\phi_i\Delta^{2} \eta_{\varepsilon}.
		\end{equation*}
		Furthermore, using H\"{o}lder's inequality again, we find
		\begin{align*}
		\left|\int_{B_{1}} \langle \mathcal{U},\Psi_{\varepsilon}\rangle\ud x\right| \leqslant C\left(\int_{\{\varepsilon\leqslant|x| \leqslant 2 \varepsilon\}} |\mathcal{U}|^{s-1}u_i\ud x\right)^{1/s} \leqslant C\left(\int_{\{\varepsilon\leqslant|x| \leqslant 2 \varepsilon\}} |\mathcal{U}|^{s}\ud x\right)^{1/s} \rightarrow 0 \quad \mbox{as} \quad \varepsilon \rightarrow 0,
		\end{align*}
		where $\Psi_{\varepsilon}=(\varsigma^{\varepsilon}_1,\dots,\varsigma^{\varepsilon}_p)\in C_c^{\infty}(B_1,\mathbb{R}^p)$.
		Finally, letting $\varepsilon \rightarrow 0$ in \eqref{yang1}, and applying the dominated convergence theorem the proof follows.
	\end{proof}
		
	In the following lemma, we employ some ideas due to L. Caffarelli et al. \cite{MR982351} and L. Sun and J. Xiong \cite{MR3558255}. 
	
	\begin{lemma}\label{integrability}
		Let $R=1$, $s\in(2_{**},\infty)$, and  $\mathcal{U}\in C^{4}(\bar{B}_1^*,\mathbb{R}^p)\cap L^{1}(B_{1},\mathbb{R}^p)$ be a nonnegative singular
		solution to 
		\begin{equation}\label{subcriticalsystem}
		\Delta^2u_i=f^s_i(\mathcal{U}) \quad {\rm in} \quad B^*_{1},
		\end{equation}
		where $f^s_i(\mathcal{U}):=|\mathcal{U}|^{s-1}u_i$. Then, $|x|^{-q} u_i^{s}\in L^{1}\left(B_{1}\right)$ for any $q<n-{4s}/{(s-1)}$. Moreover,
		\begin{equation*}
		u_i(x)=\int_{B_{1}} G_{2}(x, y)\Delta^{2} u_i(y) \ud y+\int_{\partial B_{1}} H_{1}(x, y)u_i(y) \ud \sigma_{y}-\int_{\partial B_{1}} H_{2}(x, y)\Delta u_i(y) \ud \sigma_{y}.
		\end{equation*}
	\end{lemma}
	
	\begin{proof}
		First, by Lemma~\ref{lm:integrability}, we have that $|\mathcal{U}|\in L^{s}\left(B_{1},\mathbb{R}^{p}\right)$ and $\mathcal{U}$ is a distribuitional solution in $B_1$.
		Let $\eta\in C^{\infty}(\mathbb{R})$ such that $\eta(t)=0$ for $t\leqslant 1$, $\eta(t)=1$ for $t \geqslant 2$ and $0 \leqslant\eta\leqslant 1$ for $1\leqslant t\leqslant 2$. For small $\varepsilon>0$, plugging $\eta\left(\varepsilon^{-1}|x|\right)|x|^{4-q}$ into \eqref{subcriticalsystem}, and using integration by parts, it holds
		\begin{equation}\label{integrality1}
		\int_{B_{1}} f^s_i(\mathcal{U})\eta\left(\varepsilon^{-1}|x|\right)|x|^{4-q}\ud x=\int_{B_{1}} u_i\Delta^{2}\left(\eta\left(\varepsilon^{-1}|x|\right)|x|^{4-q}\right)\ud x+\int_{\partial B_{1}} G(u_i) \ud \sigma_y,
		\end{equation}
		where $G(u_i)$ involves $u_i$ and its derivatives up to third order. Taking $q=q_{0}=0$ in \eqref{integrality1}, we have that  $\int_{B_{1}} u_i^{s}|x|^{4{\left(-q_{1}-4/s\right)}}\ud x<\infty$ as $\varepsilon \rightarrow 0$, since by hypothesis $u_i \in L^{1}\left(B_{1}\right)$ for all $i\in I$. Moreover, by H\"{o}lder inequality, if $0<q_{1}<[n(s-1)-4]/s$, we have
		\begin{equation*}
		\int_{B_{1}} u_i|x|^{-q_{1}}\ud x
		=\int_{B_{1}} u_i|x|^{4/s}|x|^{\left(-q_{1}-4/s\right)} \ud x
		\leqslant\left(\int_{B_{1}} u_i^{s}|x|^{4}\ud x\right)^{1/s}\left(\int_{B_{1}}|x|^{-\frac{4+s q_{1}}{s-1}}\ud x\right)^{(s-1)/s}<\infty.
		\end{equation*}
		Next, using that $q=q_{1}$ in \eqref{integrality1} and taking the limit $\varepsilon\rightarrow0$, we obtain $\int_{B_{1}} u^{s}|x|^{4-q_{1}}\ud x<\infty$. By the same argument, we find  $\int_{B_{1}} u|x|^{-q_{2}}\ud x<\infty$, if $q_{2}<[n(s-1)-4](s^{-1}+s^{-2})$.
		Iterating this procedure, we get
		\begin{equation*}
		\int_{B_{1}} u_i|x|^{-q_{k}}\ud x<\infty \quad \mbox{and} \quad \int_{B_{1}} u_i^{s}|x|^{4-q_{k}}\ud x<\infty,
		\end{equation*}
		if $0<q_{k}:=[n(s-1)-4]\left(\sum_{i=1}^{k}s^{-i}\right)$ for $k\in\mathbb{N}$, which proves the first statement.
		
		Next, let us consider
		\begin{equation*}
		\widehat{u}_i(x)=\int_{B_{1}} G_{2}(x, y)\Delta^{2} u_i(y) \ud y+\int_{\partial B_{1}} H_{1}(x, y)u_i(y) \ud \sigma_{y}-\int_{\partial B_{1}} H_{2}(x, y)\Delta u_i(y) \ud \sigma_{y} \quad \mbox{for} \quad i\in I.
		\end{equation*}
		and define $h_i=u_i-\widehat{u}_i$. Thus, $\Delta^{2} h_i=0$  in $B^*_{1}$ since $u_i^{s} \in L^{1}\left(B_{1}\right)$ and $\widehat{u}_i\in L_{{\rm weak}}^{2^{**}/2}\left(B_{1}\right)\cap L^{1}\left(B_{1}\right)$. Furthermore, since the Riesz potential $|x|^{4-n}$ is of weak type $\left(1,2^{**}/2\right)$ (see \cite[Chapter~9]{MR1814364}), for every $0<\varepsilon\ll1$, we can choose $\rho>0$
		such that $\int_{B_{2\rho}} f_i(\mathcal{U}(x))\ud x<\varepsilon$, which implies that for all $M\gg1$,
		\begin{equation*}
		\mathcal{L}^n\left(\{x \in B_{\rho} : |\widehat{u}_i(x)|>M\}\right) \leqslant\mathcal{L}^n\left(\left\{x \in B_{\rho} : \int_{B_{2\rho}} G_{2}(x, y) f_i(\mathcal{U}(y))\ud y>\mu / 2\right\}\right)\leqslant\varepsilon M^{-2^{**}/2}.
		\end{equation*}
		Hence, $h_i\in L_{{\rm weak}}^{2^{**}/2}\left(B_{1}\right)\cap L^{1}\left(B_{1}\right)$ for all $i\in I$ and for every $0<\varepsilon\ll1$, there exists $\rho>0$ such that for all
		sufficiently large $\mu\gg1$, it holds
		\begin{align*}
		\mathcal{L}^n\left(\{x \in B_{\rho}:|h_i(x)|>M\}\right)&\leqslant\mathcal{L}^n\left(\{x \in B_{\rho}:|u_i(x)|>M/2|\}\right)+\mathcal{L}^n\left(\{x \in B_{\rho}:|\widehat{u}_i(x)|>M/2\}\right)&\\
		&\leqslant\varepsilon M
		^{-2^{**}/2}.&
		\end{align*}
		Using the B\^ocher theorem for biharmonic functions from \cite{MR1820695}, we obtain that $\Delta^2 h_i=0$ in $B_{1}$. Finally, since $h_i=\Delta h_i=0$ on $\partial B_{1}$, by the maximum principle, we have that $h_i \equiv 0$ and thus $u_i=\widehat{u}_i$, which finishes the proof of the second part.
	\end{proof}
	
	\begin{remark}
		If $\mathcal{U} \in C^{2}\left(B^*_{1},\mathbb{R}^p\right)$ is a nonnegative superharmonic $p$-map, then $\mathcal{U}\in L^{1}\left(B_{1/2},\mathbb{R}^p\right)$. Indeed, for $0<r<1$ and $i\in I$, we have
		\begin{equation*}
		\left(r^{n-1} \overline{u}_i^{(1)}\right)^{(1)}\leqslant0,
		\end{equation*}
		where $\overline{u}_i(r)=\avint_{\partial B_{1}} u_i(r \theta)\ud\theta$, which by a direct integration implies that $\overline{u}_i(r)\leqslant C\left(r^{2-n}+1\right)$.
	\end{remark}
	
	In the next result, we use the Green identity to convert \eqref{eq:subcriticalsystem} into an integral system.
	
	\begin{proposition}\label{lm:integralrepresentation}
		Let $R<\infty$ and $\mathcal{U}$ be a nonnegative superharmonic solution to \eqref{eq:subcriticalsystem}. Then, there exists $r_0>0$ such that 
		\begin{equation}\label{integralsystem}
		u_i(x)=\int_{B_{r_0}}|x-y|^{4-n}f^s_i(\mathcal{U}) \ud y+\psi_i(x),
		\end{equation}
		where $\psi_i>0$ satisfies $\Delta^{2} \psi_i=0$ in $B_{r_0}$ for all $i\in I$. Moreover, one can find a constant $C(\widetilde{r})>0$ such that $\|\nabla \ln \psi_i\|_{C^{0}\left(B_{\widetilde{r}}\right)} \leqslant C\left(\widetilde{r}\right)$ for all $i\in I$ and $0<\widetilde{r}<r_0$. 
	\end{proposition}
	
	\begin{proof}
		Using that $-\Delta u_i>0$ in $B^*_{1}$ and $u_i>0$ in $\bar{B}_{1}$, it follows from the maximum principle that $c_{1i}:=\inf _{B_{1}} u_i=\min _{\partial B_{1}} u_i>0 $. In addition, by Lemma~\ref{integrability}, we get that $f^s_i(\mathcal{U}) \in L^{1}\left(B_{1}\right)$, which implies that there exists $r_0<1/4$ satisfying
		\begin{equation*}
		\int_{B_{\tau}}|A(x, y)|f^s_i(\mathcal{U})\ud y\leqslant\frac{c_{1i}}{2} \quad \mbox{for} \quad x\in B_{r_0},
		\end{equation*}
		where $A(x, y)$ is given by \eqref{greenfunction}. Hence, for $x \in B_{r_0}$, we get
		\begin{align*} 
		\psi_i(x)=-& \int_{B_{r}} A(x, y)f^s_i(\mathcal{U})\ud y+\int_{B_{1}\setminus B_{r}} G_{2}(x, y) f^s_i(\mathcal{U}) \mathrm{d} y&\\
		&+\int_{\partial B_{1}} H_{1}(x, y)u_i(y)\ud\sigma_{y}-\int_{\partial B_{1}} H_{2}(x, y)\Delta u_i(y)\ud\sigma_{y}&\\ 
		&\geqslant-\frac{c_{1i}}{2}+\int_{\partial B_{1}} H_{1}(x, y) u_i(y) \ud\sigma_{y}&\\
		&\geqslant-\frac{c_{1i}}{2}+\inf _{B_{1}} u_i=\frac{c_{1i}}{2}.
		\end{align*}
		By hypothesis, $\psi_i$ is biharmonic and $\psi_i\in C^{\infty}(B_{r_0})$ for all $i\in I$, which provides that $|\nabla \psi_i|\leqslant C_i(\widetilde{r})$ in $B_{\widetilde{r}}$ for all $\widetilde{r}<\rho$ and $i\in I$, where $C_i\left(\widetilde{r}\right)>0$ depends only on $n,\rho,\widetilde{r}$ and in the $L^{1}$ norm of $f^s_i(\mathcal{U})$. Consequently,
		\begin{equation*}
		\|\nabla \ln \psi_i\|_{C^{0}\left(B_{\widetilde{r}}\right)} \leqslant 2\frac{C\left(\widetilde{r}\right)}{c_{1i}} \quad \mbox{for} \quad i\in I,
		\end{equation*}
		which finishes the proof.
	\end{proof}
	
	\subsection{Kelvin transform}\label{sec:kelvintransform}
	Later we will employ the moving spheres technique, which is based on the {\it fourth order Kelvin transform} for a $p$-map. 
	To define the Kelvin transform, we need to establish the concept of {\it inversion about a sphere} $\partial B_{\mu}(x_0)$, which is a map $\mathcal{I}_{x_0,\mu}:\mathbb{R}^n\rightarrow\mathbb{R}^n\setminus\{x_0\}$ given by $\mathcal{I}_{x_0,\mu}(x)=x_0+K_{x_0,\mu}(x)^2(x-x_0)$, where $K_{x_0,\mu}(x)=\mu/|x-x_0|$ (for more details, see \cite[Section~2.7]{arXiv:2002.12491}). 
	For $\Omega=B_R$ (or $\Omega=\mathbb{R}^n$) or $\Omega=B_R^*$ (or $\Omega=\mathbb{R}^n\setminus\{0\}$), we set $\Omega_{x_0,\mu}:=\mathcal{I}_{x_0,\mu}(\Omega)$ as the image of $\Omega$ under the inversion map. 
	Notice that $\Omega_{x_0,\mu}=\Omega\setminus B_{\mu}(x_0)$.
	The following definition is a generalization of the Kelvin transform for fourth order operators applied to $p$-maps.
	
	\begin{definition}
		For any $\mathcal{U}\in C^4(\Omega,\mathbb{R}^p)$, let us consider the fourth order Kelvin transform about the sphere with center at $x_0\in\mathbb{R}^n$ and radius $\mu>0$ defined by 
		\begin{equation*}
		\mathcal{U}_{x_0,\mu}(x)=K_{x_0,\mu}(x)^{n-4}\mathcal{U}\left(\mathcal{I}_{x_0,\mu}(x)\right).
		\end{equation*}
	\end{definition}
	
	\begin{proposition}\label{prop:conformalinvariance}
		If $\mathcal{U}$ be a solution to \eqref{eq:subcriticalsystem}, then  $\mathcal{U}_{x_0,\mu}$ is a solution to 
		\begin{equation*}
		\Delta^{2}(u_{i})_{x_0,\mu}=K_{x_0,\mu}(x)^{(n-4)s-(n+4)}|\mathcal{U}_{x_0,\mu}|^{s-1}(u_{i})_{x_0,\mu} \quad {\rm in} \quad \Omega\setminus\{x_0\}\quad {\rm for} \quad i\in I,
		\end{equation*}
		where $\mathcal{U}_{x_0,\mu}=((u_1)_{x_0,\mu},\dots,(u_p)_{x_0,\mu})$.
	\end{proposition}
	
	\begin{proof}
		It directly follows by using \cite[Lemma~3.6]{MR1769247} on each component solution $u_i$.
	\end{proof}
	
	\subsection{Cylindrical transformation}\label{subsec:cylindrcialtransform} 
	Here the goal is to convert the subcritical System \eqref{eq:subcriticalsystem} into a PDE on a cylinder. 
	Then, the local behavior near the origin reduces the asymptotic global behavior for tempered solutions to a fourth order PDE defined on a cylinder. 
	We have included the computations for both second and fourth cases in Appendix \ref{app:computations}.
	
	\subsubsection{Autonomous case}
	Initially, let us introduce the so-called {\it cylindrical transformation} (see also \cite{arXiv:2002.12491}). First, we consider the cylinder $\mathcal{C}_R:=(0,R)\times\mathbb{S}^{n-1}$ and $\Delta^2_{\rm sph}$ the bi-Laplacian written in spherical (polar) coordinates,
	\begin{align*}
	\Delta^2_{\rm sph}&=\partial_r^{(4)}+ \frac{2(n-1)}{r}\partial_r^{(3)}+\frac{(n-1)(n-3)}{r^2}\partial_r^{(2)}-\frac{(n-1)(n-3)}{r^3}\partial_r&\\\nonumber
	&+\frac{1}{r^4}\Delta_{\sigma}^2+\frac{2}{r^2}\partial^{(2)}_r\Delta_{\sigma}+\frac{2(n-3)}{r^3}\partial_r\Delta_{\sigma}-\frac{2(n-4)}{r^4}\Delta_{\sigma},&
	\end{align*}
	where $\Delta_{\sigma}$ denotes the Laplace--Beltrami operator in $\mathbb{S}^{n-1}$.
	Then, in polar coordinates, we can rewrite \eqref{eq:subcriticalsystem} as the nonlinear system,
	\begin{equation*}
	\Delta^2_{\rm sph}u_i=|\mathcal{U}|^{s-1}u_i \quad {\rm in} \quad \mathcal{C}_R.
	\end{equation*}
	In addition, we apply the Emden--Fowler change of variables (or logarithm coordinates) given by
	\begin{equation*}
		\mathcal{V}(t,\theta)=r^{\gamma(s)}\mathcal{U}(r,\sigma), \quad \mbox{where} \quad r=|x|, \quad t=-\ln r, \quad \sigma=\theta=x/|x|, \quad \mbox{and} \quad \gamma(s)=4/{(s-1)},
	\end{equation*}  
	which sends the problem to a cylinder $\mathcal{C}_T:=(-\ln R,\infty)\times \mathbb{S}^{n-1}$.  
	Using this coordinate system and performing a lengthy computation, we arrive at the following fourth order nonlinear PDE system on the cylinder,
	\begin{equation}\label{eq:subcriticalcylindricalsystem}
	\Delta^2_{\rm cyl}v_i=|\mathcal{V}|^{s-1}v_i \quad {\rm on} \quad {\mathcal{C}}_T.
	\end{equation}
	Here $\Delta^2_{\rm cyl}$ is the bi-Laplacian written in cylindrical coordinates given by
	\begin{align*}
	\Delta^2_{\rm cyl}&=\partial_t^{(4)}+K_{3}(n,s)\partial_t^{(3)}+K_{2}(n,s)\partial_t^{(2)}+K_{1}(n,s)\partial_t+K_{0}(n,s)&\\\nonumber
	&+\Delta_{\theta}^{2}+2\partial^{(2)}_t\Delta_{\theta}+J_{1}(n,s)\partial_t\Delta_{\theta}+J_{0}(n,s)\Delta_{\theta},&
	\end{align*}
	where $K_{j}(n,s),J_{j}(n,s)$ are constants depending only on $n,s$ for $j=0,1,2,3,4$, which are defined by
	\begin{align}\label{autonomouscoefficients}
	&\nonumber K_{0}(n,s)={8}(s-1)^{-4}\left[(n-2)(n-4)(s-1)^{3}+2\left(n^{2}-10 n+20\right)(s-1)^{2}
	-16(n-4)(s-1)+32\right],&\\\nonumber &K_{1}(n,s)=-{2}{(s-1)^{-3}}\left[(n-2)(n-4)(s-1)^{3}+4\left(n^{2}-10 n+20\right)(s-1)^{2}-48(n-4)(s-1)+128\right],&\\\nonumber
	&K_{2}(n,s)={1}{(s-1)^{-2}}\left[\left(n^{2}-10 n+20\right)(s-1)^{2}-24(n-4)(s-1)+96\right],&\\
	&K_{3}(n,s)={2}{(s-1)}^{-1}[(n-4)(s-1)-8],&\\\nonumber
	&J_{0}(n,s)=-{2}{(s-1)^{-2}}\left[(n-4)(s-1)^{2}+4(n-4)(s-1)-16\right],&\\\nonumber
	&J_{1}(n,s)={2}{(s-1)^{-1}}\left[(n-4)(s-1)+16\right].&
	\end{align}
	
	\begin{remark}\label{rmk:signoncoefficients}
		By direct computations, notice that when $s=2^{**}-1$, it follows 
		\begin{equation*}
		K^*_1\equiv K^*_3\equiv J^*_1\equiv0, \quad 
		K^*_0=\frac{n^2(n-4)^2}{16}, \quad K^*_2=-\frac{n^2-4n+8}{2}, \quad {\rm and} \quad J^*_0=-\frac{n(n-4)}{2},
		\end{equation*}
		whereas, when $s=2_{**}$, it holds
		\begin{align*}
		&K_{0,*}\equiv0, \ K_{1,*}=2(n-4)(n+2), \ K_{2,*}=n^2-10n+20, \ K_{3,*}=J_{1,*}=2(n-4), \ {\rm and} \ J_{0,*}=-2(n-4).&
		\end{align*}
		Moreover, by \cite[Lemma~3.3]{MR4123335}, for $s\in(2_{**},2^{**}-1)$, one has
		\begin{equation*}
		K_{0}(n,s)>0, \quad 
		K_{1}(n,s)>0, \quad {\rm and} \quad
		K_{3}(n,s)<0, \quad 
		J_0(n,s)<0,
		\end{equation*}
		and, $K_{2}(n,s)$ has no defined sign. This change of sign behavior explains why one needs to use distinct methods when the power parameter $s\in(1,\infty)$ changes.
	\end{remark}
	
	\begin{definition}
		Let us consider the autonomous cylindrical transformation as follows
		\begin{equation}\label{cyltransform}
		\mathfrak{F}:C_c^{\infty}(B_R^*,\mathbb{R}^p)\rightarrow C_c^{\infty}(\mathcal{C}_{T},\mathbb{R}^p) \quad \mbox{given by} \quad
		\mathfrak{F}(\mathcal{U})=r^{\gamma(s)}\mathcal{U}(r,\sigma).
		\end{equation}
	\end{definition}
	
	Notice that in the limit case $R=\infty$, that is, $\mathcal{C}_\infty=\mathbb{R}\times\mathbb{S}^{n-1}$, blow-up solutions are rotationally invariant. 
	Therefore, \eqref{eq:subcriticalcylindricalsystem} becomes the following ODE system with constant coefficients,
	\begin{equation}\label{eq:subcriticalODEsystem}
	v_i^{(4)}+K_{3}(n,s)v_i^{(3)}+K_{2}(n,s)v_i^{(2)}+K_{1}(n,s)v_i^{(1)}+K_{0}(n,s)v_i=|\mathcal{V}|^{s-1}v_i \quad {\rm in} \quad \mathbb{R},
	\end{equation}
	which can be furnished with the suitable initial data at $t=0$ to become a well-posed Cauchy problem.
	
	\begin{remark}
		The transformation $\mathfrak{F}$ is a continuous bijection with respect to the Sobolev norms $\|\cdot\|_{\mathcal{D}^{2,2}(B_R^*,\mathbb{R}^p)}$ and $\|\cdot\|_{H^{2}(\mathcal{C}_T,\mathbb{R}^p)}$, respectively. Furthermore, this transformation sends singular solutions to \eqref{eq:subcriticalsystem} into solutions to \eqref{eq:subcriticalcylindricalsystem} and by density, we get $\mathfrak{F}:\mathcal{D}^{2,2}(B_R^*,\mathbb{R}^p)\rightarrow H^{2}(\mathcal{C}_T,\mathbb{R}^p)$.
		Moreover, notice that with minor changes, all these computations could have been carried out if the opposite sign convention on the Emden--Fowler change of variables, that is, taking $t=\ln |x|$ also leads to the same asymptotics.
	\end{remark}
	
	\begin{remark}
		In more geometrical terms, this change of variables corresponds to the conformal diffeomorphism between the cylinder and the punctured space, $\varphi:(\mathcal{C}_{\infty},g_{{\rm cyl}})\rightarrow(\mathbb{R}^n\setminus\{0\},\delta_0)$ defined by $\varphi(t,\sigma)=e^{-t}\sigma$. Here $g_{{\rm cyl}}=\ud t^2+\ud\sigma^2$ stands for the cylindrical metric and ${\ud\theta}=e^{-2t}(\ud t^2+\ud\sigma^2)$ for its volume element obtained via the pullback $\varphi^{*}\delta_0$, where $\delta_0$ is the standard flat metric.
	\end{remark}
	
	\subsubsection{Nonautonomous case}
	In the lower critical case, because of the vanishing of the coefficient $K_0(n,2_{**})$, we already know that the situation changes dramatically. In this fashion, we define the so-called {\it nonautonomous Emden--Fowler change of variables} \cite{MR875297,MR4085120,MR633393,MR727703} given by
	\begin{equation}\label{newcylindrical}
	\mathcal{W}(t,\theta)=r^{4-n}(-\ln r)^{\frac{4-n}{4}}\mathcal{U}(r,\sigma), \quad {\rm where} \quad t=-\ln r \quad {\rm and} \quad \sigma=\theta=x|x|^{-1}.
	\end{equation}
	This again sends in the punctured ball \eqref{eq:subcriticalsystem} with $s=2_{**}$ to the cylinder $\mathcal{C}_T$.  
	Using this coordinate system and performing a lengthy computation (see Appendix \ref{app:computations}), we arrive at the following fourth order nonlinear PDE system on the cylinder,
	\begin{equation}\label{eq:subcriticalcylindricalsystemlower}
	\widetilde{\Delta}^2_{\rm cyl}w_i=t^{-1}|\mathcal{W}|^{2_{**}-1}w_i \quad {\rm on} \quad {\mathcal{C}}_T \quad \mbox{for} \quad i\in I.
	\end{equation}
	Here $\widetilde{\Delta}^2_{\rm cyl}$ is the bi-Laplacian written in nonautonomous cylindrical coordinates given by
	\begin{align*}
	\widetilde{\Delta}^2_{\rm cyl}&=\partial_t^{(4)}+\widetilde{K}_{3}(n,t)\partial_t^{(3)}+\widetilde{K}_{2}(n,t)\partial_t^{(2)}+\widetilde{K}_{1}(n,t)\partial_t+\widetilde{K}_{0}(n,t)&\\\nonumber
	&+\Delta_{\theta}^{2}+2\partial^{(2)}_t\Delta_{\theta}+\widetilde{J}_{1}(n,t)\partial_t\Delta_{\theta}+\widetilde{J}_{0}(n,t)\Delta_{\theta},&
	\end{align*}
	where $\widetilde{K}_{j}(n,t),\widetilde{J}_{j}(n,t):\mathbb{R}^*\rightarrow\mathbb{R}$ are real functions of $t$ for $j=0,1,2,3$, which are defined by
	\begin{align}\label{nonautonomouscoeficcients}
	&\nonumber\widetilde{K}_{0}(n,t)=\frac{(n-4)n(n+4)(n+8)}{256t^4}-\frac{(n-4)^2n(n+4)}{32t^3}+\frac{(n-4)n(n^2-10n+20)}{16t^2}+\frac{(n-4)(n-2)(n+4)}{t},&\\\nonumber &\widetilde{K}_{1}(n,t)=\frac{(n-4)n(n+4)}{16t^3}+\frac{3n(n-4)}{8t^2}+\frac{(n-4)(n^2-10n+20)}{2t}-2(n-4)(n-2),&\\\nonumber
	&\widetilde{K}_{2}(n,t)=\frac{3n(n-4)}{8t^2}-\frac{3(n-4)^2}{2t}+n^2-10n+20,&\\
	&\widetilde{K}_{3}(n,t)=\frac{n-4}{t}+2(n-4),&\\\nonumber
	&\widetilde{J}_{0}(n,t)=\frac{n(n-4)}{8t^2}-\frac{(n-4)^2}{2t}-2(n-4),&\\\nonumber
	&\widetilde{J}_{1}(n,t)=-\frac{n-4}{t}+2(n-4).&
	\end{align}
	
	Again in the limit case $R=\infty$, by rotational invariance, \eqref{eq:subcriticalcylindricalsystemlower} becomes the following nonautonomous ODE system,
	\begin{equation}\label{nonautonomousodesystem}
	w_i^{(4)}+\widetilde{K}_{3}(n,t)w_i^{(3)}+\widetilde{K}_2(n,t)w_i^{(2)}+\widetilde{K}_{1}(n,t)w_i^{(1)}+\widetilde{K}_{0}(n,t)w_i=t^{-1}|\mathcal{W}|^{2_{**}-1}w_i\quad \mbox{in} \quad \mathbb{R}.
	\end{equation}
	Analogously to the standard autonomous case, we also consider a transformation that sends singular solution to \eqref{eq:subcriticalsystem} into solutions to a nonautonomous ODE on the cylinder \eqref{eq:subcriticalcylindricalsystemlower}.
	\begin{definition}
		Let us consider the non-autonomous cylindrical transformation as follows
		\begin{equation}\label{nonautonomouscyltransform}
		\tilde{\mathfrak{F}}:C_c^{\infty}(B_R^*,\mathbb{R}^p)\rightarrow C_c^{\infty}(\mathcal{C}_T,\mathbb{R}^p) \quad \mbox{given by} \quad
		\tilde{\mathfrak{F}}(\mathcal{U})=r^{4-n}(-\ln r)^{\frac{4-n}{4}}\mathcal{U}(r,\sigma),
		\end{equation}
		where $r=|x|$, $t=-\ln r$ and $\theta=x|x|^{-1}$.
	\end{definition}
	
	\begin{remark}	
		The nonautonomous cylindrical transformation \eqref{nonautonomouscyltransform} is essential to convert our PDE system into a far simpler nonautonomous ODE.
		Besides, it could be extended to study general PDE systems for which a nonlinear change of variables (likewise the Emden--Fowler change of variables) is not available yet.
	\end{remark}

	\begin{remark}
		Our choice for the symbol $\Delta^2_{\rm cyl}=\Delta^2_{\rm sph}\circ\mathfrak{F}^{-1}$ $($resp. $\widetilde{\Delta}^2_{\rm cyl}=\Delta^2_{\rm sph}\circ\widetilde{\mathfrak{F}}^{-1}$$)$ is an abuse of notation, since the cylindrical background metric is not flat, we should have $P_{\rm cyl}=\Delta_{\rm sph}^2\circ\mathfrak{F}^{-1}$ $($resp. $\widetilde{P}_{\rm cyl}=\Delta_{\rm sph}^2\circ\widetilde{\mathfrak{F}}^{-1}$$)$, where $P_{\rm cyl}$ $($resp. $\widetilde{P}_{\rm cyl}$$)$ stands for the Paneitz--Branson operator of this metric in the new logarithmic cylindrical coordinate system.
	\end{remark}
	
	\section{Pohozaev functionals}\label{sec:pohozaevinvariant}
	Next, we define two central characters in our studies, the so-called {\it autonomous and nonautonomous Pohozaev functionals} \footnote{In physics, it is also referred to as {\it balanced-energy-type functional appearing in the conservation of energy for physical systems, also known as virial theorems}}. 
	This is
	a type homological functional associated to \eqref{eq:subcriticalsystem}, which is also useful in many other contexts \cite{MR0192184,MR1666838}.
	The heuristics is that these Pohozaev functional classify whether or not the blow-up limit solutions in cylindrical coordinates are (asymptotically) stable equilibrium points of the associated fourth order ODEs \eqref{eq:subcriticalODEsystem} and \eqref{nonautonomousodesystem}. 
	More precisely, from the dynamical systems point of view, it works as a Lyapunov functional.
	In contrast with the critical case where this functional is constant, in the  subcritical setting, we show that the Pohozaev functionals satisfies a monotonicity property.
	Later, we introduce a family of scalings together with its blow-up (shrink-down) limit, which combined with blow-up analysis will be the crucial ingredient to perform the asymptotic methods in Section~\ref{sec:proof2}.
	Based on the tangent cone analysis due to Fleming \cite{MR157263,MR3190428},
	we also show that the vanishing of the derivative of the  Pohozaev invariant is related to the homogeneity of solutions.
	Let us remark that to prove the Pohozaev invariant to be well-defined, one needs to use some upper estimates and asymptotic symmetry that will be proved independently in Section~\ref{sec:proof2}.
	
	\subsection{Autonomous case}
	Initially, let us define {\it the Hamiltonian energy} associated to the PDE system on the cylinder given by \eqref{eq:subcriticalcylindricalsystem} (see \cite{MR3869387,MR4094467,MR4123335,arXiv:1804.00817}). 
	
	\begin{definition} 
		For any $\mathcal{V}$ nonnegative solution to \eqref{eq:subcriticalcylindricalsystem}, let us consider its {\it Hamiltonian energy} given by
		\begin{align}\label{vectenergy}
		\mathcal{H}(t,\theta,\mathcal{V}):=\mathcal{H}_{{\rm rad},1}(t,\theta,\mathcal{V})+\mathcal{H}_{{\rm rad},2}(t,\theta,\mathcal{V})+\mathcal{H}_{\rm ang}(t,\theta,\mathcal{V}).
		\end{align}
		Here 
		\begin{align*}
		&\mathcal{H}_{{\rm rad},1}(t,\theta,\mathcal{V})=-\left(\langle \mathcal{V}^{(3)}(t,\theta),\mathcal{V}^{(1)}(t,\theta)\rangle+K_{3}(n,s)\langle\mathcal{V}^{(2)}(t,\theta),\mathcal{V}^{(1)}(t,\theta)\rangle\right),&\\\nonumber
		&\mathcal{H}_{{\rm rad},2}(t,\theta,\mathcal{V})=\frac{1}{2}\left(|\mathcal{V}^{(2)}(t,\theta)|^{2}-{K_{2}(n,s)}|\mathcal{V}^{(1)}(t,\theta)|^{2}
		-{K_{0}(n,s)}|\mathcal{V}(t,\theta)|^2\right)+\frac{1}{s+1}|\mathcal{V}(t,\theta)|^{s+1}, \ {\rm and}&\\\nonumber
		&\mathcal{H}_{\rm ang}(t,\theta,\mathcal{V})=-\frac{1}{2}\left(|\Delta_{\theta}\mathcal{V}(t,\theta)|^2-{J}_{1}(n,s)|\nabla_{\theta}\mathcal{V}(t,\theta)|^2-|\partial_t\nabla_{\theta}\mathcal{V}(t,\theta)|^2\right),&
		\end{align*}
		where $K_{j}(n,s)$ for $j=0,1,2,3$ and $J_{0}(n,s)$ are defined by \eqref{autonomouscoefficients}.
	\end{definition}

	\begin{remark}
		In the limit case $R=\infty$, by radial symmetry, we have 
		the vanishing of the angular term in \eqref{vectenergy}, that is, $\mathcal{H}_{\rm ang}(t,\theta,\mathcal{V})\equiv0$ for all $t\in\mathbb{R}$ and $\mathcal{V}$ solution to \eqref{eq:subcriticalcylindricalsystem}.
		When $R<0$, this does not happen, which makes the ODE analysis more delicate. Especially, in the critical case $s=2^{**}-1$.
	\end{remark}
	
	For each fixed $t\in(0,\infty)$, we can integrate \eqref{vectenergy} over the cylindrical slice $\mathbb{S}_t^{n-1}=\{t\}\times\mathbb{S}^{n-1}$ to define a functional acting on $C^{\infty}(\mathcal{C}_T,\mathbb{R}^p)$ that will play a central role in our analysis.
	
	\begin{definition}\label{def:classicalpohozaev}
		For any $\mathcal{V}$ nonnegative solution to \eqref{eq:subcriticalcylindricalsystem}, let us define its {\it cylindrical Pohozaev functional} by 
		\begin{align*}
		\mathcal{P}_{\rm cyl}(t,\mathcal{V})&=\displaystyle\int_{\mathbb{S}_t^{n-1}}\mathcal{H}(t,\theta,\mathcal{V})\ud\theta.&
		\end{align*}
		Here $\mathbb{S}_t^{n-1}=\{t\}\times\mathbb{S}^{n-1}$ is the cylindrical ball with volume element given by $\ud\theta=e^{-2t}\ud\sigma$, where $\ud\sigma_r$ is the volume element of the ball of radius $r>0$ in $\mathbb{R}^n$. 
	\end{definition}
	
	Using the inverse of the cylindrical transformation, one can also construct the {\it spherical Pohozaev functional} given by $\mathcal{P}_{\rm sph}(r,\mathcal{U}):=\left(\mathcal{P}_{\rm cyl}\circ\mathfrak{F}^{-1}\right)\left(t,\mathcal{V}\right)$. 
	Also, one has the following Pohozaev-type identity.
	
	\begin{lemma}\label{lm:pohozaevidentity}
		Let $\mathcal{U},\widetilde{\mathcal{U}}\in C^{4}(B_1^*,\mathbb{R}^p)$ and $0<r_1\leqslant r_2<1$. 
		Then, it follows 
		\begin{align*}
		&\sum_{i=1}^p\int_{B_{r_2}\setminus B_{r_1}}\left[\Delta^2u_i\langle x,\nabla \widetilde{u}_i\rangle+\Delta^2\widetilde{u}_i\langle x,\nabla u_i\rangle-\gamma(s)\left(\widetilde{u}_i\Delta^2u_i+u_i\Delta^2\widetilde{u}_i\right)\right]\ud x&\\\nonumber
		&=\sum_{i=1}^p\left[\int_{\partial B_{r_2}}q(u_i,\widetilde{u}_i)\ud\sigma_{r_2}-\int_{\partial B_{r_1}}q(u_i,\widetilde{u}_i)\ud\sigma_{r_1}\right].&
		\end{align*}
		Here $\gamma(s)=4(s-1)^{-1}$ is the Fowler exponent, and
		\begin{equation}\label{pohozaeverrorfunction}
		q(u_i,\widetilde{u}_i)=\sum_{j=1}^{3}\bar{l}_j(x,D^{(j)}u_i,D^{(4-j)}\widetilde{u}_i)+\sum_{j=0}^{3}\tilde{l}_j(D^{(j)}u_i,D^{(3-j)}\widetilde{u}_i),
		\end{equation}
		where both $\bar{l}_j,\tilde{l}_j$ are linear in each component for $j=0,1,2,3$. 
		Moreover,
		\begin{equation*}
		\bar{l}_3(D^{(3)}u_i,\widetilde{u}_i)=\gamma(s)\int_{\partial B_{r_2}}\widetilde{u}_i\partial_{\nu}\Delta u_i\ud\sigma_{r_2}.
		\end{equation*}
		where $\nu$ is the outer normal vector to $\partial B_{r_2}$.
	\end{lemma}
	
	\begin{proof}
		See the proof in \cite[Proposition~3.3]{arXiv:1503.06412} and \cite[Propositon~A.1]{arxiv:1901.01678}.
	\end{proof}
	
	\begin{remark}
		Using the last lemma, we present a explicit formula for the spherical Pohozaev functional
		\begin{equation}\label{pohozaevsphericalautonomous}
		\mathcal{P}_{\rm sph}(r,\mathcal{U})=\sum_{i=1}^p\int_{\partial B_r}\left[q(u_i,u_i)-\frac{1}{s+1}r|\mathcal{U}|^{s+1}\right]\ud\sigma_r,
		\end{equation}
		where $q$ is defined by \eqref{pohozaeverrorfunction}.
	\end{remark}
	
	\begin{remark}\label{relatinpohozevfinctionals}
		For easy reference, let us summarize the following properties of the Pohozaev functionals:\\
		\noindent{\rm (i)} $\mathcal{P}_{\rm sph}(r,\mathcal{U})=\omega_{n-1}\mathcal{P}_{\rm cyl}(t,\mathcal{V})$, where $\omega_{n-1}$ is the $(n-1)$-dimensional surface measure of the unit sphere;\\ 
		\noindent{\rm (ii)} $\mathcal{P}_{\rm sph}(r,u)=(\mathcal{P}_{\rm cyl}\circ\mathfrak{F}^{-1})(t,v)$, where
		\begin{equation*}
		\mathcal{P}_{\rm cyl}(t,v)=\displaystyle\int_{\mathbb{S}^{n-1}_t}\mathcal{H}(t,\theta,v)\ud\theta.
		\end{equation*}
	\end{remark}
	
	Subsequently, we deduce a monotonicity formula for the cylindrical Pohozaev functional $\mathcal{P}_{\rm cyl}(t,\mathcal{V})$, which will be essential to show that the limit Pohozaev invariant is well-defined as $t\rightarrow\infty$ (or as  $r\rightarrow0$, in spherical coordinates).
	Such monotonicity property is inspired by \cite[Lemma~3.4]{MR4123335} (see also \cite{MR3190428,arXiv:1804.00817,MR4085120}).
	
	\begin{proposition}\label{lm:monotonicityformula}
		Let $s\in(2_{**},2^{**}-1)$, $-\infty<-\ln R<t_1\leqslant t_2$, and $\mathcal{V}$ be a nonnegative solution to \eqref{eq:subcriticalcylindricalsystem}. 
		Then, $\mathcal{P}_{\rm cyl}(t_2,\mathcal{V})-\mathcal{P}_{\rm cyl}(t_1,\mathcal{V})\geqslant0$. 
		More precisely, we have the monotonicity formula,
		\begin{equation}\label{eq:monotonicityformula}\partial_t\mathcal{P}_{\rm cyl}(t,\mathcal{V})= \int_{\mathbb{S}_t^{n-1}}\left[K_{1}(n,s)|\mathcal{V}^{(1)}|^{2}-K_{3}(n,s)\left(|\mathcal{V}^{(2)}|^{2}+|\partial_t\nabla_{\theta}\mathcal{V}|^2\right)\right]\ud\theta.
		\end{equation}
		In particular, ${\mathcal{P}}_{\rm cyl}(t,\mathcal{V})$ is nondecreasing, and so $\mathcal{P}_{\rm cyl}(\infty,\mathcal{V}):=lim_{t\rightarrow\infty}{\mathcal{P}}_{\rm cyl}(t,\mathcal{V})$ exists.
	\end{proposition}
	
	\begin{proof}
		Initially, observe that  \eqref{eq:monotonicityformula} follows by a direct computation, which consists in taking the inner product of \eqref{eq:subcriticalODEsystem} with $\mathcal{V}^{(1)}$, and integrating by parts on $\mathbb{S}_t^{n-1}$. Therefore, using Remark~\ref{rmk:signoncoefficients}, we find $\partial_t\mathcal{P}_{\rm cyl}(t,\mathcal{V})\geqslant0$, which proves that the cylindrical Pohozaev functional is nondecreasing. 
		Finally, since by Lemma~\ref{lm:universalestimates}, the Pohozaev functional is bounded above, the full existence of limit follows.
	\end{proof}
	
	\begin{corollary}
		Let $s\in(2_{**},2^{**}-1)$, $0<r_1\leqslant r_2<R<\infty$, and $\mathcal{U}$ be a nonnegative solution to \eqref{eq:subcriticalsystem}. Then, $\mathcal{P}_{\rm sph}(r_2,\mathcal{U})-\mathcal{P}_{\rm sph}(r_1,\mathcal{U})\leqslant0$. 
	\end{corollary}
	
	\begin{proof}
		Indeed, a simple computation shows the following relation between the derivative of the two Pohozaev functionals,
		\begin{equation*}
		\partial_r\mathcal{P}_{\rm sph}(t,\mathcal{U})=-e^{-t}\partial_t\mathcal{P}_{\rm cyl}(t,\mathcal{V}),
		\end{equation*}
		which combined with Proposition~\ref{lm:monotonicityformula} implies that the spherical Pohozaev functional is nonincreasing.
	\end{proof}
	
	\begin{remark}\label{rmk:limitrelation}
		Using the inverse of cylindrical transform, that is,  $\mathcal{P}_{\rm sph}=\mathcal{P}_{\rm cyl}\circ\mathfrak{F}^{-1}$, it follows that $\mathcal{P}_{\rm sph}(0,\mathcal{U})=\lim_{r\rightarrow0}\mathcal{P}_{\rm sph}(r,\mathcal{U})=lim_{t\rightarrow\infty}{\mathcal{P}}_{\rm cyl}(t,\mathcal{V}).$
		The last equality implies that the Pohozaev invariant is well-defined in the punctured ball, when $R\rightarrow0$.
	\end{remark}
	
	Now for any $\mathcal{U}$ a solution to \eqref{eq:subcriticalsystem}, for any $\lambda>0$, we define the following $\lambda$-rescaling given by
	\begin{equation*}
	\widehat{\mathcal{U}}_{\lambda}(x):=\lambda^{\gamma(s)}\mathcal{U}(\lambda x).
	\end{equation*}
	Notice that the $\lambda$-rescaled solution is still a nonnegative solution to \eqref{eq:subcriticalsystem} with $R=\lambda^{-1}$.
	Moreover, we get the following scaling invariance
	\begin{equation}\label{scaleinvariance}
	\mathcal{P}_{\rm sph}(r,\widehat{\mathcal{U}}_{\lambda})=\mathcal{P}_{\rm sph}(\lambda r,\mathcal{U}).
	\end{equation}
	This follows by directly using the inverse cylindrical transform as in Remark~\ref{rmk:limitrelation}.
	
	Next, we prove that the invariance of the Pohozaev invariant is equivalent to the homogeneity of the blow-up limit solutions to \eqref{eq:subcriticalsystem}.
	A similar result can also be found in \cite{MR3190428}, where a different type of functional is considered.
	
	\begin{lemma}\label{lm:homogeneity}
		Let $0<r_1\leqslant r_2<R<\infty$, $s\in(2_{**},2^{**}-1)$,  and $\mathcal{U}$ be a nonnegative solution to \eqref{eq:subcriticalsystem}. 
		Then, $\mathcal{P}_{\rm sph}(r,\mathcal{U})$ is constant, if, and only if, for any $r\in(r_{1},r_{2})$, $\mathcal{U}$ is homogeneous of degree $\gamma(s)>0$ in $B_{r_{2}}\setminus\bar{B}_{r_{1}}$, that is,
		\begin{equation*}
		\mathcal{U}(x)=|x|^{-\gamma(s)} \mathcal{U}\left(\frac{x}{|x|}\right) \quad \mbox{in} \quad B_{r_{2}}\setminus\bar{B}_{r_{1}}.
		\end{equation*}
	\end{lemma}
	
	\begin{proof}
		Notice that if $s\neq 2^{**}-1$, then $K_0(n,s)\neq 0$.
		Thus, supposing that $\mathcal{P}_{\rm sph}(r,\mathcal{U})$ is constant for $r_{1}<r<r_{2}$, together with \eqref{pohozaevsphericalautonomous} yields that $\partial_{\nu}\mathcal{U}=\gamma(s)r^{-1}\mathcal{U}$ on $\partial B_{r}$ for any $r_{1}<r<r_{2}$, where $\nu$ is the unit normal pointing towards the origin. 
		Therefore, $\mathcal{U}$ is homogeneous of degree $-\gamma(s)$ in $B_{r_{2}}\setminus\bar{B}_{r_{1}}$, which concludes the proof.
	\end{proof}
	
	\subsection{Nonautonomous case}
	In the lower critical case $s=2_{**}$, since $K_{0,*}:=K_0(n,2_{**})=0$ and $\gamma(2_{**})=n-4$, a new cylindrical transformation was defined.
	Concerning this nonautonomous cylindrical transformation, we compute the associated Pohozaev functional, which in this situation has some time-dependent terms.
	Then, we also need to prove a monotonicity formula such as the one in Lemma~\ref{lm:monotonicityformula}.
	Nevertheless, a far more subtle analysis is required because of the nonautonomous terms, which now depends on the dimension.
	We believe that this new monotonicity formula is relevant on its own, since it can be further applied to provide asymptotic analysis for general PDE systems, for which a usual autonomous Pohozaev invariant is not available.
	First, to shorten our notation, we omit the variables $n$ and $t$ for the coefficients defined by \eqref{nonautonomouscoeficcients}, and the variables $\theta$ and $t$ for $p$-maps on the cylinder, when it is convenient.
	\begin{definition} 
		For any $\mathcal{W}$ solution to \eqref{eq:subcriticalcylindricalsystemlower}, let us consider its {\it Hamiltonian energy} defined by
		\begin{align}\label{modifiedvectenergy}
		\widetilde{\mathcal{H}}(t,\theta,\mathcal{W}):=\widetilde{\mathcal{H}}_{{\rm rad},1}(t,\theta,\mathcal{W})+\widetilde{\mathcal{H}}_{{\rm rad},2}(t,\theta,\mathcal{W})+\widetilde{\mathcal{H}}_{\rm ang}(t,\theta,\mathcal{W}).
		\end{align}
		Here,
		\begin{align*}
		\widetilde{\mathcal{H}}_{{\rm rad},1}(t,\theta,\mathcal{W})&=-t\left(\langle \mathcal{W}^{(3)}(t,\theta),\mathcal{W}^{(1)}(t,\theta)\rangle+\widetilde{K}_{3}(n,t)\langle\mathcal{W}^{(2)}(t,\theta),\mathcal{W}^{(1)}(t,\theta)\rangle\right),\\\nonumber
		\widetilde{\mathcal{H}}_{{\rm rad},2}(t,\theta,\mathcal{W})&=\frac{t}{2}\left(|\mathcal{W}^{(2)}(t,\theta)|^{2}-{\widetilde{K}_{2}(n,t)}|\mathcal{W}^{(1)}(t,\theta)|^{2}
		-{\widetilde{K}_{0}(n,t)}|\mathcal{W}(t,\theta)|^2\right)\\
		&+(2_{**}+1)^{-1}|\mathcal{W}(t,\theta)|^{2_{**}+1},\\\nonumber
		\widetilde{\mathcal{H}}_{\rm ang}(t,\theta,\mathcal{W})&=-\frac{t}{2}\left(|\Delta_{\theta}\mathcal{V}(t,\theta)|^2-\widetilde{J}_{0}(n,t)|\nabla_{\theta}\mathcal{V}(t,\theta)|^2-2|\partial_t\nabla_{\theta}\mathcal{V}(t,\theta)|^2\right),
		\end{align*}
		where $\widetilde{K}_{j}(n,t)$ for $j=0,1,2,3$ and $\widetilde{J}_{1}(n,t)$ are defined by \eqref{nonautonomouscoeficcients}.
	\end{definition}

	\begin{remark}
		As in the autonomous case, when $R=\infty$, using radial symmetry, we find
		$\widetilde{\mathcal{H}}_{\rm ang}(t,\theta,\mathcal{W})\equiv0$ for all $t\in\mathbb{R}$ and $\mathcal{W}$ solution to \eqref{eq:subcriticalcylindricalsystemlower}.
	\end{remark}
	
	Again, for each fixed slice, we can integrate \eqref{modifiedvectenergy} over $\mathbb{S}_t^{n-1}$ to define another functional, which will be well-defined for any $t\in(0,\infty)$ since $\mathcal{U}$ is smooth away from the origin. 
	
	\begin{definition} 
		For any $\mathcal{W}$ solution to \eqref{eq:subcriticalcylindricalsystemlower}, let us define its {\it cylindrical Pohozaev functional} by 
		\begin{align*}
		\widetilde{\mathcal{P}}_{\rm cyl}(t,\mathcal{W})&=\displaystyle\int_{\mathbb{S}_t^{n-1}}\widetilde{\mathcal{H}}(t,\theta,\mathcal{W})\ud\theta.&
		\end{align*} 
	\end{definition}
	
	\begin{remark}\label{rmk:estimates}
		Due to Proposition~\ref{prop:asymptoticsymmetry}, for $R<\infty$, we get that $|\mathcal{W}(t,\theta)-\overline{\mathcal{W}}(t)|=\mathcal{O}\left(e^{-\beta t}\right)$ as $t\rightarrow\infty$, for some $\beta>0$, where $\overline{\mathcal{W}}(t)$ is the average of $\mathcal{W}(t,\theta)$ over $\theta\in \mathbb{S}_t^{n-1}$. In particular, one can find some large $T_{0}\gg1$ and $C>0$ (independent of $t$) satisfying
		\begin{equation}\label{angularestimates}
		\left|\nabla_{\theta} \mathcal{W}(t,\theta)\right|+\left|\Delta_{\theta} \mathcal{W}(t,\theta)\right| \leqslant Ce^{-\beta t} \quad {\rm in} \quad \mathcal{C}_{T_0}.
		\end{equation}
		Moreover, from the sharp estimate in Proposition~\ref{lm:monotonicityformula} and the gradient estimate \eqref{lm:gradientestimates}, it follows 
		\begin{equation}\label{radialestimates}
		|\mathcal{W}(t,\theta)|+|\mathcal{W}^{(1)}(t,\theta)|+|\mathcal{W}^{(2)}(t,\theta)|+|\mathcal{W}^{(3)}(t,\theta)|\leqslant C \quad {\rm in} \quad \mathcal{C}_{T_0}.
		\end{equation}
	\end{remark}
	
	Before we proceed to the proof of the monotonicity formula, we need to show an auxiliary result concerning the asymptotic behavior of local solutions to \eqref{eq:subcriticalsystem} and its derivatives
	when $t\gg1$ is sufficiently large.
	
	\begin{lemma}\label{lm:signproperties}
		Let $R<\infty$, $s=2_{**}$, and $\mathcal{W}$ be the nonautonomous cylindrical transform given by \eqref{newcylindrical}. Then, $\lim_{t\rightarrow\infty}|\mathcal{W}(t,\theta)|\in \left\{0,\widehat{K}_{0}(n)^{\frac{n-4}{4}}\right\}$, 
		where $\widehat{K}_{j}(n)=\lim_{t\rightarrow\infty}t\widetilde{K}_j(n,t)$ for $j=0,1,2,3$.
		Moreover,
		\begin{equation*}
		\lim_{t\rightarrow\infty}|{\mathcal{W}}^{(j)}(t,\theta)|=\lim_{t\rightarrow\infty}|\overline{\mathcal{W}}^{(j)}(t)|=0 \quad {\rm for \ all} \quad j\geqslant1.
		\end{equation*}
	\end{lemma} 
	
	\begin{proof}
		Indeed, using the asymptotic symmetry in Proposition~\ref{prop:asymptoticsymmetry} and the estimates in Remark~\ref{rmk:estimates}, we have that the cylindrical transformation of spherical average $\overline{\mathcal{W}}=\widetilde{\mathfrak{F}}(\overline{\mathcal{U}})$ satisfies,
		\begin{equation*}
		t\overline{w}_i^{(4)}+\widehat{K}_{3}(n)\overline{w}_i^{(3)}+\widehat{K}_{2}(n)\overline{w}_i^{(2)}+\widehat{K}_{1}(n)\overline{w}_i^{(1)}=|\overline{\mathcal{W}}|^{2_{**}-1}\overline{w}_i-\widehat{K}_{0}(n)\overline{w}_i+\mathcal{O}(\overline{w}_i(t)e^{-t}) \quad \mbox{as} \quad t\rightarrow\infty.
		\end{equation*} 
		Thus, decomposing the left-hand side of the last equation into a product of two second order operators, we can adapt the arguments in \cite[Section~2]{MR4094467} (see also \cite[Section~4]{arXiv:2002.12491})
		for the last fourth order asymptotic identity. In this fashion, we split the proof of the claim into two steps:
		
		\noindent{\bf Step 1:} Either $|\mathcal{W}(t,\theta)|=o(1)$ or $|\mathcal{W}(t,\theta)|=\widehat{K}_0(n)^{\frac{n-4}{4}}+o(1)$ as $t\rightarrow\infty$.
		
		\noindent As a matter of fact, the conclusion follows directly from the same argument in \cite[Lemma~2.1]{MR4094467} to the limit ODE \eqref{eq:subcriticalODEsystem} (see also \cite[Lemma~38]{arXiv:2002.12491}).
		
		\noindent{\bf Step 2:} $|\partial_t^{(j)}\mathcal{W}(t,\theta)|=o(1)$ as $t\rightarrow\infty$ for all $j\geqslant1$.
		
		\noindent Indeed, we can adapt the proof in \cite[Lemma~2.1]{MR4094467} to the ODE system \eqref{eq:subcriticalODEsystem} (see also \cite[Lemma~39]{arXiv:2002.12491}).	
		
		Steps 1 and 2 together give the proof for the first claim.
	\end{proof}

		To prove the full existence of $\widetilde{\mathcal{P}}_{\rm cyl}(\infty,\mathcal{W})$, we shall verify the estimates in the next lemma, which is a fourth order version of \cite[Lemma~7.13]{MR4085120} (see also \cite[Lemma~3.2]{MR875297}). 
	Due to the appearance of higher order derivative terms, we give a different proof to the ones in the lemma quoted above. 
	Namely, our proof is based on Lemma~\ref{lm:signproperties} combined with a simple L'H\^ospital rule.

	\begin{lemma}\label{lm:estimateangularparts}
		Let $R<\infty$, $s=2_{**}$, and $\mathcal{W}$ be the nonautonomous cylindrical transform given by \eqref{nonautonomouscyltransform}. Then, 
		\begin{equation}\label{radialestimate}
			\lim_{t\rightarrow\infty}\Xi_{\rm rad}(t,\mathcal{W})=0,
		\end{equation}
		where $\Xi_{\rm rad}$ is defined by \eqref{radialerror}.
	\end{lemma}
	
	\begin{proof}
		Now we combine the asymptotic estimate proved in Lemma~\ref{lm:signproperties} with the L'H\^ospital rule to prove this convergence to zero.
		In fact, notice that \eqref{radialestimate} can be rewritten as
		\begin{align*}
			\Xi_{\rm rad}(t,\mathcal{W})&:=I_4+I_3+I_2+I_1+I_0.
		\end{align*}
		Next, we estimate each term of the last identity separately by steps.
		
		\noindent{\bf Step 1:} $I_4:=\lim_{t\rightarrow\infty}\int_{\mathbb{S}^{n-1}}\langle \mathcal{W}^{(3)}(t,\theta),\mathcal{W}^{(1)}(t,\theta)\rangle\ud\theta=0$.
		
		\noindent It follows directly by Lemma~\ref{lm:signproperties}.
		
		\noindent{\bf Step 2:} $I_3=\lim_{t\rightarrow\infty}\int_{\mathbb{S}^{n-1}}\mathfrak{p}_3(n,t)\langle\mathcal{W}^{(2)}(t,\theta),\mathcal{W}^{(1)}(t,\theta)\rangle\ud\theta=0$.
		
		\noindent In fact, using \eqref{nonautonomouspohozaevcoefficients}, one has that  $\lim_{t\rightarrow\infty}\mathfrak{p}_3(n,t)<\infty$. Thus, 
		\begin{align*}
			\lim_{t\rightarrow\infty}\int_{\mathbb{S}^{n-1}}\mathfrak{p}_3(n,t)\langle\mathcal{W}^{(2)}(t,\theta),\mathcal{W}^{(1)}(t,\theta)\rangle\ud\theta=\left(\lim_{t\rightarrow\infty}\mathfrak{p}_3(n,t)\right)\left(\lim_{t\rightarrow\infty}\int_{\mathbb{S}^{n-1}}\langle\mathcal{W}^{(2)}(t,\theta),\mathcal{W}^{(1)}(t,\theta)\rangle\ud\theta\right)=0,&
		\end{align*}
		which by Lemma~\ref{lm:signproperties}
		leads to the conclusion.
		
		\noindent{\bf Step 3:}
		$I_2=\lim_{t\rightarrow\infty}\int_{\mathbb{S}^{n-1}}\mathfrak{p}_2(n,t)|\mathcal{W}^{(2)}(t,\theta)|^{2}\ud\theta=0.$
		
		\noindent Again by \eqref{nonautonomouspohozaevcoefficients}, it holds that $\lim_{t\rightarrow\infty}\mathfrak{p}_2(n,t)<\infty$ and using
		\begin{align*}
			\lim_{t\rightarrow\infty}\int_{\mathbb{S}^{n-1}}\mathfrak{p}_2(n,t)|\mathcal{W}^{(2)}(t,\theta)|^2\ud\theta=\left(\lim_{t\rightarrow\infty}\mathfrak{p}_2(n,t)\right)\left(\lim_{t\rightarrow\infty}\int_{\mathbb{S}^{n-1}}|\mathcal{W}^{(2)}(t,\theta)|^2\ud\theta\right)=0,&
		\end{align*}
		the proof of this step follows promptly.
		
		\noindent{\bf Step 4:} $I_1=\lim_{t\rightarrow\infty}\int_{\mathbb{S}^{n-1}}\mathfrak{p}_1(n,t)|\mathcal{W}^{(1)}(t,\theta)|^2\ud\theta=0$.
		
		\noindent The difference in this situation is that  $\lim_{t\rightarrow\infty}\mathfrak{p}_1(n,t)=\infty$.
		However, using \eqref{nonautonomouspohozaevcoefficients}, we can decompose $\mathfrak{p}_1(n,t)=\widetilde{\mathfrak{p}}_1(n,t)+\widehat{\mathfrak{p}}_1(n,t)$, where $\lim_{t\rightarrow\infty}\widetilde{\mathfrak{p}}_1(n,t)<\infty$ and
		\begin{equation*}
			\widehat{\mathfrak{p}}_1(n,t)=-2(n-2)(n-4)t.
		\end{equation*}
		Then, we get
		\begin{align*}
			\lim_{t\rightarrow\infty}\int_{\mathbb{S}^{n-1}}\mathfrak{p}_1(n,t)|\mathcal{W}(t,\theta)|^2\ud\theta&=\lim_{t\rightarrow\infty}\left[\widetilde{\mathfrak{p}}_1(n,t)\int_{\mathbb{S}^{n-1}}|\mathcal{W}(t,\theta)|^2\ud\theta\right]+\lim_{t\rightarrow\infty}\left[\widehat{\mathfrak{p}}_1(n,t)\int_{\mathbb{S}^{n-1}}|\mathcal{W}(t,\theta)|^2\ud\theta\right]&\\
			&:=\widetilde{I}_{1}+\widehat{I}_{1}.&
		\end{align*}
		Notice that since $\lim_{t\rightarrow\infty}\widetilde{\mathfrak{p}}_1(n,t)=0$, we trivially see that $\widetilde{I}_{1}=0$, . Moreover, in order to estimate $\widehat{I}_{1}$, we use the L'H\^ospital rule as follows
		\begin{align*}
			\widehat{I}_{1}=\lim_{t\rightarrow\infty}\int_{\mathbb{S}^{n-1}}\widehat{\mathfrak{p}}_1(n,t)|\mathcal{W}(t,\theta)|^2\ud\theta=\lim_{t\rightarrow\infty}\frac{-2(n-2)(n-4)}{\partial^{(1)}_t\left[\left(\int_{\mathbb{S}^{n-1}}|\mathcal{W}(t,\theta)|^2\ud\theta\right)^{-1}\right]}=0,
		\end{align*}
		which proves this step.
		
		\noindent{\bf Step 5:}
		$I_0=\lim_{t\rightarrow\infty}\int_{\mathbb{S}^{n-1}}\mathfrak{p}_0(n,t)|\mathcal{W}(t,\theta)|^2\ud\theta=0$.
		
		\noindent As  before, we can compute $\lim_{t\rightarrow\infty}\mathfrak{p}_0(n,t)=0$. Consequently, we get
		\begin{align*}
			\lim_{t\rightarrow\infty}\int_{\mathbb{S}^{n-1}}\mathfrak{p}_0(n,t)|\mathcal{W}^{(1)}(t,\theta)|^2\ud\theta=\left(\lim_{t\rightarrow\infty}\mathfrak{p}_0(n,t)\right)\left(\lim_{t\rightarrow\infty}\int_{\mathbb{S}^{n-1}}|\mathcal{W}^{(1)}(t,\theta)|^2\ud\theta\right)=0,&
		\end{align*}
		which gives us the desired conclusion.
		
		Finally, putting together all the steps, it follows that Claim 2 holds, and therefore, the proof is concluded.
	\end{proof}
	
	The next proposition is the most important result of this subsection, namely the monotonicity of this new Pohozaev functional.
	We should point out that different from the second order case, the asymptotic sign of the Pohozaev functional derivative depends on the dimension.
	Notice that for a well-defined limit behavior, it is only necessary that the sign is constant and non-zero for large time.
	
	\begin{proposition}\label{lm:lowermonotonicity}
		Let $R<\infty$, $s=2_{**}$, and $\mathcal{U}$ be a nonnegative solution to \eqref{eq:subcriticalsystem}.
		Then, there exists $T_1\gg1$ sufficiently large such that $\operatorname{sgn}(\partial_t\widetilde{\mathcal{P}}_{\rm cyl}(t,\mathcal{W}))$ is non-vanishing and constant for $t>T_1$. Namely,\\
		\noindent{\rm (i)}
		if $5\leqslant n\leqslant7$, then
		there exists $T_*\gg1$ such that $\widetilde{\mathcal{P}}_{\rm cyl}(t,\mathcal{W})$ is nonincreasing for $t>T_*$.\\
		\noindent{\rm (ii)} if $n\geqslant8$, then
		there exists $T^*\gg1$ such that $\widetilde{\mathcal{P}}_{\rm cyl}(t,\mathcal{W})$ is nondecreasing for $t>T^*$.\\
		Moreover, $\widetilde{\mathcal{P}}_{\rm cyl}(\infty,\mathcal{W}):=lim_{t\rightarrow\infty}\widetilde{\mathcal{P}}_{\rm cyl}(t,\mathcal{W})$ exists.
	\end{proposition}
	
	\begin{proof}
		By differentiating $\widetilde{\mathcal{H}}(t,\mathcal{W})$ with respect to $t$, we get
		\begin{align*}
		\partial_t\widetilde{\mathcal{H}}(t,\theta,\mathcal{W})&=-\left(\langle \mathcal{W}^{(3)},\mathcal{W}^{(1)}\rangle+\widetilde{K}_{3}\langle\mathcal{W}^{(2)},\mathcal{W}^{(1)}\rangle\right)+\frac{1}{2}\left(|\mathcal{W}^{(2)}|^{2}-{\widetilde{K}_{2}}|\mathcal{W}^{(1)}|^{2}
		-{\widetilde{K}_{0}}|\mathcal{W}|^2\right)&\\
		&-\frac{1}{2}\left(|\Delta_{\theta}\mathcal{W}|^2+\widetilde{J}_{0}|\nabla_{\theta}\mathcal{W}|^2-2|\partial_t\nabla_{\theta}\mathcal{W}|^2\right)+\langle|\mathcal{W}|^{2_{**}}\mathcal{W},\mathcal{W}^{(1)}\rangle&\\\nonumber
		&-t\left(\widetilde{K}_{3}|\mathcal{W}^{(2)}|^2+\widetilde{K}^{(1)}_{3}\langle\mathcal{W}^{(2)},\mathcal{W}^{(1)}\rangle-\frac{\widetilde{K}^{(1)}_{2}}{2}|\mathcal{W}^{(1)}|^2-\frac{\widetilde{K}^{(1)}_{0}}{2}|\mathcal{W}|^2+\frac{\widetilde{J}^{(1)}_{0}}{2}|\nabla_{\theta}\mathcal{W}|^2\right)&\\\nonumber
		&-t\left(\langle \mathcal{W}^{(4)}+\widetilde{K}_{3}\mathcal{W}^{(3)}+\widetilde{K}_{2}\mathcal{W}^{(2)}+\widetilde{K}_{0}\mathcal{W},\mathcal{W}^{(1)}\rangle\right)-\frac{t}{2}\langle\Delta_{\theta}\mathcal{W},\partial_t\Delta_{\theta}\mathcal{W}\rangle&\\
		&-\frac{t}{2}\langle\widetilde{J}_{0}\nabla_{\theta}\mathcal{W},\partial_t\nabla_{\theta}\mathcal{W}\rangle-t\langle 2\partial^{(2)}_t\nabla_{\theta}\mathcal{W},\mathcal{W}^{(1)}\rangle.&
		\end{align*}
		Now integrating by parts over $\mathbb{S}_t^{n-1}$ (using differentiation under the integral sign one can even omit the dependence on $t$) combined with \eqref{eq:subcriticalcylindricalsystemlower}, provides
		\begin{align*}
		\partial_t\widetilde{\mathcal{P}}_{\rm cyl}(t,\mathcal{W})&=\int_{\mathbb{S}^{n-1}_t}\left[-\left(\langle \mathcal{W}^{(3)},\mathcal{W}^{(1)}\rangle+\widetilde{K}_{3}\langle\mathcal{W}^{(2)},\mathcal{W}^{(1)}\rangle\right)+\frac{1}{2}\left(|\mathcal{W}^{(2)}|^{2}-{\widetilde{K}_{2}}|\mathcal{W}^{(1)}|^{2}
		-{\widetilde{K}_{0}}|\mathcal{W}|^2\right)\right.&\\
		&-t\left(\widetilde{K}_{3}|\mathcal{W}^{(2)}|^2+\widetilde{K}^{(1)}_{3}\langle\mathcal{W}^{(2)},\mathcal{W}^{(1)}\rangle+\frac{\widetilde{K}^{(1)}_{2}}{2}|\mathcal{W}^{(1)}|^2+\frac{\widetilde{K}^{(1)}_{0}}{2}|\mathcal{W}|^2-\widetilde{K}_{1}|\mathcal{W}^{(1)}|^2\right)&\\\nonumber
		&\left.-\frac{1}{2}\left(|\Delta_{\theta}\mathcal{W}|^2+\widetilde{J}_{0}|\nabla_{\theta}\mathcal{W}|^2-2|\partial_t\nabla_{\theta}\mathcal{W}|^2+\widetilde{J}^{(1)}_{0}|\nabla_{\theta}\mathcal{W}|^2-\widetilde{J}_{1}|\partial_t\Delta_{\theta}\mathcal{W}|\right)\right]\ud\theta,&
		\end{align*}
		which can be reformulated as
		\begin{align*}
		\partial_t\widetilde{\mathcal{P}}_{\rm cyl}(t,\mathcal{W})
		=\Xi_{\rm ang}(t,\mathcal{W})+\Xi_{\rm rad}(t,\mathcal{W}).
		\end{align*}
		Here 
		\begin{align}\label{angularerror}
		\Xi_{\rm ang}(t,\mathcal{W}):=-\frac{1}{2}\int_{\mathbb{S}^{n-1}}\left(|\Delta_{\theta}\mathcal{W}|^2+\widetilde{J}_{0}|\nabla_{\theta}\mathcal{W}|^2-2|\partial_t\nabla_{\theta}\mathcal{W}|^2+\widetilde{J}^{(1)}_{0}|\nabla_{\theta}\mathcal{W}|^2-\widetilde{J}_{1}|\partial_t\Delta_{\theta}\mathcal{W}|^2\right)\ud\theta
		\end{align}
		and
		\begin{align}\label{radialerror}
		\Xi_{\rm rad}(t,\mathcal{W})&:=\int_{\mathbb{S}^{n-1}}\left[\langle -\mathcal{W}^{(3)}+\mathfrak{p}_3\mathcal{W}^{(2)},\mathcal{W}^{(1)}\rangle+\mathfrak{p}_2|\mathcal{W}^{(2)}|^{2}+\mathfrak{p}_1|\mathcal{W}^{(1)}|^2+\mathfrak{p}_0|\mathcal{W}|^2\right]\ud\theta,&
		\end{align}
		where
		\begin{align*}
		&\mathfrak{p}_3(n,t):=-\left[\widetilde{K}_{3}(n,t)+\widetilde{K}_{3}^{(1)}(n,t)\right], \quad \mathfrak{p}_2(n,t):=-\frac{1}{2}\left[2t\widetilde{K}_{3}(n,t)-1\right],&\\
		&\mathfrak{p}_1(n,t):=-\frac{1}{2}\left[{\widetilde{K}_{2}(n,t)}+
		t{\widetilde{K}_{2}^{(1)}(n,t)}-2t\widetilde{K}_{1}(n,t)\right], \quad  \mbox{and} \quad \mathfrak{p}_0(n,t):=-\frac{1}{2}\left[{\widetilde{K}_{0}(n,t)}+{t}{\widetilde{K}_{0}^{(1)}(n,t)}\right].&
		\end{align*}
		More explicitly,
		\begin{align}\label{nonautonomouspohozaevcoefficients}
		\nonumber
		&\mathfrak{p}_3(n,t):=-\frac{n-4}{t^2}-\frac{n-4}{t}-2(n-4),&\\	\nonumber
		&\mathfrak{p}_2(n,t):=-\frac{n-4}{t}-\frac{4n-17}{2},&\\
		&\mathfrak{p}_1(n,t):=\frac{n(n+7)(n-4)}{16t^2}+\frac{3n(n-4)^2}{8t}+5(7n-10)-2(n-2)(n-4)t,&\\\nonumber &\mathfrak{p}_0(n,t):=\frac{3(n-4)n(n+4)(n+8)}{512t^4}+\frac{(n-4)^2n(n+4)}{32t^3}+\frac{(n-4)n(n^2-10n+20)}{32t^2}.&
		\end{align}
		Let us consider
		\begin{equation*}
		\Upsilon(t,\mathcal{W})=\langle -\mathcal{W}^{(3)}+\mathfrak{p}_3(n,t)\mathcal{W}^{(2)},\mathcal{W}^{(1)}\rangle+\mathfrak{p}_2(n,t)|\mathcal{W}^{(2)}|^{2}+\mathfrak{p}_1(n,t)|\mathcal{W}^{(1)}|^2+\mathfrak{p}_0(n,t)|\mathcal{W}|^2,
		\end{equation*}
		the integrand in \eqref{radialerror}.
		Notice that by Remark~\ref{rmk:estimates} and Lemma~\ref{lm:signproperties} there exists $T_0\gg1$ such that $\operatorname{sign}(\Upsilon(t,\mathcal{W}))=\operatorname{sign}(\mathfrak{p}_0(n,t)|\mathcal{W}|^2)$ for $t>T^*$. 
		Besides, we directly verify for $n\geqslant8$ that there exist $t_0\gg1$ sufficiently large such that 
		$\mathfrak{p}_0(n,t)>0$ for $t>t_0$, which, by taking $T^*:=\max\{T_0,t_0\}$, implies that $\widetilde{\mathcal{P}}_{\rm cyl}(t,\mathcal{W})$ is nondecreasing for $t>T^*$.
		However, for $5\leqslant n\leqslant7$, there exists $t_1\gg1$ sufficiently large such that 
		$\mathfrak{p}_0(t)>0$ for $t>t_1$; thus, setting $T_*:=\max\{T_0,t_1\}$, we get that $\widetilde{\mathcal{P}}_{\rm cyl}(t,\mathcal{W})$ is nonincreasing for $t>T_*$.
		Hence, the existence of $\widetilde{\mathcal{P}}_{\rm cyl}(\infty,\mathcal{W})$ follows  since $\widetilde{\mathcal{P}}_{\rm cyl}(t,\mathcal{W})$ is uniformly bounded both from above and below as $t \rightarrow \infty$, and, by \eqref{angularestimates}, it gives us
		\begin{equation*}
		\lim_{t\rightarrow\infty}\Xi_{\rm ang}(t,\mathcal{W})=0,
		\end{equation*}
		where $\Xi_{\rm ang}$ is defined by \eqref{angularerror}.
		This combined with \eqref{radialestimates} yields that $\liminf _{t \rightarrow \infty} \widetilde{\mathcal{P}}_{\rm cyl}(t,\mathcal{W})<\infty$.
		The proof is concluded.
	\end{proof}

	\section{Classification for the blow-up limit solutions}\label{sec:blowlimitcase}
	This section is devoted to provide the proof of Theorem~\ref{Thm3.1:classification}. 
	Our strategy is based on a blow-up (shrink-down) argument inspired by \cite{MR157263,MR3190428} combined with the moving spheres technique from \cite{MR2055032} for the system of integral equations given by \eqref{integralsystem}. Next, we use the study of the limiting levels of the Pohozaev functional under the family of rescalings to prove the remaining part of the theorem. The outline here is similar in spirit to the one in \cite{MR4085120}.
	
	\subsection{Non-singular case: The strong Liouville theorem}\label{subsec:non-singularcase}
	We assume that the origin is a removable singularity. Our objective is to prove that solutions must be trivial, a type of Pohozaev non-existence result for System \eqref{eq:subcriticalsystem}.
	To classify the solutions to \eqref{eq:subcriticalsystem} in the whole space $\mathbb{R}^{n}$, we use an integral version of the moving spheres method, based on the Kelvin transform introduced in Subsection~\ref{sec:kelvintransform}.
	
	\begin{proposition}\label{prop:nonexistencenonsingular}
		Let $R=\infty$, $s\in(1,2^{**}-1)$, and $\mathcal{U}$ be a nonnegative non-singular solution to \eqref{eq:subcriticalsystem}. If the origin is a removable singularity, then
		$\mathcal{U}\equiv0$.
	\end{proposition}
	
	Initially, we have the following background results from \cite[Lemma~2.1]{MR2001065}, which allow us to run the sliding technique for any component solution.
	
	\begin{lemmaletter}\label{lmlt:zhang1}
		Let $u \in C^{2}\left(\mathbb{R}^{n}\right)$ be a  nonnegative superharmonic function in $\mathbb{R}^{n}$. Then, for each $x \in \mathbb{R}^{n}$, there exists $\mu_{0}>0,$ which may depend on $u$ and $x\in\mathbb{R}^n$ such that for all $0<\mu<\mu_{0}$, it follows that $u_{x,\mu} \leqslant u$  in $\mathbb{R}^{n}\setminus B_{\mu}(z)$.
	\end{lemmaletter}
	
	\begin{lemmaletter}\label{lmll:zhang2}
		Let $u \in C^{2}\left(\mathbb{R}^{n}\right), x \in \mathbb{R}^{n},$ and $\mu_{0}>0$ satisfying  $-\Delta\left(u-u_{x,\mu}\right) \geqslant0$ and $u_{x,\mu_{0}}<u$ in $\mathbb{R}^{n}\setminus \bar{B}_{\mu_{0}}(z)$.
		Then, there exists a small $0<\varepsilon\ll1$ such that for any $\mu_{0}<\mu<\mu_{0}+\varepsilon$, it follows that $u_{x,\mu}<u$ in $\mathbb{R}^{n}\setminus B_{\mu}(z)$.
	\end{lemmaletter}
	
	We start our analysis by showing that the blow-up limit solutions are superharmonic, which proof is based on \cite{MR1769247}.
	Let us remark that the same argument appears in \cite[Proposition~11]{arXiv:2002.12491} when $s=2^{**}-1$.
	Nonetheless, we include the proof here for the sake of completeness.
	
	\begin{lemma}\label{lm:strongsuperharmonicity}
		Let $R=\infty$, $s\in(2_{**},2^{**}-1)$, and $\mathcal{U}$ be a nonnegative non-singular solution to \eqref{eq:subcriticalsystem}. Then, $\mathcal{U}$ is superharmonic, that is, $-\Delta u_i\geqslant0$ in $\mathbb{R}^n$ for all $i\in I$.
	\end{lemma}
	
	\begin{proof} 
		By contradiction, assume that the proposition does not hold.
		Whence, one can find $i\in I$ and $x_0\in\mathbb{R}^n$ such that $-\Delta u_{i}(x_0)<0$. By the invariance of the Laplacian, we suppose without loss of generality that $x_0=0$. Let us rewrite \eqref{eq:subcriticalsystem} as the following system in the whole space $\mathbb{R}^n$,
		\begin{align}\label{lane-emden-system}
		\begin{cases}
		-\Delta {u}_i={h}_i\\
		-\Delta {h}_i=|{\mathcal{U}}|^{s-1}{u}_i.
		\end{cases}
		\end{align}
		
		Letting $B_r\subseteq\mathbb{R}^n$ be the ball of radius $r>0$, and $\omega_{n-1}$ be 
		the $(n-1)$-dimensional surface measure of the unit sphere, we consider 
		\begin{equation*}
		\overline{u}_i=\frac{1}{n\omega_{n-1}r^{n-1}}\displaystyle\int_{\partial{B_r}}u_i\ud \sigma_r \quad {\rm and} \quad \overline{h}_i=\frac{1}{n\omega_{n-1}r^{n-1}}\displaystyle\int_{\partial{B_r}}h_i\ud \sigma_r, 
		\end{equation*}
		the spherical averages of $u_i$ and $h_i$, respectively. Now taking the spherical average on the first line of \eqref{lane-emden-system}, and using that $\overline{\Delta u_i}=\Delta\overline{u}_i$, implies
		\begin{equation*}
		\Delta\overline{u}_i+\overline{h}_i=0.
		\end{equation*}
		Furthermore, we rewrite the second equality of \eqref{lane-emden-system} to get $\Delta h_i+|\mathcal{U}|^{s-1}u_i=0$, which, by taking again the spherical average in both sides, provides
		\begin{equation*}
		\Delta \overline{h}_i+\overline{u}_i^{s}\leqslant0,
		\end{equation*}
		from which one can generate a contradiction using the iteration argument in \cite[Theorem~2.1]{MR1769247}.
		The proof is concluded.
	\end{proof}
	
	We have the following result from \cite[Lemma 3.1]{arxiv:1901.01678}, whose proof we also include here for completeness. 
	
	\begin{lemma}\label{movingpsheresineq}
		Suppose $\psi \in C^{1}\left(B_{2}\right)$ is positive and
		\begin{equation*}
		|\nabla \ln \psi|\leqslant C_{0} \quad {\rm in} \quad B_{3/2},
		\end{equation*}
		for some $C_{0}>0$. Then, there exists $0<r_{0}<1/2$, depending only on $n$ and $C_{0}>0$, such that for every $x\in B_{1}$ and $0<\mu \leqslant r_{0}$, it holds that $\psi_{x,\mu}(y) \leqslant \psi(y)$ for any $|y-x| \geqslant \mu$ and $y\in B_{3/2}$.
	\end{lemma}
	
	\begin{proof}
		For any $x \in B_{1}$, we have
		\begin{align}\label{derivative}
		\frac{\ud}{\ud r}\left(r^{\gamma(s)}\psi(x+r \theta)\right)
		&=r^{\gamma(s)-1} \psi(x+r \theta)\left(\gamma(s)-r \frac{\nabla \psi\cdot\theta}{\psi}\right)\geqslant r^{\gamma(s)-1} \psi(x+r \theta)\left(\gamma(s)-C_{0} r\right)>0.&
		\end{align}
		Hence, for $0<r<\bar{r}:=\min\left\{\frac{1}{2},\frac{\gamma(s)}{C_{0}}\right\}$ and $\theta\in \mathbb{S}^{n-1}$. For any $y \in B_{\bar{r}}(x)$ and $0<\mu<|y-x|\leqslant \bar{r}$, let us consider
		\begin{equation*}
		\theta=\frac{y-x}{|y-x|}, \quad r_{1}=|y-x| \quad \mbox{and} \quad r_{2}=\frac{\mu^{2}}{|y-x|^{2}}r_{1}.
		\end{equation*}
		Besides, from \eqref{derivative},  it follows that $r_{2}^{\gamma(s)} \psi\left(x+r_{2}\theta\right)<r_{1}^{\gamma(s)}\psi\left(x+r_{1} \theta\right)$, and so $\psi_{x, \mu}(y)\leqslant \psi(y)$, for $0<\mu<|y-x|\leqslant\bar{r}$.
		On the other hand, we get
		\begin{equation*}
		\psi_{x, \mu}(y)=\left(\frac{\mu}{|y-x|}\right)^{n-4} \psi\left(\mathcal{I}_{x, \mu}(y)\right) \leqslant\left(\frac{\mu}{\bar{r}}\right)^{n-4} \max_{B_{3/2}(x)} \psi\leqslant e^{\frac{3}{2} C_{0}}\left(\frac{\mu}{\bar{r}}\right)^{n-4} \inf_{B_{3 / 2}(x)}\psi\leqslant \psi(y),
		\end{equation*}
		for $|y-x| \geqslant \bar{r}$. Finally, choosing $\mu\leqslant r_{0}$ with $e^{\frac{3}{2} C_{0}}\left(\frac{r_{0}}{r}\right)^{n-4}\leqslant 1$, the proof of the lemma follows.
	\end{proof}
	
	The following lemma is the first step to apply the integral moving spheres method.
	
	\begin{lemma}\label{msestimate}
		Let $R=\infty$, $s\in(1,2^{**}-1]$, and $\mathcal{U}$ be a nonnegative non-singular solution to \eqref{integralsystem}. For any $x \in B_{1}$, $z \in B_{2}\setminus\left(\{0\}\cup B_{\mu}(x)\right)$ and $\mu<1$, it holds that $    u_i(z)-(u_i)_{x, \mu}(z)>0$ for $i\in I$.
	\end{lemma}
	
	\begin{proof}
		Let $\mathcal{U}$ be a nonnegative solution to \eqref{integralsystem}. By rescaling, we can replace $u_i(x)$ by $r^{\gamma(s)}u_i(r x)$ for $r={1}/{2},$ we may consider the equation defined in $B^*_{2}$ for convenience, that is,
		\begin{equation*}
		u_i(x)=\int_{B_{2}}|x-y|^{4-n}f^s_i(\mathcal{U}(y))\ud y+\psi_i(x) \quad \mbox{in} \quad B^*_{2}
		\end{equation*}
		such that $u_i\in C\left(B^*_{2}\right) \cap L^{s}\left(B_{2}\right)$ and $|\nabla \ln \psi_i| \leqslant C_{0} $ in $B_{3/2}$. Extending $u_i$ to be identically zero outside $B_{2}$, we have
		\begin{equation*}
		u_i(x)=\int_{\mathbb{R}^n} |x-y|^{4-n}f^s_i(\mathcal{U}(y)) \ud y+\psi_i(x) \quad \mbox{in} \quad B^*_{2}.
		\end{equation*}
		Using the identities in \cite[page 162]{MR2055032}, one has 
		\begin{equation*}
		\left(\frac{\mu}{|z-x|}\right)^{n-4}\int_{|y-x| \geqslant\mu} {\left|\mathcal{I}_{x,\mu}(z)-y\right|^{n-4}}{f^s_i(\mathcal{U}(y))} \ud y=\int_{|y-x| \leqslant \mu} \left|z-y\right|^{n-4}{f^s_i(\mathcal{U}(z))}\ud y
		\end{equation*}
		and
		\begin{equation*}
		\left(\frac{\mu}{|z-x|}\right)^{n-4} \int_{|y-x|\leqslant\mu} {\left|\mathcal{I}_{x,\mu}(z)-y\right|^{n-4}}{f^s_i(\mathcal{U}(y))} \ud y=\int_{|y-x|\geqslant \mu} \left|z-y\right|^{n-4}{f^s_i(\mathcal{U}(y))}\ud y,
		\end{equation*}
		which yields
		\begin{equation}\label{integralconformalinvariance}
		{(u_i)}_{x, \mu}(z)=\int_{\mathbb{R}^{n}} \left|z-y\right|^{n-4}{f^s_i(\mathcal{U}(y))}\ud y+{(\psi_i)}_{x,\mu}(z) \quad \mbox{for} \quad z\in \mathcal{I}_{x,\mu}(B_{2}),
		\end{equation}
		Consequently, for any $x \in B_{1}$ and $\mu<1$, we have that for $z\in B^*_{2}\cup B_{\mu}(x)$,
		\begin{equation*}
		u_i(z)-(u_i)_{x, \mu}(z)=\int_{|y-x|\geqslant\mu} E(x,y,\mu,z)\left[f^s_i(\mathcal{U})-f^s_i(\mathcal{U}_{x,\mu})\right]\ud y+\left[(\psi_i)_{x, \mu}(z)-\psi_i(z)\right],
		\end{equation*}
		where 
		\begin{equation}\label{kernelkelvintransform}
		E(x,y,z,\mu):={|z-y|^{4-n}}-\left(\frac{|z-x|}{\mu}\right)^{4-n} {\left|\mathcal{I}_{x,\mu}(z)-y\right|^{4-n}}
		\end{equation}
		is used to estimate the difference between a $\mathcal{U}$ and its Kelvin transform $\mathcal{U}_{x,\mu}$. Finally, using the its decay properties, it is straightforward to check that $E(x,y,z,\mu)>0$ for all $|z-x|>\mu>0$, which concludes the proof.
	\end{proof}
	
	\begin{remark}
		Notice that the same proof in \cite[Lemma~3.1]{MR2055032} (see also \cite[Lemma~2.1]{MR2001065}) would also adapt to this situation.
	\end{remark}
	
	Next, let us introduce the {\it critical sliding parameter} as the supremum for which an inequality relating a component function and its Kelvin transform is satisfied.
	
	\begin{definition}
		Given $x \in \mathbb{R}^{n}$,  for each $i\in I$, let us define
		\begin{equation}\label{criticalparameter}
		\mu_i^*(x)=\sup \left\{\mu>0:(u_{i})_{r,x}\leqslant u_{i} \ {\rm in} \ \mathbb{R}^{n}\setminus B_{r}(x) \ {\rm for \ any} \ 0<r<\mu\right\}.
		\end{equation}
		Since each $u_{i}$ is superharmonic for all $i\in I$, by using Lemma~\ref{lmlt:zhang1},  we get $\mu^*_{i}(x)>0$ for $i\in I$. Thus, we can define
		\begin{equation*}
		\mu^*(x)=\inf_{i\in I} \mu^*_{i}(x)>0.
		\end{equation*}
	\end{definition}
	
	The next lemma is essentially the moving spheres technique in its integral form.
	This method provides the exact form for any blow-up limit solution to \eqref{eq:subcriticalsystem}, which depends on whether the critical sliding parameter $\mu^*(x)$ is finite or infinite. 
	The main ingredient in the proof is the integral version of the moving spheres technique contained in \cite[Lemma~3.2]{MR2055032} (see also \cite[Proposition~3.2]{arxiv:1901.01678} and \cite[Proposition~59]{arXiv:2003.03487}).
	
	\begin{lemma}\label{lm:movingspheres}
		Let $R=\infty$, $s\in(1,2^{**}-1]$, and $\mathcal{U}$ be a nonnegative non-singular solution to \eqref{eq:subcriticalsystem}, $z\in\mathbb{R}^n$ and $\mu^*(z)>0$ given by \eqref{criticalparameter}. 
		Assume that 
		the origin is a removable singularity.
		The following holds:\\
		\noindent{\rm (i)} if $\mu^*(x)<\infty$ is finite, then 
		$\mathcal{U}_{x,\mu^*(x)}=\mathcal{U}$ in $\mathbb{R}^{n} \setminus\{x\}$.\\
		\noindent{\rm (ii)} if $\mu^*(x_0)=\infty$, for some $x_{0} \in \mathbb{R}^{n}$, then $\mu^*(x)=\infty$ for all $x \in \mathbb{R}^{n}$.	
	\end{lemma}
	
	\begin{proof}
		Without loss of generality, we may assume $x_0=0$. Let us fix $\mu^*=\mu^*(0)$ and $(u_i)_{\mu}=(u_i)_{0,\mu}$ for all $i\in I$. 
		By the definition of $\mu^*(x)$, we have that $(u_i)_{\mu^*}(x) \leqslant u_i(x)$, for all $|x| \geqslant \mu^*$.
		Thus, by \eqref{integralconformalinvariance}, with $x=0$ and $\mu=|x|\geqslant \mu^*$, and the positivity of the kernel $E(0,y,z,\mu)$ given by \eqref{kernelkelvintransform}, either $u_{\mu^*}(y)=u(y)$ for all $|x|\geqslant \mu^*$ or $(u_i)_{\mu^*}(y)<(u_i)(y)$ for all $|x|>\mu^*$. In the former case, the conclusion easily follows.
		In the sequel, we assume that the last condition holds.
		Hence, Proposition~\ref{lm:integralrepresentation} yields
		\begin{align*}
		\liminf_{|z|\rightarrow\infty}|z|^{n-4}\left[u_i(z)-(u_i)_{\mu^*}(z)\right]&=\liminf_{|z|\rightarrow \infty} \int_{|y| \geqslant \mu^*}|z|^{n-4} E(0,y,z,\mu^*)\left[f_i^s(\mathcal{U}(y))-f_i^s(\mathcal{U}_{\mu^*}(y))\right]\ud y&\\
		&\geqslant \int_{|y|\geqslant \mu^*}\left(1-\left(\frac{\mu^*}{|y|}\right)^{n-4}\right)\left[f_i^s(\mathcal{U}(y))-f_i^s(\mathcal{U}_{\mu^*}(y))\right]\ud y>0,&
		\end{align*}
		which implies that there exists $\varepsilon_{1} \in(0,1)$ satisfying $u_i(z)-(u_i)_{\mu^*}(z)\geqslant {\varepsilon_{1}}{|z|^{4-n}}$ for all $|z| \geqslant \mu^*+1$ and $i\in I$.
		Moreover, there exists $\varepsilon_{2}\in(0,\varepsilon_{1})$ such that,
		for $|z|\geqslant \mu^*+1 \ \mbox{and} \ \mu^* \leqslant \mu \leqslant \mu^*+\varepsilon_{2}$, we find 
		\begin{equation}\label{blowup1}
		\left(u_i-(u_i)_{\mu^*}\right)(z) \geqslant {\varepsilon_{1}}{|z|^{4-n}}+\left((u_i)_{\mu^*}-(u_i)_{\mu}\right)(z) \geqslant \frac{\varepsilon_{1}}{2}|z|^{4-n}.
		\end{equation}
		Whence, for any $\varepsilon \in\left(0, \varepsilon_{2}\right)$ (to be chosen later), $\mu^* \leqslant \mu \leqslant \mu^*+\varepsilon$, and $\mu \leqslant|y| \leqslant \mu^*+1$, we have
		\begin{align*}
		\left(u_i-(u_i)_{\mu^*}\right)(z)&=\int_{|y| \geqslant \mu} E(0,y,z,\mu)\left[f_i^s(\mathcal{U}(y))-f_i^s(\mathcal{U}_{\mu}(y))\right]\ud y\\
		&\geqslant \int_{\mu \leqslant|y| \leqslant \mu+1} E(0,y,z,\mu)\left[f_i^s(\mathcal{U}(y))-f_i^s(\mathcal{U}_{\mu}(y))\right]\ud y\\
		&+\int_{\mu^*+2 \leqslant|z| \leqslant \mu^*+3} E(0,y,z,\mu)\left[f_i^s(\mathcal{U}(y))-f_i^s(\mathcal{U}_{\mu}(y))\right]\ud y\\
		&\geqslant \int_{\mu^* \leqslant|y| \leqslant \mu^*+1} E(0,y,z,\mu)\left[f_i^s(\mathcal{U}_{\mu^*}(y))-f_i^s(\mathcal{U}_{\mu}(y))\right]\ud y\\
		&+\int_{\mu^*+2 \leqslant|y| \leqslant \mu^*+3} E(0,y,z,\mu)\left[f_i^s(\mathcal{U}(y))-f_i^s(\mathcal{U}_{\mu}(y))\right] \ud y.
		\end{align*}
		Now using \eqref{blowup1}, there exists $\delta_{1}>0$ such that $f_i^s(\mathcal{U}(y))-f_i^s(\mathcal{U}_{\mu}(y))\geqslant\delta_1$ for $\mu^*+2 \leqslant|y| \leqslant \mu^*+3$.
		Since $E(0, y, z, \mu)=0$ for all $|z|=\lambda$ and
		\begin{equation*}
		\nabla_{z} E(0, y, z, \mu)\cdot z\big|_{|z|=\mu}=(n-4)|z-y|^{6-n}\left(|z|^{2}-|y|^{2}\right)>0 \quad \mbox{for all} \quad \mu^*+2 \leqslant|y| \leqslant \mu^*+3,
		\end{equation*}
		where $\delta_{2}>0$ is a constant independent of $\varepsilon$. 
		Then, there exists $C>0$ such that, for $\mu^* \leqslant \mu \leqslant \mu^*+\varepsilon$, we get
		\begin{equation*}
		\left|f_i^s(\mathcal{U}(y))-f_i^s(\mathcal{U}_{\mu^*}(y))\right| \leqslant C(\mu-\mu^*) \leqslant C \varepsilon \quad \mbox{for all} \quad \mu^* \leqslant \mu \leqslant|y| \leqslant \mu^*+1.
		\end{equation*}
		Furthermore, recalling that $\mu \leqslant|z| \leqslant \mu^*+1$, we obtain
		\begin{align*}
		\int_{\mu \leqslant|y| \leqslant \mu^*} E(0, y, z, \mu) \ud y &\leqslant \left|\int_{\mu \leqslant|y| \leqslant \mu^*+1}\left[{|y-z|^{4-n}}-{\left|\mathcal{I}_{\mu}(z)-y\right|^{4-n}}+\left(\frac{\mu}{|z|}-1\right)^{n-4}\left|\mathcal{I}_{\mu}(z)-y\right|^{n-4}\right] \ud y\right|\\
		&\leqslant C\left|\mathcal{I}_{\mu}(z)-z\right|+C(|z|-\mu) \leqslant C(|z|-\mu),
		\end{align*}
		which, provides that for small $0<\varepsilon\ll1$, $\mu^* \leqslant \mu \leqslant \mu^*+\varepsilon$, and $\mu\leqslant|z| \leqslant \mu^*+1$, it follows
		\begin{align*}
		\left(u_i-(u_i)_{\mu^*}\right)(z)& 
		\geqslant-C \varepsilon \int_{\mu \leqslant|y| \leqslant \mu^*+1} E(0, y, z, \mu)\ud y+\delta_{1}\delta_{2}(|z|-\mu) \int_{\mu^*+2 \leqslant|z| \leqslant \mu^*+3} \ud y&\\
		&\geqslant\left(\delta_{1}\delta_{2} \int_{\mu^*+2 \leqslant|y| \leqslant \mu^*+3}\ud z-C\varepsilon\right)(|z|-\mu)&\\
		&\geqslant 0.&
		\end{align*}
		This is a contradiction to the definition of $\mu^*>0$. Therefore, the first part of the lemma is established.
		
		Next, by the definition of $\mu^*(x)$, we know that $|\mathcal{U}_{x, \mu}(z)|\leqslant|\mathcal{U}(z)|$ for all $0<\mu<\mu^*(x)$, $|z-x|\geqslant\mu$; thus, multiplying it by $|z|^{n-4}$, and taking the limit as $|z|\rightarrow\infty$, yields
		\begin{equation}\label{blowup2}
		l=\liminf_{|z| \rightarrow \infty}|z|^{n-4}|\mathcal{U}(z)|\geqslant \mu^{n-4}|\mathcal{U}(z)| \quad \mbox{for all} \quad 0<\mu<\mu^*(x).
		\end{equation}
		On the other hand, if $\mu^*(x_0)<\infty$, multiplying the identity obtained in (i) by $|z|^{n-4}$ and passing to the limit when $|z|\rightarrow\infty$, we obtain
		\begin{equation}\label{blowup3}
		l=\lim_{|z| \rightarrow \infty}|z|^{n-4} |\mathcal{U}(z)|=\mu^*(x_0)^{n-4} |\mathcal{U}(x_0)|<\infty.
		\end{equation}
		Finally, by \eqref{blowup2} and \eqref{blowup3}, if there exists $x_0 \in \mathbb{R}^{n}$ such that $\mu^*(x_0)<\infty$, then $\mu^*(x)<\infty$ for all $x \in \mathbb{R}^{n}$.
	\end{proof}

Eventually, we are ready to classify the non-singular global solutions using a standard blow-up technique.
Let us also remark that Appendix~\ref{app:alternativamovingspheres} contains an alternative proof of this fact, which directly uses the superharmonicity lemma.

\begin{proof}[Proof of Proposition~\ref{prop:nonexistencenonsingular}]
Using Lemma~\ref{lm:movingspheres}, we get that either $\mu^*(z)<\infty$ or $\mu^*(z)=\infty$ for all $z \in \mathbb{R}^{n}$. 
If $\mu^*(z)<\infty$ for all $z \in \mathbb{R}^{n}$, then $\mathcal{U}_{z,\mu^*(z)}=\mathcal{U}$ in $\mathbb{R}^{n} \setminus\{z\}$ for every $z \in \mathbb{R}^{n}$. 
Hence, we can use \cite[Lemma~11.1]{MR2001065} to conclude that there exist $a_{i} \geqslant 0$, $\mu_{i}>0$ and $z_{i} \in \mathbb{R}^{n}$ for $i\in I$ such that
\begin{equation}\label{classification1}
u_{i}(x)=a_i \mu_{i}^{-\frac{n-4}{2}}\left(\frac{\mu_{i}}{\mu_{i}^{2}+\left|x-z_{i}\right|^{2}}\right)^{\frac{n-4}{2}}.
\end{equation}
Otherwise, if $\mu^*(z)=\infty$ for all $z\in\mathbb{R}^{n}$, then \eqref{criticalparameter} holds for any $\mu>0$ and $z \in \mathbb{R}^{n}$. 
Thus, due to \cite[Lemma~11.2]{MR2001065}, there exist $\lambda_{i} \geqslant 0$ for $i\in I$ such that 
\begin{equation}\label{classification2}
u_{i}(x)=\lambda_{i}.
\end{equation}

Moreover, suppose that $\mathcal{U}$ satisfies \eqref{classification2}, that is, $\mathcal{U}$ is constant everywhere on $\mathbb{R}^{n}$. 
Since $\mathcal{U}$ is a nonnegative solution to \eqref{eq:subcriticalsystem} with $R=\infty$, it follows that $\mathcal{U}$ must be equivalently zero, which proves the proposition.

Finally, for $s=2^{**}-1$, if $u_{i}$ satisfies \eqref{classification1}, then the proof is equal to the one in \cite{arXiv:2002.12491}.
\end{proof}

\begin{remark}
Another approach to prove the last lemma would be to use the equivalence between the PDE system \eqref{eq:subcriticalsystem} and the IE system \eqref{integralsystem}.
In this fashion, the scaling moving spheres technique, recently developed in \cite[Section~2.2]{arXiv:1810.02752}, could provide the desired radial symmetry results.
In this fashion, instead of using the classical maximum principle, one needs a narrow region principle.
\end{remark}

\subsection{Radial symmetry of the blow-up limit solutions}\label{subsec:singularcase}
In contrast with the last subsection, we assume that the origin is a non-removable singularity. In this case, a more delicate analysis is required, namely, we need to divide the proof into two cases: $s\in(1,2_{**}]$ and $s\in(2_{**},2^{**}]$. The exponent restriction is related to the removability of isolated singularities to nonnegative solution to \eqref{eq:subcriticalsystem}.
When $p=1$, the same symmetry result was proved in \cite{MR1611691}  by using a classical moving planes method, whose integral form could also be used in our situation.

\begin{proposition}\label{prop:asymptoticsymmetrylimit}
Let $R=\infty$, $s\in(1,2^{**}-1)$, and $\mathcal{U}$ be a nonnegative singular solution to \eqref{eq:subcriticalsystem}. Then, $|\mathcal{U}|$ is radially symmetric about the origin.
\end{proposition}

The proof is also based on the integral form of the moving spheres method performed in the previous section, and relies on the following result from \cite[Lemma~2.1]{MR2001065}.

\begin{lemmaletter}
If $u \in C^{1}\left(\mathbb{R}^{n}\setminus\{0\}\right)$ satisfies, for each $z \in \mathbb{R}^{n} \backslash\{0\}$ and for any $0<\mu<|z|$  that $u_{z, \mu} \leqslant u$ in $\mathbb{R}^{n} \setminus\left(B_{\mu}(z) \cup\{0\}\right)$,	then $u$ is radially symmetric about the origin.
\end{lemmaletter}

Before, we give the proof of the rotational symmetry for singular solutions to \eqref{eq:subcriticalsystem}, we need to establish a set of lemmas. The first one concerns the superharmonicity of the component solutions to \eqref{eq:subcriticalsystem}.
Here the arguments are similar in spirit to the ones in \cite[Theorem~3.7]{MR1769247}.
	
	\begin{lemma}\label{lm:superharmonicity}
		Let $R=\infty$, $s\in(1,2^{**}-1)$, and $\mathcal{U}$ be a nonnegative singular solution to \eqref{eq:subcriticalsystem}. Then, $-\Delta \mathcal{U}$ is a superharmonic $p$-map in the distributional sense, that is, for all nonnegative $\Phi\in C^{\infty}_c(\mathbb{R}^{n},\mathbb{R}^{p})$, one has
		\begin{equation*}
		\int_{\mathbb{R}^{n}} \langle\Delta \mathcal{U}, \Delta \Phi\rangle\ud x\geqslant 0.
		\end{equation*}
		Moreover, $-\Delta \mathcal{U}\geqslant0$ in $\mathbb{R}^{n} \setminus\{0\}$.
	\end{lemma}
	
	\begin{proof}
		Proceeding similarly to Lemma~\ref{integrability}, one can prove that $|\mathcal{U}| \in L_{\rm loc}^{s}(\mathbb{R}^{n})$. 
		Let $\eta_{\varepsilon} \in C^{\infty}(\mathbb{R}^{n})$ be the cut-off function given by \eqref{cutoff} and $\Phi \in C_{c}^{\infty}(\mathbb{R}^{n},\mathbb{R}^p)$ be a nonnegative test $p$-map. 
		Then, multiplying \eqref{eq:subcriticalsystem} by $\eta_{\varepsilon} \phi_i$ for each $i\in I$, and integrating by parts twice, we get
		\begin{align*}
		0 & \leqslant \int_{\mathbb{R}^{n}} \eta_{\varepsilon} \phi_i |\mathcal{U}|^{s-1}u_i\ud x \\
		&=\int_{\mathbb{R}^{n}} \Delta\left(\eta_{\varepsilon}\phi_i\right) \Delta u_i\ud x \\
		&=\int_{\mathbb{R}^{n}} \Delta u_i\left(\Delta \phi_i \eta_{\varepsilon}+\varsigma^{\varepsilon}_i\right)\ud x,
		\end{align*}
		where $\varsigma^{\varepsilon}_i:=2 \langle\nabla \phi_i, \nabla \eta_{\varepsilon}\rangle+\phi_i\Delta \eta_{\varepsilon}$.
		Notice that  $\varsigma_i^{\varepsilon}(x) \equiv 0$ when $|x| \leqslant \varepsilon$ or $|x| \geqslant 2 \varepsilon$, and $|\Delta \varsigma^{\varepsilon}_i(x)| \leqslant C \varepsilon^{-4}$, for some $C>0$. 
		In addition, since $n-4-n/s>0$, the following estimate holds,
		\begin{align*}
		\left|\int_{\mathbb{R}^{n}} \Delta u_i \varsigma^{\varepsilon}_i\ud x\right| & \leqslant \int_{\mathbb{R}^{n}} u_i|\Delta \varsigma^{\varepsilon}_i|\ud x\\
		& \leqslant C \varepsilon^{-4}\left(\int_{\{\varepsilon \leqslant|x| \leqslant 2 \varepsilon\}} |\mathcal{U}|^{s-1}u_i\ud x\right)^{1/s}\varepsilon^{n(1-1/s)} \\
		&\leqslant C \varepsilon^{-4}\left(\int_{\{\varepsilon \leqslant|x| \leqslant 2 \varepsilon\}} |\mathcal{U}|^{s}\ud x\right)^{1/s}\varepsilon^{n(1-1/s)} \\
		& \leqslant C \varepsilon^{n-4-n/s} \rightarrow 0 \quad \mbox{as} \quad \varepsilon \rightarrow 0,
		\end{align*}
		which implies
		\begin{align*}
			\int_{\mathbb{R}^{n}} \Delta u_i \Delta \phi_i\ud x =\lim_{\varepsilon \rightarrow 0}\int_{\mathbb{R}^{n}}\left(\Delta u_i \Delta \varsigma_i^{\varepsilon}\ud x+\eta_\varepsilon\Delta\phi_i \Delta u_i\right)\ud x=\int_{\mathbb{R}^{n}}\phi_i |\mathcal{U}|^{s-1}u_i\ud x\geqslant 0.
		\end{align*}
		Thus, $-\Delta u_i$ is superharmonic in the whole space $\mathbb{R}^{n}$ in the distributional sense for all $i\in I$, which proves the first part of the proof.
		
		To prove the second statement, given $0<\varepsilon\ll1$, let us consider $\widetilde{u}_i^{\varepsilon}:=-\Delta u_i+\varepsilon$.
		Using Lemma~\ref{integrability}, there exists a constant $C>0$, depending only on $n$ and $s$, such that for all $|x| \geqslant 4$, it holds
		\begin{equation*}
		\sum_{j=0}^3|x|^{\gamma(s)+j}\left|D^{(j)} u_i(x)\right| \leqslant C,
		\end{equation*}
		which yields that  $\lim _{|x| \rightarrow \infty}|\Delta u_i(x)|=0$ for all $i\in I$.
		Whence, for any $0<\varepsilon\ll1$, there exists $R_{\varepsilon}\gg1$ such that $\widetilde{u}_i^\varepsilon>{\varepsilon}/{2}$ for $|x| \geqslant R_{\varepsilon}$.
		Finally, using that $\widetilde{u}_i^{\varepsilon}$ is superharmonic in $\mathbb{R}^{n}$ in the distributional sense, we have that
		$\widetilde{u}_i^{\varepsilon} \geqslant 0$ in $\mathbb{R}^{n}\setminus\{0\}$, which, by passing to the limit as $\varepsilon\rightarrow 0$, provides $-\Delta u_i \geqslant 0$ in $\mathbb{R}^{n}\setminus\{0\}$.
		The last inequality concludes the proof of the lemma.
	\end{proof}
	
	We prove the main result of this subsection, namely, the radial symmetry of the blow-up limit solutions. 
	To this end, we use an asymptotic version of moving spheres techniques used in the proof of Lemma~\ref{lm:movingspheres} with minor modifications inspired by  \cite[Proposition~2.1]{MR2479025}.
	
	\begin{proof}[Proof of Proposition~\ref{prop:asymptoticsymmetrylimit}]
		Fixing $x \in \mathbb{R}^{n} \backslash\{0\}$ arbitrary, there exists $0<r_{0}<|x|$ such that for any $0<\mu \leqslant \mu_{0}$, $r\in(0,r_0)$ and $i\in I$, it follows
		$\left(u_{i}\right)_{x, \mu}^{*} \leqslant u_{i}$ in $\mathbb{R}^{n}\setminus\left(B_{r}(x) \cup\{0\}\right)$.
		Hence, similarly as before, let us define
		\begin{equation*}
		\widetilde{\mu}_{i}(x)=\sup \left\{\mu>0:\left(u_{i}\right)_{x,r} \leqslant u_{i} \ \mbox{ in } \ \mathbb{R}^{n} \setminus\left(B_{r}(x)\cup\{0\}\right) \ \mbox{ for any } \ 0<r<\mu\right\}
		\end{equation*}
		and
		\begin{equation*}
		\widetilde{\mu}^*(x)=\inf_{i\in I}\widetilde{\mu}_{i}(x).
		\end{equation*}
		
		Let us divide the rest of the argument into two steps.
		
		\noindent{\bf Step 1:} $0<\widetilde{\mu}^*(z) \leqslant|z|$.
		
		\noindent Indeed, it is straightforward to see that $\widetilde{\mu}^*(x)>0$.
		Next, since the origin is a non-removable singularity, there exist $\{x_{k}\}_{k\in\mathbb{N}}\subset\mathbb{R}^n$ such that $x_k\rightarrow 0$ as $k\rightarrow\infty$ and $i\in I$ satisfying $u_{i}\left(x_{k}\right) \rightarrow \infty$ as $k\rightarrow\infty$.
		Now suppose by contradiction that $\widetilde{\mu}^*(x)>|x|$, then there exist $r>|x|$ such that
		\begin{equation}\label{asymptoticsymmetry2}
		\left(u_{i}\right)_{x, r} \leqslant u_{i} \quad \mbox{in} \quad \mathbb{R}^{n} \setminus B_{r}(x).
		\end{equation}
		Defining $y_{k}:=\mathcal{I}_{z,\mu}(x_k)$ for all $k\in\mathbb{N}$ to be the reflection of $x_{k}$ about $\partial B_{\mu}(z)$ and using that $x_{k} \rightarrow 0$, we find that $y_{k} \in \mathbb{R}^{n} \setminus B_{r}(x)$ for sufficiently large $k\gg1$.
		Moreover, $y_{k} \rightarrow y_{0}$ as  $k\rightarrow\infty$, where $y_{0}:=[1-({r}^2{|x|}^{-2})]x$.
		Whence, choosing $r>0$ satisfying $0<|r-|x||\ll1$, we have $y_{0} \neq 0$, which implies that $u_{i}$ is smooth at $y_{0}$. On the other hand, we get 
		\begin{equation*}
		u_{i}\left(y_{0}\right)=\lim_{k \rightarrow \infty} u_{i}\left(y_{k}\right) \geqslant \lim _{k \rightarrow \infty}\left(u_{i}\right)_{x,r}\left(y_{k}\right) \geqslant\left(\frac{|x|}{r}\right)^{n-4} \lim _{k\rightarrow \infty} u_{k}\left(x_{k}\right)=\infty.
		\end{equation*}
		This is a contradiction, and the first step is proved.
		
		\noindent{\bf Step 2:} $\widetilde{\mu}^*(x)=|x|$.
		
		\noindent Suppose by contradiction that
		$\widetilde{\mu}^*(x)<|x|$. Thus, since the origin is a non-removable singularity, together with Lemma~\ref{lm:strongsuperharmonicity}, and applying the classical maximum principle twice, we find that there exists $i\in I$ satisfying
		\begin{equation}\label{asymptoticsymmetry1}
		u_{i}>\left(u_{i}\right)_{x, \widetilde{\mu}^*(x)}^{*} \quad \mbox{ in } \quad \mathbb{R}^{n} \setminus\left(\bar{B}_{\widetilde{\mu}^*(x)}(x) \cup\{0\}\right),
		\end{equation}
		which yields $|\mathcal{U}|>|\mathcal{U}_{x, \widetilde{\mu}^*(x)}^{*}|$ in $\mathbb{R}^{n} \setminus\left(\bar{B}_{\widetilde{\mu}^*(x)}(x) \cup\{0\}\right)$.
		Hence, the maximum principle again implies that the strict inequality in \eqref{asymptoticsymmetry1} is valid for all nontrivial components. 
		Whence, as in Lemma~\ref{lmlt:zhang1}, one can find $0<\varepsilon\ll1$ such that \eqref{asymptoticsymmetry1} holds for all $i\in I$, where $\widetilde{\mu}^*(x)$ is replaced by some $\widetilde{\mu}^*(x)<r<\widetilde{\mu}^*(x)+\varepsilon$, which contradicts \eqref{asymptoticsymmetry2}; the proof of this step is finished.
		
		Finally, using Steps 1 and 2, one can prove that for any $z \in \mathbb{R}^{n} \setminus\{0\}$, $0<\mu<|x|$, and $i\in I$, it follows $\left(u_{i}\right)_{x,\mu}\leqslant u_{i}$ in $\mathbb{R}^{n} \setminus(B_{r}(x) \cup\{0\})$.
		Therefore, we can apply \cite[Lemma~2.1]{MR2479025} to conclude that $u_{i}$ is radially symmetric about the origin for any $i\in I$, which proves the proposition.
	\end{proof}
	
	\subsection{Singular case: The limiting Pohozaev levels}
	After proving the radial symmetry of singular solutions to \eqref{eq:subcriticalsystem}, we shall classify them in the blow-up and shrink-down limit.
	The idea is to use a blow-up/shrink-down analysis, which comes from tangent cone techniques from minimal hypersurface theory, and will be described in the sequel.
	
	For any $\mathcal{U}$ solution to \eqref{eq:subcriticalsystem} and $\lambda>0$, let us define the following $\lambda$-rescaling solution given by
	\begin{equation}\label{scalingfamily}
	\widehat{\mathcal{U}}_{\lambda}(x):=\lambda^{\gamma(s)}\mathcal{U}(\lambda x).
	\end{equation}
	Notice that the $\lambda$-rescaled solution is still a nonnegative solution to \eqref{eq:subcriticalsystem} with $R=\lambda^{-1}$.
	Besides, by a blow-up (resp. shrink-down) solution $\mathcal{U}_0$ (resp. $\mathcal{U}_{\infty}$) to \eqref{eq:subcriticalsystem}, we mean the limit $\mathcal{U}_{0}:=\lim_{\lambda\rightarrow0}\widehat{\mathcal{U}}_{\lambda}$ (resp. $\mathcal{U}_{\infty}:=\lim_{\lambda\rightarrow\infty}\widehat{\mathcal{U}}_{\lambda}$).
	In fact, utilizing some a priori estimates and the compactness of the family $\{\widehat{\mathcal{U}}_{\lambda}\}_{\lambda>0}\subset C^{4,\zeta}_{\rm loc}(\mathbb{R}^n,\mathbb{R}^p)$ is compact, for some $\zeta\in(0,1)$, these limits will be proven to exist.
	Next, we study the limit Pohozaev functional both as $r\rightarrow0$ (blow-up) and $r\rightarrow\infty$ (shrink-down), this will give the desired information about the asymptotic behavior for solutions to \eqref{eq:subcriticalsystem}.
	Here is our main result of this subsection:
	
	\begin{proposition}\label{prop:limitlocalbehavior}
		Let $R=\infty$ and $\mathcal{U}$ be a nonnegative singular solution to \eqref{eq:subcriticalsystem}.
		\begin{itemize}
			\item[{\rm (a)}] If $s\in(1,2_{**}]$, then $\mathcal{U}\equiv0$; 
			\item[{\rm (b)}] If $s\in(2_{**},2^{**}-1)$, then $\mathcal{P}_{\rm cyl}(r,\mathcal{U})$ converges both as $r\rightarrow0$ and $r\rightarrow\infty$, namely
			\begin{equation}\label{limintingconstant}
			\{\mathcal{P}_{\rm sph}(0, \mathcal{U}),\mathcal{P}_{\rm sph}(\infty, \mathcal{U})\}=\{-l^*(n,s),0\}, \quad \mbox{where} \quad l^*(n,s)=\frac{s-1}{2(s+1)}K_0(n,s)^{\frac{s+1}{s-1}}.
			\end{equation} 
			Moreover, there exists $\Lambda^*\in\mathbb{S}^{p-1}_{+,*}$ such that 
			\begin{equation*}
			\mathcal{U}(x)=\Lambda^* K_0(n,s)^{\frac{1}{s-1}}|x|^{-\gamma(s)},
			\end{equation*}
			where $\gamma(s)=\frac{4}{s-1}$ is the Fowler scaling exponent.
		\end{itemize}    
	\end{proposition}
	
	The following lemma provides an upper bound estimate for singular solutions to \eqref{eq:subcriticalsystem}.
	
	\begin{lemma}\label{lm:limituppperestiaamte}
		Let $R=\infty$, $s\in(1,2^{**}-1)$, and $\mathcal{U}$ be a nonnegative singular solution to \eqref{eq:subcriticalsystem}. Then,
		\begin{equation*}
		|\mathcal{U}(x)| \leqslant \left(\frac{s-1}{2n}\right)^{-\frac{1}{s-1}}|x|^{-\gamma(s)} \quad {\rm in}  \quad \mathbb{R}^{n} \setminus\{0\}.
		\end{equation*}
	\end{lemma}
	
	\begin{proof}
		Let $i\in I_+$, that is, $u_i>0$ in $\mathbb{R}^n\setminus\{0\}$. By Lemma~\ref{lm:superharmonicity}, we know that $u_{i}$ is superharmonic in $\mathbb{R}^{n} \setminus\{0\}$, which, by using the extended maximum principle \cite[Theorem~1]{MR1814364}, gives us
		\begin{equation*}\label{limitlevel1}
		\liminf _{x \rightarrow 0} u_{i}(x)>0.
		\end{equation*}
		Considering $\varphi_i=u_{i}^{1-s}$, a direct computation, provides
		\begin{equation*}
		\Delta \varphi_i \geqslant \frac{s}{s-1} \frac{|\nabla \varphi_i|^{2}}{\varphi_i}+s-1 \quad \mbox{in} \quad \mathbb{R}^{n} \setminus\{0\}.
		\end{equation*}
		Thus, for any $r>0$, let us consider 
		the auxiliary function $\widetilde{\varphi}_i(x)=\varphi_i(x)-\frac{s-1}{2n}|x|^{2}$, which is
		subharmonic in $B^*_{r}$.
		Furthermore, \eqref{limitlevel1} implies that $\widetilde{\varphi}_i$ is bounded close to the origin, and thus, again, by the extended maximum principle, we find
		\begin{equation*}
		0 \leqslant\limsup_{x \rightarrow 0} \widetilde{\varphi}_i(x) \leqslant \sup_{\partial B_{r}} \widetilde{\varphi}_i=\sup_{\partial B_{r}} \varphi_i-\frac{s-1}{2n} r^{2},
		\end{equation*}
		which yields
		\begin{equation*}
		\inf_{\partial B_{r}} u_{i} \leqslant\left(\frac{s-1}{2 n}\right)^{-\frac{1}{s-1}} r^{-\gamma(s)}.
		\end{equation*}
		Finally, a direct application of Proposition~\ref{prop:asymptoticsymmetrylimit} finishes this proof.
	\end{proof}
	
	As a consequence of this uniform upper bound, we prove the compactness of the family $\{\widehat{\mathcal{U}}_{\lambda}\}_{\lambda>0}\subset C^{4,\zeta}_{\rm loc}(\mathbb{R}^n,\mathbb{R}^p)$, for some $\zeta\in(0,1)$, which provides the existence of both blow-up and shrink-down limits for the scaling family defined by \eqref{scalingfamily}.
	
	\begin{lemma}\label{lm:limitcompactness}
		Let $R=\infty$, $s\in(1,2^{**}-1)$, and $\mathcal{U}$ be a nonnegative singular solution to \eqref{eq:subcriticalsystem}. Then, $\{\widehat{\mathcal{U}}_{\lambda}\}_{\lambda>0}\subset C^{4,\zeta}_{\rm loc}(\mathbb{R}^n,\mathbb{R}^p)$ is uniformly bounded, for some $\zeta\in(0,1)$.
	\end{lemma}
	
	\begin{proof}
		If the origin is a removable singularity of $\widehat{\mathcal{U}}_{\lambda}$ for all $\lambda>0$, according to Theorem~\ref{Thm3.1:classification} (i), $\mathcal{U}$ is trivial and the conclusion follows.
		
		On the other hand, assuming that the origin is a removable singularity, Lemma~\ref{lm:limituppperestiaamte} provides that $\{\widehat{\mathcal{U}}_{\lambda}\}_{\lambda>0}$ is globally bounded in $\mathbb{R}^{n}\setminus\{0\}$. Thus, we know that $\{\widehat{\mathcal{U}}_{\lambda}\}_{\lambda>0}$ is uniformly bounded in each compact subset of $K\subset\mathbb{R}^{n} \setminus\{0\}$.
		Moreover, since for each $\lambda>0$, the scaling $\widehat{\mathcal{U}}_{\lambda}$ also satisfies \eqref{eq:subcriticalsystem} with $R=\infty$, it follows from standard elliptic estimates that $\{\widehat{\mathcal{U}}_{\lambda}\}_{\lambda>0}$ is uniformly bounded in $C^{4,\zeta}(K,\mathbb{R}^{p})$, for some $\zeta\in(0,1)$, which concludes the proof.
	\end{proof}
	
	Recall that $\mathcal{P}_{\rm sph}(r, \mathcal{U})$ is the Pohozaev functional introduced in \eqref{pohozaevsphericalautonomous}, which by Proposition~\ref{lm:monotonicityformula} is monotonically nonincreasing in $r>0$ when $s\in(1,2^{**}-1)$.
	
	\begin{lemma}\label{lm:limitlimitinglevels}
		Let $R=\infty$, $s\in(1,2^{**}-1)$, and $\mathcal{U}_{0}$ $($or $\mathcal{U}_{\infty}$$)$ be a blow-up $($or shrink-down$)$ solution under the scaling $\{\widehat{\mathcal{U}}_{\lambda}\}_{\lambda>0}$. 
		Then,
		$\mathcal{P}_{\rm sph}(r, \mathcal{U}_{0},s)\equiv\mathcal{P}_{\rm sph}(0, \mathcal{U},s)$ $($or $\mathcal{P}_{\rm sph}(r, \mathcal{U}_{\infty},s)\equiv\mathcal{P}_{\rm sph}(\infty, \mathcal{U},s)$$)$ is constant
		for all $r>0$. In particular, both $\mathcal{U}_{0}$ and $\mathcal{U}_{\infty}$ are homogeneous of degree
		$\gamma(s)$.
	\end{lemma}
	
	\begin{proof}
		Let $\{\lambda_{k}\}_{k\in\mathbb{N}}\subset(0,\infty)$  be a blow-up sequence such that $\lambda_{k}\rightarrow 0$, and  $\mathcal{U}_{0}\in{C^{4,\zeta}(\mathbb{R}^n\setminus\{0\})}$ be its blow-up limit, that is, $\widehat{\mathcal{U}}_{\lambda_k}\rightarrow\mathcal{U}_{0}$ in ${C_{\rm loc}^{4,\zeta}(\mathbb{R}^n\setminus\{0\})}$ as $k\rightarrow\infty$, for some $\zeta\in(0,1)$.
		Now, using Proposition~\ref{lm:monotonicityformula} and Lemma~\ref{lm:limitcompactness}, there exists the limiting level $\mathcal{P}_{\rm sph}(0, \mathcal{U},s)$.
		Moreover, due to the scaling invariance of the Pohozaev functional in \eqref{scaleinvariance}, for any $r>0$, it follows
		\begin{equation*}
		\mathcal{P}_{\rm sph}(r, \mathcal{U}_{0},s)=\lim_{k \rightarrow \infty}\mathcal{P}_{\rm sph}(r, \widehat{\mathcal{U}}_{\lambda_k},s)=\lim_{k \rightarrow \infty} \mathcal{P}_{\rm sph}(r\lambda_{k}, \mathcal{U},s)=\mathcal{P}_{\rm sph}(0,\mathcal{U},s),
		\end{equation*}
		which finishes the proof of the first assertion.
		Now, we can check that the homogeneity follows from Lemma~\ref{lm:homogeneity}.
		Finally, notice that the same argument can readily be employed, replacing the blow-up limit by the shrink-down limit, so we omit it here.
	\end{proof}
	
	\begin{lemma}\label{lm:blowupclassification}
		Let $R=\infty$, $s\in(1,2^{**}-1)$, and $\mathcal{U}$ be a nonnegative singular solution to \eqref{eq:subcriticalsystem}. Assume that $|\mathcal{U}|$ is homogeneous of degree $\gamma(s)>0$.\\
		\noindent{\rm (a)} If $s\in(1,2_{**}]$, then $|\mathcal{U}|\equiv0$.\\
		\noindent{\rm (b)} If $s\in(2_{**},2^{**}-1)$, then either $|\mathcal{U}|\equiv0$, or $|\mathcal{U}|\equiv K_0(n,s)^{\frac{1}{s-1}}|x|^{-\gamma(s)}$.
	\end{lemma} 
	
	\begin{proof}
		Since $\mathcal{U}$ is homogeneous of degree $\gamma(s)$, the cylindrical transform $\mathcal{V}=\mathfrak{F}(\mathcal{U})$ given by \eqref{cyltransform} satisfies
		\begin{equation}\label{angularsystem}
		-\Delta_{\theta}^{2}v_i-J_{0}(n,s)\Delta_{\theta}v_i-K_0(n,s)v_i+|\mathcal{V}|^{s-1}v_i=0 \quad \mbox{on} \quad \mathbb{S}_t^{n-1} \quad \mbox{for} \quad i\in I,
		\end{equation}
		where $\Delta_{\theta}$ is the Laplace--Beltrami operator on $\mathbb{S}_t^{n-1}$.
		Now we divide the the proof into two cases:
		
		\noindent{\bf Case 1:} $s\in(1,2_{**}]$.
		
		\noindent Initially, since $K_0(n,s)\leqslant 0$ for all $s\in(1,2^{**}-1)$, we get 
		$\mathscr{L}_{\theta}(\mathcal{V})\leqslant 0$ on $\mathbb{S}_t^{n-1}$, where $\mathscr{L}_{\theta}(\mathcal{V}):=-\Delta_{\theta}^{2}v_i-J_{0}(n,s)\Delta_{\theta}v_i$.
		Next, observe that $\mathscr{L}_{\theta}$ is the composition of two elliptic operators, which, together with Lemma~\ref{lm:superharmonicity}, implies that $v_{i}$ does not attain any strict local minimum on $\mathbb{S}_t^{n-1}$. 
		Therefore, since $\mathbb{S}_t^{n-1}$ is a compact manifold, it follows that $v_{i}$ is constant for all $i\in I$, which yields $\mathcal{V}\equiv\Lambda$ for some $\Lambda\in\mathbb{S}_+^{p-1}$.
		Nevertheless, using that $K_0(n,s) \leqslant 0$, any nonnegative constant solution to \eqref{angularsystem} is trivial.
		By using the inverse of the cylindrical transformation, it holds that $\mathcal{U}$ is trivial on $\partial B_{1}$, which by the superharmonicity of each component, implies that $\mathcal{U}$ is trivial in the whole domain.
		This conclusion finishes the proof of the first case, and so part (a) of the lemma follows.
		
		\noindent{\bf Case 2:} $s\in(2_{**},2^{**}-1)$.
		
		\noindent Assume that $\mathcal{U}$ is a nontrivial limit solution in the punctured space. 
		Hence, since each component of $\mathcal{U}$ is nonnegative and superharmonic, it quickly follows that $|\mathcal{U}|>0$ in $\mathbb{R}^n\setminus\{0\}$. 
		By homogeneity, the origin is a non-removable singularity of $\mathcal{U}$. 
		Hence, by Proposition~\ref{prop:asymptoticsymmetrylimit}, $\mathcal{U}$ is radially symmetric; thus, $\mathcal{U}\equiv\Lambda$ is a nonnegative constant vector $\Lambda\in\mathbb{S}^{p-1}_{+}$.
		Moreover, by \eqref{angularsystem}, it holds
		\begin{equation*}
		|\Lambda|=K_0(n,s)^{\frac{1}{s-1}},
		\end{equation*}
		which, by using the homogeneity of $\mathcal{U}$ and Lemma~\ref{lm:homogeneity}, finishes the proof of the second case, and so (b) holds.
	\end{proof}
	
	At last, we can prove the main result of this part.
	
	\begin{proof}[Proof of Proposition~\ref{prop:limitlocalbehavior}]
		Let $\mathcal{U}_{0}$ and $\mathcal{U}_{\infty}$ be, respectively, a blow-up and a shrink-down limit of $\mathcal{U}$.
		According to Lemma~\ref{lm:limitinglevels} both $\mathcal{U}_{0}$ and $\mathcal{U}_{\infty}$ are homogeneous of degree $\gamma(s)$.
		In what follows, we divide the rest of the proof into two cases:
		
		\noindent{\bf Case 1:} $s\in(1,2_{**}]$. 
		
		\noindent Here, it follows from Lemma~\ref{lm:blowupclassification} (a) that both $\mathcal{U}_{0}$ and $\mathcal{U}_{\infty}$ are trivial, which, by Lemma~\ref{lm:limitlimitinglevels}, provides $\mathcal{P}_{\rm sph}(0,\mathcal{U})=\mathcal{P}_{\rm sph}(\infty, \mathcal{U})=0$. 
		In addition, using the monotonicity property of the Pohozaev functional, we find $\mathcal{P}_{\rm sph}(r, \mathcal{U})=0$ for all $r>0$. Hence, by Lemma \ref{lm:homogeneity}, $\mathcal{U}$ is homogeneous of degree $\gamma(s)$.
		Therefore, the proof of (a) of Proposition~\ref{prop:limitlocalbehavior} is now an immediate consequence of Lemma~\ref{lm:blowupclassification} (a).
		
		\noindent{\bf Case 2:} $s\in(2_{**},2^{**}-1)$.
		
		\noindent Initially, by Lemmas~\ref{lm:limitlimitinglevels} and \ref{lm:blowupclassification} (b), any blow-up $\mathcal{U}_{0}$ is either trivial or has the form \eqref{gidassprucklimitsolution}.
		If $\mathcal{U}_{0}$ is trivial, then clearly $\mathcal{P}_{\rm sph}\left(r, \mathcal{U}_{0},s\right)=0$ for all $r>0$, which combined with Lemma~\ref{lm:limitlimitinglevels} implies that $\mathcal{P}_{\rm sph}(0, \mathcal{U},s)=0$. 
		Otherwise, a simple computation shows  $\mathcal{P}_{\rm sph}\left(r, \mathcal{U}_{0},s\right)=-l^*(n,s)$ for all $r>0$.
		Therefore, using again Lemma~\ref{lm:limitlimitinglevels}, we have $\mathcal{P}_{\rm sph}(0, \mathcal{U},s)=-l^*(n,s)$. 
		Since the converse trivially follows, we obtain that $\mathcal{P}_{\rm sph}(0, \mathcal{U},s) \in$ $\{-l^*(n,s),0\}$.
		Moreover, $\mathcal{P}_{\rm sph}(0, \mathcal{U},s)=0$, if and only if, all the blow-ups are trivial, whereas $\mathcal{P}_{\rm sph}(0, \mathcal{U},s)=-l^*(n,s)$, if and only if, all the blow-ups are of the form \eqref{gidassprucklimitsolution}.
		In the case of shrink-down $\mathcal{U}_{\infty}$ solution, the strategy is similar, so we omit it.
		These conclusions finish the proof of Case 2, and therefore the proposition holds.
	\end{proof}
	
	\section{Local asymptotic behavior near the isolated singularity}\label{sec:proof2}
	In this section, we present the Proof of Theorem~\ref{Thm3.2:asymptotics}. 
	First, we show an asymptotic symmetry result, which permits us to migrate to an ODE setup. 
	Second, we prove some universal upper bound estimates, not depending on the superharmonic assumption.
	However, we should emphasize that in the rest of the argument, there is a significant change of behavior of radial solutions \eqref{eq:subcriticalODEsystem} for distinct values of the power $s\in(1,2^{**}-1]$. 
	This difference occurs due to the change of sign of the coefficients in the bi-Laplacian written in cylindrical coordinates.
	These signs control the Lyapunov stability of the solutions to linearized operator around a limit blow-up solution, and so the asymptotic behavior of the local solutions near the isolated singularity.
	
	We divide our argument into five subsections, where we prove, respectively, some a priori upper bound estimates, asymptotic radial symmetry of singular solutions to \eqref{eq:subcriticalODEsystem}, and its precise local behavior near the isolated singularity for the situations: $s\in(1,2_{**}]$, $s=2_{**}$, and $s\in(2_{**},2^{**}-1)$.
	We will invert the order of studying the cases $s=2_{**}$ and $s\in(2_{**},2^{**}-1)$ since in the former case a more involved analysis is needed.
	In what follows, we recall the notation for the Fowler rescaling exponent $\gamma(s)={4}{(s-1)^{-1}}$.
	
	\subsection{A priori upper bounds}\label{subsec:universalupperbounds}
	This subsection is devoted to provide a priori upper bounds for local ($R<\infty$) solutions to \eqref{eq:subcriticalsystem} with $s\in(1,2^{**}-1)$, which further provide interior gradient estimates and compactness of the scaling functions. 
	Our strategy relies on the classification of the limit solutions to \eqref{eq:subcriticalsystem} combined with a blow-up argument. 
	
	\begin{lemma}\label{lm:upperestimatenonsingular}
		Let $R=1$, $s\in(1,2^{**}-1)$, and $\mathcal{U}\in C^{4}\left(B_{1}, \mathbb{R}^{p}\right) \cap$ $C\left(\bar{B}_{1},\mathbb{R}^{p}\right)$ be a nonnegative non-singular solution to \eqref{eq:subcriticalsystem}. Then,
		there exists $C>0$, (depending only on $n$ and $p$) and $s>0$ such that
		\begin{equation}\label{upperestimatenonsingular}
		|\mathcal{U}(x)| \leqslant C(1-|x|)^{-\gamma(s)} \quad {\rm in} \quad B_{1}.
		\end{equation}
	\end{lemma}
	
	\begin{proof}
		Arguing by contradiction, suppose that there exist a blow-up sequence $\{x_{k}\}_{k\in\mathbb{N}}\subset B_1$ and $\{\mathcal{U}_{k}\}_{k\in\mathbb{N}} \subset C^{4}\left(B_{1},\mathbb{R}^{p}\right)\cap C\left(\bar{B}_{1},\mathbb{R}^{p}\right)$ solutions to \eqref{eq:subcriticalsystem} such that
		\begin{equation}\label{upper1}
		\quad M_{k}:=\sup _{|x| \leqslant 1}\left((1-|x|)^{\gamma(s)}\left|\mathcal{U}_{k}(x)\right|\right)=\left(1-\left|x_{k}\right|\right)^{\gamma(s)}\left|\mathcal{U}_{k}\left(x_{k}\right)\right| \rightarrow \infty \quad \mbox{as} \quad k\rightarrow\infty.
		\end{equation}
		In addition, since $\mathcal{U}$ is continuous up to the boundary $\partial B_{1}$, we get 
		\begin{equation*}
		(1-|x|)^{\gamma(s)}\left|\mathcal{U}_{k}(x)\right|=0 \quad \mbox{on} \quad \partial B_1.
		\end{equation*} 
		Thus, we have that $\{x_{k}\}_{k\in\mathbb{N}}\subset B_{1}$, which allows us to take
		\begin{equation}\label{upper2}
		r_{k}:=\frac{1}{2}\left(1-\left|x_{k}\right|\right)>0
		\end{equation}
		and
		\begin{equation*}
		\widetilde{\mathcal{U}}_{k}(x)=(1-|x|)^{\gamma(s)} \mathcal{U}_{k}(x) \quad \mbox{in} \quad B_{1}.
		\end{equation*}
		By \eqref{upper1}, it follows that $|\widetilde{\mathcal{U}}_{k}|$ achieves its supremum in $B_{1}$ at $x_{k}$; thus,  $|\widetilde{\mathcal{U}}_{k}(x)| \leqslant|\widetilde{\mathcal{U}}_{k}\left(x_{k}\right)|=M_{k}$ for any $x \in B_{1}$.
		Next, taking $x \in B_{r_{k}}\left(x_{k}\right) \subset$
		$B_{1}$, we obtain that $1-|x|>r_{k}=\frac{1}{2}\left(1-\left|x_{k}\right|\right)$, which combined with \eqref{upper1}, yields
		\begin{equation*}
		\left|\mathcal{U}_{k}(x)\right| \leq\left(\frac{1-\left|x_{k}\right|}{1-|x|}\right)^{\gamma(s)}\left|\mathcal{U}_{k}\left(x_{k}\right)\right| \leqslant 2^{\gamma(s)}\left|\mathcal{U}_{k}\left(x_{k}\right)\right| \quad \quad \mbox{in} \quad B_{r_{k}}\left(x_{k}\right).
		\end{equation*}
		On the other hand, by substituting \eqref{upper2} into \eqref{upper1}, we find
		\begin{equation}\label{upper4}
		\left|\mathcal{U}_{k}\left(x_{k}\right)\right|=\left(2 r_{k}\right)^{-\gamma(s)} M_{k}.
		\end{equation}
		We also have 
		\begin{equation*}
		\delta_{k}=\left|\mathcal{U}_{k}\left(x_{k}\right)\right|^{-\gamma(s)^{-1}}=2 r_{k} M_{k}^{-\gamma(s)^{-1}} \rightarrow 0,
		\end{equation*}
		and
		\begin{equation}\label{upper6}
		R_{k}=\frac{r_{k}}{\delta_{k}}=\frac{1}{2} M_{k}^{\gamma(s)^{-1}} \rightarrow \infty \quad \mbox{as} \quad k\rightarrow\infty.
		\end{equation}
		Besides, by defining
		\begin{equation*}
		\widehat{\mathcal{U}}_{k}(x)=\delta_{k}^{\gamma(s)}{\mathcal{U}}_{k}\left(\delta_{k} x+x_{k}\right) \quad \mbox{in} \quad B_{R_{k}},
		\end{equation*}
		since $B_{r_{k}}\left(x_{k}\right) \subset B_{1}$, it holds that $\widehat{\mathcal{U}}_{k}$ is well-defined in $B_{R_{k}}$.
		Furthermore, $\widehat{\mathcal{U}}_{k}$ is nonnegative and satisfies
		\begin{equation}\label{upper8}
		\Delta^2\widehat{\mathcal{U}}_{k}=|\widehat{\mathcal{U}}_{k}|^{s-1}\widehat{u}_i \quad \mbox{in} \quad  B_{R_{k}}.
		\end{equation}
		According to \eqref{upper1}, we obtain
		\begin{equation}\label{upper9}
		|\widehat{\mathcal{U}}_{k}(x)| \leqslant 2^{\gamma(s)} \quad \mbox{in} \quad B_{R_{k}}.
		\end{equation}
		Therefore, from \eqref{upper8}, \eqref{upper9} and \eqref{upper4}, we get
		\begin{equation}\label{upper10}
		|\widehat{\mathcal{U}}_{k}(0)|=1.
		\end{equation} 
		Using \eqref{upper8}, \eqref{upper9} and \eqref{upper6}, we can apply the elliptic regularity theory to conclude that $\{\widehat{\mathcal{U}}_{k}\}_{k\in\mathbb{N}}$ is bounded in $C_{\rm loc}^{4,\zeta}\left(\mathbb{R}^{n},\mathbb{R}^{p}\right)$, for some $\zeta\in(0,1)$, which by passing to a subsequence, provides that there exists ${\mathcal{U}}_{0}\in C_{\rm loc}^{4,\zeta}\left(\mathbb{R}^{n},\mathbb{R}^{p}\right)$ such that $\widehat{\mathcal{U}}_{k} \rightarrow \widehat{\mathcal{U}}_{0}$ in $C_{\rm loc}^{4,\zeta}\left(\mathbb{R}^{n},\mathbb{R}^{p}\right)$.
		Thus, from \eqref{upper8} we deduce that $\widehat{\mathcal{U}}_{0}$ is a nonnegative solution to
		\begin{equation}\label{upper11}
		\Delta^2\widehat{\mathcal{U}}_{0}=|\widehat{\mathcal{U}}_{0}|^{s-1}(\widehat{u}_0)_i \quad \mbox{in} \quad \mathbb{R}^{n}.
		\end{equation}
		In addition, \eqref{upper9} and \eqref{upper10} yields $|\widehat{\mathcal{U}}_{0}(x)|\leqslant 2^{\gamma(s)}$ in $\mathbb{R}^{n}$ and $|\widehat{\mathcal{U}}_{0}(x)|=1$.
		However, by Proposition~\ref{prop:nonexistencenonsingular}, there exists no smooth global solution to \eqref{upper11}, which is a contradiction. 
		Therefore, the lemma holds for some $C>0$ depending only on $n$ and $s$.
	\end{proof}
	
	\begin{lemma}\label{lm:universalestimates}
		Let $R=1$, $s\in(1,2^{**}-1)$, and $\mathcal{U}$ be a nonnegative singular solution to \eqref{eq:subcriticalsystem}. 
		Then, there exists $C>0$, depending only on $n$, $p$ and $s$, such that
		\begin{equation}\label{universalestimates}
		|\mathcal{U}(x)| \leqslant C|x|^{-\gamma(s)} \quad {\rm in} \quad B^*_{1/2}.
		\end{equation}
	\end{lemma}
	
	\begin{proof}
			Fixing $x_{0} \in B^*_{1/2}$, let us define $r=\frac{1}{2}|x_{0}|$.
		Then, since $\bar{B}_{r}\left(x_{0}\right) \subset B^*_{1}$, it is well-defined the rescaled $p$-map given by 
		\begin{equation*}
			\widetilde{\mathcal{U}}_r(x)=r^{\gamma(s)}\mathcal{U}\left(rx+x_{0}\right) \quad \mbox{in} \quad \bar{B}_{1}.
		\end{equation*}
		Since $\mathcal{U}$ is a nonnegative singular solution to \eqref{eq:subcriticalsystem}, we obtain that $\widetilde{\mathcal{U}}_r$ is a nonnegative non-singular solution to
		\begin{equation*}
			\Delta^2\widetilde{\mathcal{U}}_r=|\widetilde{\mathcal{U}}_r|^{s-1}\widetilde{\mathcal{U}}_r \quad \mbox{in} \quad B_1,
		\end{equation*}
		which is continuous up to the boundary.
		Therefore, we can apply Lemma~\ref{lm:upperestimatenonsingular} to $\widetilde{\mathcal{U}}$, which, by taking $x=0$ in \eqref{upperestimatenonsingular}, provides $|\widetilde{\mathcal{U}}(0)|\leqslant C$. 
		Hence, by rewriting in terms of $\mathcal{U}$, we get $\left|{\mathcal{U}}\left(x_{0}\right)\right| \leqslant C r^{-\gamma(s)}$.
		At last, since $x_{0} \in B^*_{1/2}$ is arbitrary and $r=\frac{1}{2}\left|x_{0}\right|$, the proof is finished.
	\end{proof}
	
	\begin{corollary}\label{cor:universalestimates}
		Let $R=1$, $s\in(1,2^{**}-1)$, and $\mathcal{U}$ be a nonnegative singular solution to \eqref{eq:subcriticalsystem}. 
		Then,
		there exists $C>0$, depending only on $n$, $p$ and $s$, such that
		\begin{equation*}
			\sum_{j=0}^{3}|x|^{\gamma(s)+j}\left|D^{(j)} \mathcal{U}(x)\right| \leqslant C \quad {\rm in} \quad B^*_{1/2}.
		\end{equation*}	
	\end{corollary}
	
	\begin{proof}
		Fixing $x_{0} \in B^*_{1/2}$, let us define $r=\frac{1}{2}|x_{0}|$, and 
		\begin{equation*}
			\widetilde{\mathcal{U}}_r(x)=r^{\gamma(s)}\mathcal{U}\left(rx+x_{0}\right) \quad \mbox{in} \quad \bar{B}_{1}.
		\end{equation*}
		Then, $\widetilde{\mathcal{U}}_r$ is a nonnegative solution to \eqref{eq:subcriticalsystem} with $R=1$. 
		Hence, by the Lemma~\ref{lm:universalestimates}, it follows $\{|\widetilde{\mathcal{U}}_r|\}_{r>0}\subset C^{4,\zeta}(B_1)$, for some $\zeta\in(0,1)$, which, by standard elliptic estimates, gives us
		\begin{equation*}
			\sum_{j=0}^{3}|x|^{\gamma(s)+j}\left|D^{(j)} \widetilde{\mathcal{U}}_r(x)\right| \leqslant C.
		\end{equation*}	
		This result is proved by rescaling back to $\mathcal{U}$.
	\end{proof}
	
	The last lemma is a Harnack type inequality for local solutions to \eqref{eq:subcriticalsystem}.
	
	\begin{lemma}\label{lm:gradientestimates}
		Let $R=1$, $s\in(1,2^{**}-1)$, and $\mathcal{U}$ be a nonnegative singular solution to \eqref{eq:subcriticalsystem}.
		Then, there exists $c_1,c_2>0$ such that
		\begin{equation*}
		\sup_{B_{r}\setminus \bar{B}_{r / 2}} u_{i} \leqslant c_1\inf _{B_{r}\setminus\bar{B}_{r/2}} u_{i} \quad {\rm and} \quad \left|\nabla u_{i}(x)\right| \leqslant c_2\frac{u_{i}(x)}{|x|} \quad {\rm in} \quad B^*_{1/2} \quad {\rm for} \quad 0<r<1/2 \quad {\rm and} \quad i\in I.
		\end{equation*}
		Moreover, the constant $c_1>0$ depends only on $n$, $p$ and $s$.
	\end{lemma}
	
	\begin{proof}
		After a scaling argument, we get that \eqref{universalestimates} holds in $B^*_{3/4}$.
		Then, we consider $\mathcal{U}$ satisfying that $-\Delta^2 u_{i}=V(x) u_{i}$ in $B_1^*$, where $V(x)=|u|^{s-1}$.
		Moreover, using Lemma~\ref{lm:universalestimates}, we know that there exists $C>0$ such that $0 \leqslant V(x) \leqslant C|x|^{-2}$ in $B^*_{3/4}$.
		Thus, the conclusion follows quickly from the Harnack inequality from \cite[Theorem~3.6]{MR2240050}.
	\end{proof}
	
	\begin{remark}
		The argument here could be simplified by using twice the doubling property from \cite{MR2350853}.
		Nevertheless, we have preferred to give a more direct proof for the readers' convenience.
	\end{remark}
	
	\subsection{Asymptotic radial symmetry}\label{subsec:asymptoticsymmetry}
	Here we prove the first part of Theorem~\ref{Thm3.2:asymptotics} asserting that solutions to \eqref{eq:subcriticalsystem} are radially symmetric about the origin. 
	This symmetry will later be used to convert the singular PDE into a non-singular ODE on the cylinder.
	We use an asymptotic moving spheres technique in the same spirit of \cite{MR4085120}.
	
	\begin{proposition}\label{prop:asymptoticsymmetry}
		Let $R<\infty$, $s\in(1,2^{**}-1)$, and $\mathcal{U}$ be a nonnegative superharmonic singular solution to \eqref{eq:subcriticalsystem}. 
		Then,
		\begin{equation*}
		|\mathcal{U}(x)|=(1+\mathcal{O}(|x|))|\overline{\mathcal{U}}(x)| \quad {\rm as} \quad x\rightarrow0,
		\end{equation*}
		where $|\overline{\mathcal{U}}|(r)=\avint_{\partial B_{1}} |\mathcal{U}|(r \theta)\ud\theta$ is the spherical average of $|\mathcal{U}|$.
	\end{proposition}
	
	\begin{proof}
		Initially, if the origin is a removable singularity, then the conclusion is clear. Hence, we suppose that the origin is a non-removable singularity.
		Using the notation introduced in Section~\ref{sec:kelvintransform}, we divide the proof into some claims.
		
		\noindent{\bf Claim 1:} There exists small $0<\varepsilon\ll1$ such that for any $z \in B^*_{\varepsilon/2}$,  it holds
		\begin{equation}\label{asymptoticmovingplanes}
		\left(u_{i}\right)_{z, r}\leqslant u_{i} \quad {\rm in} \quad B_{1} \setminus\left(B_{r}(z)\cup\{0\}\right) \quad \mbox{for} \quad 0<r \leqslant|z| \quad \mbox{and} \quad i\in I.
		\end{equation}
		
		\noindent Indeed, the proof follows almost the same lines as the one in Lemma~\ref{prop:asymptoticsymmetrylimit}, so we omit it.
		
		In the next claim, we provide some estimates to be used later in the proof.
		
		\noindent{\bf Claim 2:} There exists $z\in B^*_{\varepsilon/2}$, $0<r<|z|$ and $\mu_*\gg1$ large such that
		\begin{equation}\label{estimatesradius}
		\frac{y_{\mu}}{|y_{\mu}|^{2}}-z=\left(\frac{r}{\left|y /|y|^{2}-z\right|}\right)^{2}\left(\frac{y}{|y|^{2}}-z\right) \quad {\rm and} \quad \frac{|y_{\mu}|}{|y|} \leqslant \frac{1}{r}\left|\frac{y}{|y|^{2}}-z\right| \quad {\rm for \ any} \quad \mu>\mu_*.
		\end{equation}
		Here $y_{\mu}=y+2(\mu-y\cdot{\bf e}){{\bf e}}$ is the reflection of $y$ about the hyperplane $\partial H_{\mu}({\bf e})$, where $H_{\mu}({\bf e})=\{x: \langle x,{\bf e}\rangle>\mu\}$ and ${\bf e}\in\mathbb{S}^{n-1}$. In other words, ${y_{\mu}}{|y_{\mu}|^{-2}}$ is the reflection point of ${y}{|y|^{-2}}$ about $\partial B_{r}(z)$. 
		
		\noindent As matter of fact,  choosing $r=|z|$, it involves an elementary computation, as follows
		\begin{equation*}
		z=\frac{y}{|y|^{2}} +\frac{\left|y_{\mu}\right|^{2}}{|y|^{2}-\left|y_{\mu}\right|^{2}}\left(\frac{y}{|y|^{2}} -\frac{y_{\mu}}{\left|y_{\mu}\right|^{2}} \right)=\frac{\left(y-y_{\mu}\right)}{|y|^{2}-\left|y_{\mu}\right|^{2}}.
		\end{equation*}
		
		Next, we establish a comparison involving the Kelvin transform of a component solution with itself.
		
		\noindent{\bf Claim 3:} For any $\mu>\frac{1}{\varepsilon}$ and ${\bf e}\in\partial B_{1}$, if $\langle x,{\bf e}\rangle>\mu$ and $|y_{\mu}|>1$ for each $i\in I$, it holds
		\begin{equation*}
		(u_{i})_{0,1}(y) \leqslant (u_{i})_{0,1}\left(y_{\mu}\right).
		\end{equation*}
		
		\noindent In fact, to prove the last inequality, let us note first that $y\in B_{1/\varepsilon}$, if and only if, ${y}{|y|^{-2}}\in B_{\varepsilon}$.
		Now given $y\in \mathbb{R}^{n}$ such that $\langle y,{\bf e}\rangle>\mu$, $|y_{\mu}|>1$ and $0<r<|z|<{\varepsilon}/{2}$ satisfying \eqref{estimatesradius}. Let us define $x={y}{|y|^{-2}}$ and $x_{z,r}={y_{\mu}}{|y_{\mu}|^{-2}}$.
		Then, since $\langle y,{\bf e}\rangle>\mu>{\varepsilon}^{-1}$ and $\left|y_{\mu}\right|>1$, we have $x \in B_{r}(z)$ and $x_{z,r}\in B_{1}\setminus B_{r}(z)$. Hence, using \eqref{asymptoticmovingplanes} and \eqref{estimatesradius}, we find that $(u_{i})_{0,1}(y)\leqslant (u_{i})_{0,1}(y_{\mu})$, which proves the claim.
		
		Ultimately, using Claim 3, we invoke \cite[Theorem~6.1 and Corollary~6.2]{MR982351} to find
		$C>0$, independent of $\varepsilon>0$, such that if $|y|\geqslant |x|+C\varepsilon^{-1}$, it follows $(u_{i})_{0,1}(y) \leqslant (u_{i})_{0,1}(x)$ for any $i\in I$.
		Therefore, since $(u_{i})_{0,1}$ is nonnegative and superharmonic, the last inequality implies
		\begin{equation*}
		u_{i}^{*}(|x|)=\left(1+\mathcal{O}\left(\frac{1}{R}\right)\right)\left(\inf _{\partial B_{R}} u_{i}^{*}\right) \quad \mbox{as} \quad R \rightarrow \infty \quad {\rm for} \quad i\in I,
		\end{equation*} 
		uniformly on $\partial B_{R}$, which in terms of $u_{i}$ implies the asymptotic radial symmetry, and
		the proof is concluded.
	\end{proof}
	
	\subsection{Serrin--Lions case}\label{subsec:serrin-lions}
	We prove Theorem~\ref{Thm3.2:asymptotics} (a). 
	The asymptotic analysis for this case is straightforward.
	We can reduce the problem to the scalar case ($p=1$) by considering the sum function $\mathcal{U}_{\Sigma}=u_1+\cdots+u_p$ and applying \cite[Section~4]{MR1436822}.
	However, we give a more direct proof.
	In the sequel, we aim to prove the following proposition.
	
	\begin{proposition}\label{prop:serrinlionscase}
		Let $R<\infty$, $s\in(1,2_{**})$, and $\mathcal{U}$ be a nonnegative superharmonic singular solution to \eqref{eq:subcriticalsystem}. 
		Then, there exist $C_1,C_2>0$ (depending on $\mathcal{U}$) such that $C_1|x|^{4-n}\leqslant|\mathcal{U}(x)|\leqslant C_2|x|^{4-n}$ for $0<|x|\ll1$, or equivalently, $|\mathcal{U}(x)|\simeq |x|^{4-n}$ as $x\rightarrow 0$.
	\end{proposition}
	
	First, we prove an upper bound estimate based on a Green identity from Proposition~\ref{lm:integralrepresentation}.
	
	\begin{lemma}\label{lm:upperestimateserrin}
		Let $R<\infty$, $s\in(1,2_{**})$, and $\mathcal{U}$ be a nonnegative singular solution to \eqref{eq:subcriticalsystem}. 
		Then, there exists $C_2>0$, depending only on $|\mathcal{U}|$, such that
		$|\mathcal{U}(x)|\leqslant C_2|x|^{4-n}$ as $x\rightarrow 0$.
	\end{lemma}
	
	\begin{proof}
		Initially, by Lemma~\ref{lm:integrability}, we have that $\mathcal{U}\in L^{s}\left(B_{1},\mathbb{R}^p\right)$. Moreover, since $s\in(1,2_{**})$ and $\mathcal{U}$ satisfies the Harnack inequality in Lemma~\ref{lm:gradientestimates}, it follows
		\begin{equation*}
		|\mathcal{U}(x)|=o\left(|x|^{-\gamma(s)}\right) \quad \mbox{as} \quad x\rightarrow 0,
		\end{equation*}
		which by $n-4<\gamma(s)$, implies that for any $n-4<q<\gamma(s)$, there exists $0<r_{q}<1$ depending only on $n$ $p$, $s$ and $q$ such that
		\begin{equation}\label{lions1}
		|\mathcal{U}(x)|<|x|^{-q} \quad \mbox{in} \quad B^*_{r_{q}},
		\end{equation}
		where in the last claim we have used a blow-up argument.
		Now taking $r_{q}>0$ as before, and using Lemma~\ref{lm:integrability} again, we get that $\Delta^2 \mathcal{U}=|\mathcal{U}|^{s-1} \mathcal{U} \in L^{1}\left(B_{1},\mathbb{R}^p\right)$. 
		Thus, using \eqref{integralsystem}, we decompose 
		\begin{equation}\label{lions2}
		\mathcal{U}(x)=\Lambda|x|^{4-n} -\int_{B_{r_q}}|x-y|^{4-n} \Delta^2\mathcal{U}(y) \ud y+\Psi(x) \quad \mbox{in} \quad B^*_{r_{q}},
		\end{equation}
		where $\Lambda\in\mathbb{R}^p$ has all nonnegative components and $\Psi\in C^{\infty}\left(B_{1},\mathbb{R}^p\right)$ is such that $\Delta^2\Psi=0$ in $B_{r_{q}}$. Nevertheless, using \eqref{lions1}, it is not hard to see from \eqref{eq:subcriticalsystem} that there exists $C_q>0$, depending only on $n$, $p$, $s$, and $q$ such that
		\begin{equation*}
		\left|\int_{B_{r_{q}}}|x-y|^{4-n}\Delta^2 \mathcal{U}(y) \ud y\right| \leqslant \int_{B_{r_{q}}}|x-y|^{4-n}|y|^{-sq} \ud y \leqslant C_{q}|x|^{4-n}.
		\end{equation*}
		Hence, fixing $n-4<q<\gamma(s)$ and choosing suitable $r_{q}>0$ and $C_{q}>0$ on the last inequality, the proof follows directly from \eqref{lions2}.
	\end{proof}
	
	Second, we give a sufficient condition to classify whether the origin is a removable singularity or non-removable singularity.
	
	\begin{lemma}\label{lm:removabilityserrin}
		Let $R<\infty$, $s\in(1,2_{**})$, and $\mathcal{U}$ be a nonnegative superharmonic singular solution to \eqref{eq:subcriticalsystem}. 
		Assume that
		\begin{equation}\label{serrinestimate}
		|\mathcal{U}(x)|=o\left(|x|^{4-n}\right) \quad \mbox{as} \quad x\rightarrow 0.
		\end{equation}
		Then, the origin is a removable singularity.
	\end{lemma}
	
	\begin{proof}
		By \eqref{serrinestimate}, we get $\mathcal{U} \in L^{q}\left(B_{1},\mathbb{R}^{p}\right)$ for any $q\in[1,2_{**})$.
		Moreover, since $s\in(1,2_{**})$ and $|\Delta^2\mathcal{U}| \leqslant|\mathcal{U}|^{s}$, it follows $-\Delta^2 \mathcal{U} \in L^{q/s}\left(B_{1},\mathbb{R}^{p}\right)$ for any $q\in[1,2_{**})$. Whence, we can use standard elliptic theory for each component of $\mathcal{U}$, and a bootstrap argument to find $\mathcal{U}\in W^{4,N}\left(B_{1},\mathbb{R}^{p}\right)$ for any $N\in(1,\infty)$. In particular, it holds from the Morrey embedding that $\mathcal{U} \in C^{3, \zeta}\left(B_{1},\mathbb{R}^{p}\right)$ for any $\zeta\in(0,1)$. 
		Therefore, $\mathcal{U}$ must have a removable singularity at the origin.
	\end{proof}
	
	Now we are in a position to prove our main result of this part.
	
	\begin{proof}[Proof of Proposition~\ref{prop:serrinlionscase}]
		Suppose that $\mathcal{U}$ has a non-removable singularity at the origin. Using Lemma~\ref{lm:removabilityserrin}, we get that $\mathcal{U}$ does not satisfy \eqref{serrinestimate}, that is, there exists $\rho>0$, $i\in I$, and $\{r_{k}\}_{k\in\mathbb{N}}$ such that $r_{k}\rightarrow 0$ as $k\rightarrow\infty$ satisfying
		\begin{equation*}
		\sup_{\partial B_{r_{k}}}u_{i} \geqslant \rho r_{k}^{4-n}.
		\end{equation*}
		On the other hand, by the Harnack inequality in Lemma~\ref{lm:gradientestimates}, there exists $c_1>0$ satisfying
		\begin{equation*}
		\inf_{\partial B_{r_{k}}} u_{1} \geqslant c_{1} \rho r_{k}^{4-n},
		\end{equation*}
		where $c_{1}>0$ depends only on $n$, $p$ and $s$. 
		Taking $0<\rho\ll1$ smaller to ensure that there exists $c_2\rho \leqslant \inf_{\partial B_{1/2}}u_{i}$, it follows from using twice the maximum principle that
		\begin{equation*}
		u_{i}(x) \geqslant c_{2} \rho|x|^{4-n} \quad \mbox{in} \quad B^*_{1/2},
		\end{equation*}
		which proves the asymptotic lower bound estimate in this case, and together with Lemma~\ref{lm:upperestimateserrin}, the proof of the proposition is concluded. 
	\end{proof}
	
	\subsection{Gidas--Spruck case}\label{subsec:gidasspruck}
	The objective of this subsection is to prove Theorem~\ref{Thm3.2:asymptotics} (c).
	Our strategy is based on the monotonicity formula for the Pohozaev functional in cylindrical coordinates (see Proposition~\ref{lm:monotonicityformula}), which relies on \cite{MR4123335}.
	More precisely, we show that the local models near the origin are the limit blow-up solutions, whose limits are provided by its image under the action of the spherical Pohozaev functional.
	Finally, to prove the removability of the singularity theorem, we use a technique relying on the regularity lifting method in \cite{MR1338474}.
	
	\begin{proposition}\label{prop:gidasspruckcase}
		Let $R<\infty$, $s\in(2_{**},2^{**}-1)$, and $\mathcal{U}$ be a nonnegative singular solution to \eqref{eq:subcriticalsystem}. 
		Then, there exists $K_{0}(s,n)>0$ such that
		\begin{equation*}
		|\mathcal{U}(x)|=(1+o(1))K_0(n,s)^{\frac{1}{s-1}}|x|^{-\gamma(s)}.
		\end{equation*}
	\end{proposition}
	
	Before we provide the proof of this proposition, we need to establish some auxiliary results. 
	First, we show that solution $\mathcal{U}$ to \eqref{eq:subcriticalsystem} satisfy some uniform bound in cylindrical coordinates $\mathcal{V}=\mathfrak{F}(\mathcal{U})$ (see \eqref{cyltransform}).
	
	\begin{lemma}\label{lm:uniformestimategidasspruck}
		Let $R<\infty$, $s\in(2_{**},2^{**}-1)$, and $\mathcal{U}$ be a nonnegative singular solution to \eqref{eq:subcriticalsystem}. 
		Then, there exists $C>0$ such that
		\begin{equation*}
		|\mathcal{V}(t,\theta)|+|\mathcal{V}^{(1)}(t,\theta)|+|\mathcal{V}^{(2)}(t,\theta)|+|\mathcal{V}^{(3)}(t,\theta)|+|\nabla_{\theta}\mathcal{V}(t,\theta)|+|\Delta_{\theta}\mathcal{V}(t,\theta)|\leqslant C \quad {\rm in} \quad \mathcal{C}_{-\ln2}.	\end{equation*}
	\end{lemma} 
	
	\begin{proof}
		First, by Lemma~\ref{lm:universalestimates}, we know that $\mathcal{V}$ is uniformly bounded. 
		Moreover, using Corollary~\ref{cor:universalestimates}, we get
		\begin{align*}
		|\mathcal{V}^{(1)}(t,\theta)|+|\nabla_{\theta} \mathcal{V}(t,\theta)|& \leqslant C \sum_{j=0}^{1}|x| ^{\gamma(s)}|D^{(j)} \mathcal{U}(x)| \leqslant C, \\
		|\mathcal{V}^{(2)}(t,\theta)|+|\Delta_{\theta} \mathcal{V}(t,\theta)| & \leqslant C \sum_{j=0}^{2}|x|^{\gamma(s)+j}|D^{(j)}\mathcal{U}(x)| \leqslant C, \\
		|\mathcal{V}^{(3)}(t,\theta)| & \leqslant C \sum_{j=0}^{3}|x|^{\gamma(s)+j}|D^{(j)} \mathcal{U}(x)| \leqslant C,
		\end{align*}
		for $0<|x|<1/2$, which by a direct rescaling proves the lemma.
	\end{proof}
	
	Now we use this rescaled family $\{\widehat{\mathcal{U}}_{\lambda}\}_{\lambda>0}\subset C^{4,\zeta}(B^*_1,\mathbb{R}^p)$, for some $\zeta\in(0,1)$, to obtain the blow-up limit for \eqref{eq:subcriticalsystem}. This allows us to study the limiting values for the Pohozaev functional, by using the classification results from Section~\ref{sec:blowlimitcase}.
	
	\begin{lemma}\label{lm:limitinglevesgidasspruck}
		Let $R<\infty$, $s\in(2_{**},2^{**}-1)$, and $\mathcal{U}$ be a nonnegative singular solution to \eqref{eq:subcriticalsystem} with $\mathcal{V}={\mathfrak{F}}(\mathcal{U})$ its cylindrical transform given by \eqref{newcylindrical}. Then, ${\mathcal{P}}_{\rm cyl}(\infty,\mathcal{V}) \in\left\{-l^*(n,s),0\right\}$, where $l^*(n,s)$ is defined by \eqref{limintingconstant}.
		Moreover, it follows\\
		\noindent{\rm (i)} ${\mathcal{P}}_{\rm cyl}(\infty,\mathcal{V})=0$ if, and only if,
		\begin{equation}\label{asymptoticgidasspruck}
		|\mathcal{U}(x)|=o\left(|x|^{-\gamma(s)}\right) \quad \mbox{as} \quad x \rightarrow 0.
		\end{equation}
		\noindent{\rm (ii)} ${\mathcal{P}}_{\rm cyl}(\infty,\mathcal{V})=-l^*(n,s)$ if, and only if,
		\begin{equation*}
		|\mathcal{U}(x)|=(1+o(1)K_0(n,s)^{\frac{1}{s-1}}|x|^{-\gamma(s)} \quad \mbox{as} \quad x \rightarrow 0,
		\end{equation*}
	\end{lemma}
	
	\begin{proof}
		Initially, by Lemma~\ref{lm:universalestimates}, for any $K\subset B_{1/2\lambda}$ compact subset, the family 
		$\{\widehat{\mathcal{U}}_{\lambda}\}_{\lambda>0}\subset C^{4,\zeta}(B^*_1,\mathbb{R}^p)$ is uniformly bounded, for some $\zeta\in(0,1)$.
		Then, by standard elliptic theory, there exists a nonnegative function $\mathcal{U}_{0} \in C^{4,\zeta}\left(\mathbb{R}^{n}\setminus\{0\},\mathbb{R}^p\right)$, such that, up to a subsequence, we have that $\|\widehat{\mathcal{U}}-\mathcal{U}_{0}\|_{C_{\rm loc}^{4,\zeta}(\mathbb{R}^{n} \setminus\{0\})}$ as $\lambda \rightarrow 0$, where $\mathcal{U}_{0}$ satisfies the blow-up limit system \eqref{eq:subcriticalsystem}.
		Moreover, by Lemma~\ref{lm:superharmonicity},  we know that $\mathcal{U}_0$ is superharmonic, that is, $-\Delta(u_0)_i\geqslant 0$ in $\mathbb{R}^{n} \setminus\{0\}$, which, by the maximum principle, yields that either $(u_0)_i \equiv 0$
		or $(u_0)_i>0$
		in $\mathbb{R}^{n} \backslash\{0\}$
		for all $i\in I$.
		Therefore, by Theorem~\ref{thm:andrade-do2020}, the blow-up limit $\mathcal{U}_{0}$ is radially symmetric about the origin.
		Furthermore, by the scaling invariance of the Pohozaev functional, we get
		\begin{equation}\label{limitinggs}
			\mathcal{P}_{\rm sph}(r,\mathcal{U}_{0},s)=\lim_{\lambda\rightarrow 0} \mathcal{P}_{\rm sph}(r, \widehat{\mathcal{U}}_{\lambda},s)=\lim_{\lambda \rightarrow0} \mathcal{P}_{\rm sph}(\lambda r,\mathcal{U}_0,s)=\mathcal{P}_{\rm sph}(0, \mathcal{U}_0,s).
		\end{equation}
		In addition, if $\mathcal{V}_0=\mathfrak{F}(\mathcal{U}_0)$, then it satisfies \eqref{eq:subcriticalcylindricalsystem}, which by \eqref{limitinggs}, yields that $\mathcal{P}_{\rm cyl}(t,\mathcal{V}_{0},s)=\mathcal{P}_{\rm sph}(r,\mathcal{U}_{0},s)$ is a constant.
		Consequently, by the monotonicity formula in Proposition~\ref{lm:monotonicityformula}, we get 
		\begin{equation*}
			\frac{\ud}{\ud t} \mathcal{P}_{\rm cyl}(t,\mathcal{V}_{0},s)=\left[K_{3}(n,s)|\mathcal{V}_0^{(2)}|^{2}-K_{1}(n,s)|\mathcal{V}_0^{(1)}|^{2}\right] \equiv 0.
		\end{equation*}
		Moreover, since $K_{3}(n,s)<0$ and $K_{1}(n,s)>0$, we find that $|\mathcal{V}^{(1)}_0|\equiv 0$ in $\mathbb{R}$, and so $|\mathcal{V}_0|$ is constant, which can be directly computed, namely either
		\begin{equation*}
			|\mathcal{V}_0|=0 \quad \mbox{or} \quad |\mathcal{V}_0|=K_{0}(n,s)^{\frac{1}{s-1}}.
		\end{equation*}
		Moreover, by \eqref{limitinggs} and Remark~\ref{relatinpohozevfinctionals}, it follows
		\begin{equation*}
			\mathcal{P}_{\rm cyl}(0, \mathcal{V}_0,s) \in\left\{-l^*(n,s) ,0\right\} \quad \mbox{and} \quad \mathcal{P}_{\rm sph}(0, \mathcal{U}_0,s)\in\left\{-\omega_{n-1}l^*(n,s) ,0\right\}.
		\end{equation*}
		
		Finally, if $\mathcal{P}_{\rm sph}(0, \mathcal{U}_0,s)=0$, then, by uniqueness of the limit $|\mathcal{U}_0|\equiv 0$. Whence, we conclude that $\|\widehat{\mathcal{U}}_{\lambda}\|_{C^{4,\zeta}(K,\mathbb{R}^p)} \rightarrow 0$ for any sequence of $\lambda\rightarrow0$, for some $\zeta\in(0,1)$, which straightforwardly provides \eqref{asymptoticgidasspruck}.
		Otherwise, we have
		\begin{equation*}
			|\mathcal{U}_0|\equiv K_{0}(n,s)^{\frac{1}{s-1}}|x|^{-\gamma(s)},
		\end{equation*}
		which proves (ii) of this lemma and finishes the proof of the lemma.
	\end{proof}
	
	Next, we use the last lemma to prove the removable singularity theorem. 
	Our proof is based on regularity lifting methods combined with the De Giorgi--Nash-Moser iteration technique.
	
	\begin{lemma}\label{lm:removablesingularitygidasspruck}
		Let $R<\infty$, $s\in(2_{**},2^{**}-1)$, and $\mathcal{U}$ be a nonnegative solution to \eqref{eq:subcriticalsystem}. 
		If
		\begin{equation*}
		|\mathcal{U}(x)|=o\left(|x|^{-\gamma(s)}\right) \quad \mbox{as} \quad x \rightarrow 0,
		\end{equation*}
		then the origin is a removable singularity.
	\end{lemma} 
	
	\begin{proof}
		The proof will be divided into some claims.
		
		\noindent{\bf Claim 1:} If
		$|\mathcal{U}(x)|=o\left(|x|^{-\gamma(s)}\right)$ as $x \rightarrow 0
		$, then
		\begin{equation}\label{condition}
		\int_{B_{1/2}}|\mathcal{U}|^{n\gamma(s)^{-1}}\ud x<\infty.	
		\end{equation}
		
		\noindent In fact, let us consider $\phi(|x|)=|x|^{\zeta(n,s)}$, where $\zeta(n,s)=-2\gamma(s)(2^{**}-1)(n-2_{**})(s-1)^{-1}$.
		Then, a direct computation, provides
		\begin{equation*}
		\Delta^{2}\phi=\zeta(n,s)(\zeta(n,s)-2)(\zeta(n,s)+n-2)(\zeta(n,s)+n-4)|x|^{\zeta(n,s)-4},
		\end{equation*}
		which, since $\zeta(n,s)+n-4=\gamma(s)>0$, it follows that
		\begin{equation*}
		A(n,s):=\zeta(n,s)(\zeta(n,s)-2)(\zeta(n,s)+n-2)(\zeta(n,s)+n-4)>0.
		\end{equation*}
		Thus, we can write
		\begin{equation}\label{gs4}
		\frac{\Delta^{2} \phi}{\phi}=\frac{A(n,s)}{|x|^{4}} \quad \mbox{in} \quad \mathbb{R}^{n}\setminus\{0\}.
		\end{equation}
		
		For any $0<\varepsilon \ll 1$, let us consider $\eta_{\varepsilon}\in C^{\infty}\left(\mathbb{R}^{n}\right)$ with $0\leqslant\eta_{\varepsilon}\leqslant1$ a cut-off function satisfying
		\begin{equation}\label{cutoff2}
		\eta_{\varepsilon}(x)=
		\begin{cases}
		0, & \mbox{for} \ \varepsilon\leqslant|x| \leqslant 1/2\\
		1, & \mbox{for} \ |x|\leqslant\varepsilon/2 \ \mbox{or} \ |x|\geqslant3/4,
		\end{cases}
		\end{equation}
		and $|D^{(j)} \eta_{\varepsilon}(x)| \leqslant C \varepsilon^{-j}$ for $j=0,1,2,3,4$.
		Defining $\xi_{\varepsilon}=\eta_{\varepsilon}\phi$, multiplying \eqref{eq:subcriticalsystem} by $\xi_{\varepsilon}$ and integrating by parts in $B_{1}$, we obtain
		\begin{equation}\label{gs1}
		\int_{B_{1}} \eta_{\varepsilon} u_i \phi\left(\frac{\Delta^{2} \phi}{\phi}-|\mathcal{U}|^{s-1}u_i\right)\ud x=-\int_{B_{1}} u_i \mathfrak{T}\left(\eta_{\varepsilon}, \phi\right)\ud x \quad \mbox{for all} \quad i\in I,
		\end{equation}
		where $\mathfrak{T}_\varepsilon:C^{\infty}_c(B_1)\rightarrow C^{\infty}_c(B_1)$ is defined by
		\begin{equation*}
		\mathfrak{T}_\varepsilon(\phi)=\mathfrak{T}\left(\eta_{\varepsilon}, \phi\right)=4 \nabla \eta_{\varepsilon}\nabla \Delta \phi+2 \Delta \eta_{\varepsilon} \Delta\phi+4 \Delta\eta_{\varepsilon}\Delta\phi+4\nabla\Delta\eta_{\varepsilon}\nabla\phi+\phi \Delta^{2}\eta_{\varepsilon}.
		\end{equation*}
		Using Lemma~\ref{lm:universalestimates} combined with the estimates on the cut-off function \eqref{cutoff2} and its derivatives, there exist $c_1,c_2>0$, independent of $\varepsilon$, satisfying the following estimates,
		\begin{equation*}
		\left|\int_{B_{1}} u_i \mathfrak{T}_\varepsilon(\phi)\ud x\right| \leqslant  c_{1}+c_{2} \varepsilon^{n} \varepsilon^{\zeta(n,s)-4} \varepsilon^{-\gamma(s)}<\infty,
		\end{equation*}
		which implies that the right-hand side of \eqref{gs1} is uniformly bounded.
		In addition, assumption \eqref{condition} yields that $|\mathcal{U}|^{s-1}(x)=o(1)|x|^{-4}$ as $x\rightarrow 0$, which together with \eqref{gs4} and \eqref{gs2} provides that there exists $C>0$ satisfying
		\begin{equation}\label{gs2}
		\int_{B_{1}}\eta_{\varepsilon} u_i|x|^{\zeta(n,s)-4}\ud x\leqslant C \quad \mbox{for all} \quad i\in I.
		\end{equation}
		Therefore, by Lemma~\ref{lm:universalestimates}, it holds
		\begin{align}\label{gs3}
		\int_{\{\varepsilon \leqslant|x| \leqslant {1}/{2}\}} |\mathcal{U}|^{n\gamma(s)^{-1}}\ud x &=\int_{\{\varepsilon \leqslant|x| \leqslant {1}/{2}\}} |\mathcal{U}||\mathcal{U}|^{n\gamma(s)^{-1}-1}\ud x\\\nonumber
		& \leqslant C\int_{\{\varepsilon \leqslant|x| \leqslant {1}/{2}\}} |\mathcal{U}||x|^{\zeta(n,s)-4}\ud x\\\nonumber
		&\leqslant C\int_{B_{1}} \eta_{\varepsilon} |\mathcal{U}||x|^{\zeta(n,s)-4}\ud x<\infty,
		\end{align}
		where the last inequality comes from \eqref{gs2}.
		Finally, passing to the limit as $\varepsilon \rightarrow 0$ in \eqref{gs3}, the proof of Claim 1 follows by applying the dominated convergence theorem.
		
		\noindent{\bf Claim 2:} If \eqref{condition}, holds then $\mathcal{U} \in L^{q}(B_1,\mathbb{R}^p)$ for all $q>2^{**}$.
		
		\noindent Indeed, by Proposition~\ref{lm:integralrepresentation}, there exist a Green function with homogeneous Dirichlet boundary conditions $G_2(x,y)$ and $\psi_i\in L^{\infty}_{\rm loc}(B_{1/2})$ with $\Delta\psi_i=0$ for all $i\in I$ such that
		\begin{equation*}
		u_i(x)=\int_{B_1}G_2(x,y)\Delta^2u_i\ud x+\psi_i(x) \quad \mbox{in} \quad B_{1/2}.
		\end{equation*}
		More precisely, $G_{2}(x, y)$ is a distributional solution to the Dirichlet problem
		\begin{equation*}
		\begin{cases}
		\Delta^{2} G_{2}(x,y)=\delta_x(y) & \mbox{in} \quad B_{1/2}\\
		G_{2}(x, y)=\partial_{\nu}G_{2}(x, y)=0 & \mbox{on} \quad \partial B_{1/2}.
		\end{cases}
		\end{equation*}
		and there exists positive constant $C_{n}>0$ such that
		\begin{equation*}
		0<G_{2}(x, y) \leqslant \Gamma_{2}(|x-y|):=C_{n}|x-y|^{4-n} \quad \mbox{for} \quad  x,y\in B_{1/2} \quad \mbox{and} \quad |x-y|>0,
		\end{equation*}
		where $\Gamma_2(x,y)=C_n|x-y|^{4-n}$ is the fundamental solution to $\Delta^2$ in $\mathbb{R}^n$.
		Recall that $\mathcal{U}=(u_1,\dots,u_p)$ satisfies
		\begin{equation}\label{potentialsystem}
		\Delta^2 u_{i}=V(x) u_{i} \quad \mbox{in} \quad B_1^*,
		\end{equation} 
		where $V(x)=|\mathcal{U}|^{s-1}$. 
		Moreover, using \eqref{condition}, we find that $V\in L^{n/4}(B_{1/2})$.
		
		Let us consider the $Z=C^{\infty}_c(B_{1/4})$, $X=L^{2^{**}}(B_{1/4})$ and $Y=L^{s}(B_{1/4})$ for $q>2^{**}$. 
		Hence, it is well-defined the following inverse operator
		\begin{equation*}
		(Tu)(x)=\int_{B_{1/4}}\Gamma_2(x,y)u(y)\ud y.
		\end{equation*}
		We also consider the operator $T_M:=\Gamma_2\ast V_M$, which applied in both sides of \eqref{potentialsystem}, provides $u_i=T_Mu_i+\widetilde{T}_Mu_i$, where 
		\begin{equation*}
		(T_M u_i)(x)=\int_{B_{1/4}}\Gamma_2(x,y)V_{M}(y)u_i(y)\ud y \quad \text{and} \quad (\widetilde{T}_M u_i)(x)=\int_{B_{1/4}}\Gamma_2(x,y)\widetilde{V}_M(y)u_i(y)\ud y.
		\end{equation*}
		Here, for $M>0$, we define  $\widetilde{V}_M(x)=V(x)-V_M(x)$, where
		\begin{equation*}
		V_M(x)=
		\begin{cases}
		V(x), \ {\rm if} \ |V(x)|\geqslant M,\\
		0,\ {\rm otherwise}.
		\end{cases}
		\end{equation*}
		
		Now we can run the regularity lifting method, which is divided into two steps.
		
		\noindent{\bf Step 1:} For $n/(n-4)<q<\infty$, there exists $M\gg1$ large such that $T_M:L^{q}(B_{1/4})\rightarrow L^{q}(B_{1/4})$ is a contraction.
		
		\noindent In fact, for any $q\in(n/(n-4),\infty)$, there exists $m\in (1,n/4)$ such that $q=nm/(n-4m)$. Then, by the Hardy--Littlewood--Sobolev and H\"{o}lder inequalities \cite{MR717827}, for any $u\in L^{q}(\mathbb{R}^n)$, we get
		\begin{equation*}
		\|T_Mu\|_{L^{q}(B_{1/4})}\leqslant\|\Gamma_2\ast V_{M}u\|_{L^{q}(B_{1/4})}\leqslant C\|V_M\|_{L^{{n}/{4}}(B_{1/4})}\|u\|_{L^{q}(B_{1/4})}.
		\end{equation*}
		Since $V_{M}\in L^{n/4}(B_{1/4})$ it is possible to choose a large $M\gg1$ satisfying $\|V_M\|_{L^{{n}/{4}}(B_{1/4})}<{1}/{2C}$. 
		Therefore, we arrive at $\|T_Mu\|_{L^{q}(B_{1/4})}\leqslant{1}/{2}\|u\|_{L^{q}(B_{1/4})}$, which yields that $T_M$ is a contraction.
		
		\noindent{\bf Step 2:} For any $n/(n-4)<q<\infty$, it follows that $\widetilde{T}_Mu_i\in L^{q}(B_{1/4})$ for all $i\in I$.
		
		\noindent Indeed, for any $n/(n-4)<q<\infty$, we pick $1<m<n/4$ satisfying $q=nm/(n-4m)$. 
		Since $\widetilde{V}_M$ is bounded, we get
		\begin{equation*}
		\|\widetilde{T}_M\|_{L^{q}(B_{1/4})}=\|\Gamma_2\ast\widetilde{V}_Mu_i\|_{L^{q}(B_{1/4})}\leqslant C\|\widetilde{V}_Mu_i\|_{L^{m}(B_{1/4})}\leqslant C\|u_i\|_{L^{m}(B_{1/4})}.
		\end{equation*}
		However, using \eqref{condition}, we have that  $u_i\in L^{q}(B_{1/4})$ for $q\in(1,n\gamma(s)^{-1})$.
		Besides, $q=(s-2)n\gamma(s)^{-1}$ when $m=n\gamma(s)^{-1}$. Thus, we obtain that $u_i\in L^{q}(B_{1/4})$ for
		\begin{equation*}
		\begin{cases}
		1<q<\infty,& {\rm if} \ s\geqslant2\\
		1<q\leqslant(2-s)^{-1}n\gamma(s)^{-1},& {\rm if} \ 1<s<2.
		\end{cases}
		\end{equation*}
		Now we can repeat the argument for $m=(s-2)n\gamma(s)^{-1}$ to get that $u_i\in L^{q}(B_{1/4})$ for
		\begin{equation*}
		\begin{cases}
		1<q<\infty,& {\rm if} \ s\geqslant2\\
		1<q\leqslant(2-s)^{-1}n\gamma(s)^{-1},& {\rm if} \ 1<s<2.
		\end{cases}
		\end{equation*}
		Therefore, by proceeding inductively as in \cite[Lemma~3.8]{MR4123335}, the proof of the claim follows.
		Ultimately, combining Steps 1 and 2, we can apply \cite[Theorem~3.3.1]{MR1338474} to show that $u_i \in L^{s}(B_{1/4})$ for all $s>2^{**}$ and $i\in I$. In particular, the proof of the claim is finished.
		
		Now, by the Morrey embedding theorem, it follows that $u_i\in C^{0,\zeta}(B_{1/4})$, for some $\zeta\in(0,1)$. Finally using Schauder estimates, one gets that $u_i\in C^{4,\zeta}(B_{1/4})$, which provides $\mathcal{U}\in C^{4,\zeta}(B_{1/4},\mathbb{R}^p)$. In particular, the singularity at the origin is removable, which concludes the proof of the lemma.
	\end{proof}

	\begin{proof}[Proof of Proposition~\ref{prop:gidasspruckcase}]
		Suppose that $\mathcal{U}$ has a non-removable singularity at the origin, then by Lemma~\ref{lm:removablesingularitygidasspruck}, $\mathcal{U}$ does not satisfy \eqref{asymptoticgidasspruck}. Therefore, the proof follows as a consequence of Lemma~\ref{lm:limitinglevesgidasspruck}.
	\end{proof}
	
	\subsection{Aviles case}\label{subsec:aviles}
	Finally, we prove Theorem~\ref{Thm3.2:asymptotics} (b).
	The asymptotic analysis for the lower critical exponent, $s=2_{**}$ exhibits its subtlety. First, since $\gamma(2_{**})=n-4$, one would expect the singular solutions to \eqref{eq:subcriticalsystem} to have the same behavior as the fundamental solution to the bi-Laplacian near the origin; thus, by classical results of J. Serrin, the isolated singularity would be removable.
	However, the results in \cite{MR875297} suggest that there exists a more refined asymptotic profile.
	In the lower critical case $s=2_{**}$ (resp. $s=2_*$) in Theorems~\ref{Thm3.2:asymptotics} and \ref{thm:lin-soranzo-yang-frank-konig-ratzkin} (resp. Theorem~\ref{thm:druet-hebey-vetois-ghergu-kim-shahgohlian-caju-doo-santos} and \ref{thm:serrin-lions-aviles-caffarelli-gidas-spruck}), a more accurate type of cylindrical transformation shall be considered.
	Here, we also recall that $\gamma(2^{**}-1)=\frac{n-4}{2}$.
	Our objective is to prove the proposition below
	
	\begin{proposition}\label{prop:avilescase}
		Let $R<\infty$, $s=2_{**}$, and $\mathcal{U}$ be a nonnegative superharmonic singular solution to \eqref{eq:subcriticalsystem}. 
		Then, 
		\begin{equation*}
		|\mathcal{U}(x)|=(1+o(1))\widehat{K}_{0}(n)^{\frac{n-4}{4}}|x|^{4-n}(-\ln|x|)^{\frac{n-4}{2}},
		\end{equation*}
		where $\widehat{K}_{0}(n):=\lim_{t\rightarrow\infty}t\widetilde{K}_0(n,t)=\frac{(n-4)(n-2)(n+4)}{2}.$
	\end{proposition}
	
	First, let us mention that the motivation to consider this transformation is the following asymptotic upper bound for singular solutions to \eqref{eq:subcriticalsystem}.
	
	\begin{lemma}\label{lm:sharpestimate}
		Let $R<\infty$, $s=2_{**}$, and $\mathcal{U}$ be a nonnegative solution to \eqref{eq:subcriticalsystem}.
		Then, then there exist $C_0(n)>0$ and $0<r_0<R$ such that
		\begin{equation*}
		|\overline{u}_{i}(x)|\leqslant C_0(n)|x|^{4-n}(-\ln|x|)^{\frac{4-n}{4}} \quad {\rm for} \quad 0<|x|<r_0 \quad {\rm and} \quad i\in I,
		\end{equation*}
		where $\overline{u}_{i}$ is the spherical average of $u_{i}$ over the sphere $\partial B_{R}$.
	\end{lemma}
	
	\begin{proof}
		Note that for each $i\in I$, we get that $\overline{u}_{i}$ satisfies, 
		\begin{equation*}
		\partial_r^{(4)}\overline{u}_{i}+\frac{2(n-1)}{r}\partial_r^{(3)}\overline{u}_{i}+\frac{(n-1)(n-3)}{r^2}\partial_r^{(2)}\overline{u}_{i}-\frac{(n-1)(n-3)}{r^3}\partial_r\overline{u}_{i}-\overline{u}_{i}^{2_{**}}=0,
		\end{equation*}
		for $0<r<R$. Whence the conclusion follows directly from \cite[Theorem~5]{MR1436822}.
	\end{proof}
	
	As in the autonomous case, we use the limiting energy levels $\widetilde{\mathcal{P}}_{\rm cyl}(\infty,\mathcal{W})$ to classify the local behavior near the isolated singularity. 
	
	\begin{lemma}\label{lm:limitinglevels}
		Let $R<\infty$, $s=2_{**}$, and $\mathcal{U}$ be a nonnegative solution to \eqref{eq:subcriticalsystem} with $\mathcal{W}=\widetilde{\mathfrak{F}}(\mathcal{V})$ its nonautonomous cylindrical transform given by \eqref{newcylindrical}. 
		Then, $\widetilde{\mathcal{P}}_{\rm cyl}(\infty,\mathcal{W}) \in\left\{-l^*(n), 0\right\}$, where
		\begin{equation*}
		l^*(n)=\frac{2^{\frac{n-8}{n-4}}(n-4)\left[(n-2)(n^2-16)\right]^{\frac{2(n-2)}{n-4}}+(n-2)^5(n^2-16)^4}{16(n-2)}.
		\end{equation*}
		Moreover, it follows\\
		\noindent{\rm (i)} $\widetilde{\mathcal{P}}_{\rm cyl}(\infty,\mathcal{W})=0$ if, and only if,
		\begin{equation}\label{asymptoticsaviles}
		|\mathcal{U}(x)|=o\left(|x|^{4-n}(-\ln |x|)^{\frac{4-n}{4}}\right) \quad \mbox{as} \quad x \rightarrow 0.
		\end{equation}
		\noindent{\rm (ii)} $\widetilde{\mathcal{P}}_{\rm cyl}(\infty,\mathcal{W})=l^*(n)$ if, and only if,
		\begin{equation*}
		|\mathcal{U}(x)|=(1+o(1))\widehat{K}_{0}(n)^{\frac{n-4}{4}}|x|^{4-n}(-\ln |x|)^{\frac{4-n}{4}} \quad \mbox{as} \quad x \rightarrow 0.
		\end{equation*}
	\end{lemma}
	
	\begin{proof}
		First, combining \eqref{angularestimates} with Proposition~\ref{lm:lowermonotonicity} and Lemma~\ref{lm:estimateangularparts}, we find
		\begin{equation*}
		\widetilde{\mathcal{P}}_{\rm cyl}(\infty, \mathcal{W})=\lim _{t \rightarrow \infty} \int_{\mathbb{S}_t^{n-1}}\left((2_{**}+1)^{-1}|\mathcal{W}|^{2_{**}+1}+\widehat{K}_{0}(n)|\mathcal{W}|^{2}\right)\ud\theta.
		\end{equation*}
		Furthermore, by \eqref{angularestimates}, we see that for any $\{t_k\}_{k\in\mathbb{N}}$ such that $t_{k} \rightarrow \infty$ as $k \rightarrow \infty$, it follows that $\{\mathcal{W}(t_{k}, \theta)\}_{k\in\mathbb{N}}$ converges to a limit, which is independent of $\theta \in \mathbb{S}_t^{n-1}$. Hence, up to subsequence, there exists $\Lambda\in\mathbb{R}^p$ such that $\mathcal{W}\left(t_{k}, \theta\right)\rightarrow \Lambda$ (uniformly on $\theta \in \mathbb{S}_t^{n-1}$ ), which gives us
		\begin{equation}\label{bestconstant}
		\widetilde{\mathcal{P}}_{\rm cyl}(\infty, \mathcal{W})=(2_{**}+1)^{-1}|\Lambda|^{2_{**}+1}+\widehat{K}_{0}(n)|\Lambda|^{2}.
		\end{equation}
		Thus, since the right-hand side of the last equation has at most three nonnegative roots, the limit $|\Lambda|$, under the uniform convergence of $|\mathcal{W}(t, \theta)|$ on $\mathbb{S}_t^{n-1}$ as $t \rightarrow \infty$, is unique.
		Finally, taking the inner product of \eqref{eq:subcriticalcylindricalsystemlower} with $\mathcal{W}$, integrating both sides over $\left(t_{0}, \infty\right) \times \mathbb{S}^{n-1}$, and using  \eqref{angularestimates},\eqref{radialestimates} and Lemma~\ref{lm:estimateangularparts}, it follows
		\begin{equation*}
		\left|\int_{t_{0}}^{\infty} \frac{1}{t} \int_{\mathbb{S}_t^{n-1}}\left(\widehat{K}_0(n)-|\mathcal{W}|^{2_{**}-1}\right) |\mathcal{W}|^{2}\ud \theta\ud t\right|<\infty.
		\end{equation*}
		Now since $\lim_{t \rightarrow \infty}|\mathcal{W}(t,\theta)|=|\Lambda|$ uniformly on $\mathbb{S}_t^{n-1}$, we get either $|\Lambda|=0$ or $|\Lambda|=\widehat{K}_0(n)^{\frac{n-4}{4}}$, which by substituting into \eqref{bestconstant}, implies that either $\widetilde{\mathcal{P}}_{\rm cyl}(\infty, \mathcal{W})=0$ if, and only if, $|\Lambda|=0$, or $\widetilde{\mathcal{P}}_{\rm cyl}(\infty, \mathcal{W})=l^*(n)$, otherwise.
		The proof trivially follows by applying the inverse $\widetilde{\mathfrak{F}}^{-1}$ of the nonautonomous cylindrical transform.
	\end{proof}
	
	Now we are left to show that, if (i) of Lemma~\ref{lm:limitinglevels} holds, then the singularity at the origin is removable.
	Here, we are based on the barriers construction in \cite{MR2055032} (see also \cite{arxiv:1901.01678}), which is available due to the integral representation \eqref{integralsystem}.
	
	\begin{lemma}\label{lm:removablesingularityaviles}
		Let $R<\infty$, $s=2_{**}$, and $\mathcal{U}$ be a nonnegative solution to \eqref{eq:subcriticalsystem}. 
		Suppose that
		\begin{equation*}
		|\mathcal{U}(x)|=o\left(|x|^{4-n}(-\ln |x|)^{\frac{4-n}{4}}\right) \quad \mbox{as} \quad x \rightarrow 0.
		\end{equation*}
		Then, the origin is a removable singularity.
	\end{lemma} 
	
	\begin{proof}
		For any $i\in I$ and $\delta>0$, we choose $0<\rho\ll 1$ such that $u_i(x)\leqslant\delta|x|^{-\gamma(s)}$ in $B^*_{\rho}$.
		Fixing $\varepsilon>0$, $\kappa\in\left(0, \gamma(s)\right)$ and $M\gg1$ to be chosen later, we define
		\begin{equation*}
		\varsigma_i(x)=
		\begin{cases}
		{M|x|^{-\kappa}+\varepsilon|x|^{4-n-\kappa},} & \mbox{if} \ {0<|x|<\rho}\\ 
		{u_i(x)},& \mbox{if} \ {\rho<|x|<2}.
		\end{cases}
		\end{equation*}
		Notice that for every $0<\kappa<n-4$ and $0<|x|<2$, by a change a variables, there exists $C>0$ such that
		\begin{align*}
		\int_{\mathbb{R}^{n}}{|x-y|^{4-n}|y|^{-4-\kappa}}\ud y 
		&=|x|^{4-n}\int_{\mathbb{R}^{n}}{\left||x|^{-1}x-|x|^{-1}y\right|^{4-n}|y|^{\-\kappa-4}}\ud y&\\
		&=|x|^{-\kappa+4}\int_{\mathbb{R}^{n}} {\left||x|^{-1}x-z\right|^{4-n}|z|^{\-\kappa-4}}\ud z&\\ 
		&\leqslant C\left(\frac{1}{n-4-\kappa}+\frac{1}{\kappa}+1\right)|x|^{-\kappa},
		\end{align*}
		which, for $0<|x|<2$ and $0<\delta\ll1$, yields 
		\begin{align*}
		\int_{B_{\rho}}{u_i^{2_{**}-1}(y)\varsigma_i(y)}{|x-y|^{4-n}}\ud y
		&\leqslant \delta^{2_{**}-1} \int_{\mathbb{R}^{n}}{\varsigma_i(y)}{|x-y|^{n-4}|y|^{-4}}\ud y&\\
		&\leqslant C\delta^{2_{**}-1} \varsigma_i(x)&\\ 
		&<\frac{1}{2} \varsigma_i(x).&
		\end{align*}
		Moreover, for $0<|x|<\rho$ and $\bar{x}=\rho x|x|^{-1}$, we get
		\begin{align*}
		\int_{B_{2}\setminus B_{\rho}}{u_i^{2_{**}-1}(y)\varsigma_i(y)}{|x-y|^{4-n}}\ud y
		&=\int_{B_{2}\setminus B_{\rho}}\frac{|\bar{x}-y|^{n-4}}{|x-y|^{n-4}}\frac{u_i^{2_{**}}(y)}{|\bar{x}-y|^{n-4}}\ud y&\\
		&\leqslant 2^{n-4}\int_{B_{2}\setminus B_{\rho}}\frac{u_i^{2_{**}}(y)}{|\bar{x}-y|^{n-4}}\ud y&\\ 
		&\leqslant 2^{n-4} u_i(\bar{x})&\\
		&\leqslant 2^{n-4}\max_{\partial B_{\rho}} u_i.
		\end{align*}
		The last inequality implies that for $0<|x|<\tau$ and $M\geqslant\max_{\partial B_{\rho}}u_i$,
		\begin{equation*}
		\psi_i(x)+\int_{B_{2}}\frac{u_i^{2_{**}-1}(y)\varsigma_i(y)}{|x-y|^{4-n}}\ud y \leqslant \psi_i(x)+2^{n-4}\max_{\partial B_{\rho}}u_i+\frac{1}{2}\varsigma_i(x)<\varsigma_i(x).
		\end{equation*}
		
		In the next claim, we show that $\varsigma_i$ can be taken indeed as a barrier for any $u_i$.
		
		\noindent{\bf Claim 1:} For any $i\in I$, it holds that $u_i(x)\leqslant \varsigma_i(x)$ in $B^*_{\rho}$. 
		
		\noindent In fact, suppose by contradiction that the conclusion is not true. Then, since $u_i(x)\leqslant\delta|x|^{-\gamma}(s)$ in $B^*_{\rho}$, by the definition of $\varsigma_i$, there exists $\widetilde{\tau} \in(0, \rho)$, depending on $\varepsilon$, such that $\varsigma_i\geqslant u_i$ in $B^*_{{\rho}}$ and $\varsigma_i>u_i$ close to the boundary $\partial B_{\rho}$. Let us consider, 
		\begin{equation*}
		\bar{\tau}:=\inf\left\{\tau>1 : \tau\psi_{i}>u_i \ \mbox{in} \ B^*_{\rho}\right\}.
		\end{equation*}
		Then, we have that $\bar{\tau}\in(1,\infty)$ and there exists $\bar{x}\in B_{\rho} \setminus\bar{B}_{\widetilde{\tau}}$ such that $\bar{\tau}\varsigma_i(\bar{x})=u_i(\bar{x})$ and, for    $0<|x|<\tau$, it follows 
		\begin{equation*}
		\bar{\tau}\varsigma_i(x)\geqslant\int_{B_{2}}{u_i^{2_{**}-1}(y)\bar{\tau} \varsigma_i(y)}{|x-y|^{4-n}}\ud y+\bar{\tau}\psi_i(x)\geqslant\int_{B_{2}}{u_i^{2_{**}-1}(y)\bar{\tau}\varsigma_i(y)}{|x-y|^{4-n}}\ud y+\psi_i(x),
		\end{equation*}
		which gives us
		\begin{equation*}
		\bar{\tau}\varsigma_i(x)-u_i(x)\geqslant\int_{B_{2}}{u_i^{2_{**}-1}(y)(\bar{\tau} \varsigma_i(y)-u_i(y))}{|x-y|^{4-n}}\ud y.
		\end{equation*}
		Finally, by evaluating the last inequality at $\bar{x}\in B_{\rho}\setminus\bar{B}_{\widetilde{\tau}}$, we get a contradiction and the claim is proved.
		
		Consequently, we find
		$u_i(x)\leqslant\varsigma_i(x)\leqslant M|x|^{-\kappa}+\varepsilon|x|^{4-n-\kappa}$ in $B^*_{\rho}$,
		which yields that $u_i^{2_{**}-1}\in L^{s}(B^*_{\rho})$ for some $s>{n}/{4}$ and any $i\in I$. Therefore, standard elliptic regularity concludes the proof of the lemma.
	\end{proof}
	
	Ultimately, the proof of the main result in this section is merely a consequence of the last results.
	
	\begin{proof}[Proof of Proposition~\ref{prop:avilescase}]
		Suppose that $\mathcal{U}$ has a non-removable singularity at the origin, then by Lemma~\ref{lm:removablesingularityaviles}, $\mathcal{U}$ does not satisfy \eqref{asymptoticsaviles}. Therefore, the proof follows as a consequence of Lemma~\ref{lm:limitinglevels}.
	\end{proof}
	
	\appendix
	
	\section{The general Emden--Fowler change of coordinates}\label{app:computations}
	In this appendix, using the software Mathematica 12, we compute the coefficients of the bi-Laplacian written in cylindrical coordinates (see also \cite{MR4094467,MR4123335}). More generally, let us consider the following change of coordinates 
	\begin{equation}\label{eq:generalizedchangeofvariables}
	u(r)=\rho(r)v(t) \quad \mbox{with} \quad t=\psi(r), 
	\end{equation}
	where $\rho,\psi:\mathbb{R}\rightarrow\mathbb{R}$ are smooth functions, and $\psi$ is a smooth diffeomorphism. Here we adopt the notations $\psi_r={\ud\psi}/{\ud r}$, $\rho_r={\ud\rho}/{\ud r}$, $\psi^{(j)}_r={\ud^{j}\psi}/{\ud r^{j} }$, and $\rho^{(j)}_r={\ud^{j}\rho}/{\ud r^{j} }$ $\partial_{t}^{(j)}={\partial}^{j}/{\partial t}^{j}$ (resp. $\partial_{r}^{(j)}={\partial}^{j}/{\partial r}^{j}$) with the convention $\partial_t^{(0)}$ equals the identity operator on $C^{\infty}(\mathbb{R})$, and we omit $u,v$ when it is convenient.
	Now the idea is to express the operator $\partial_r^{(j)}$ for $j\in\mathbb{N} $ in terms of $\partial_t^{(\ell)}$ for $\ell=1,\dots, j$, that is,
	\begin{equation*}
	\partial_r^{(j)}=\sum_{\ell=0}^jc_{j\ell}(\rho,\psi)\partial_t^{(\ell)},
	\end{equation*}
	where $c_{j\ell}:\mathbb{R}^{2(\ell+1)+1}\rightarrow\mathbb{R}$ are the coefficient functions, depending on $\rho,\psi$ and all theirs derivative until $\ell$-th order. 
	Notice that $C=(c_{j\ell})_{j\ell}$ is a lower triangular matrix, {\it i.e.}, $c_{j\ell}\equiv0$ when $j<\ell$. In fact, a direct computation shows
	\begin{align}\label{eq:decompositionderivative}
	&\partial_r^{(0)}=\rho\partial_t^{(0)}&\\\nonumber
	&\partial_r^{(1)}=\rho_r\partial_t^{(0)}+\psi_r\rho\partial_t^{(1)}&\\\nonumber
	&\partial_r^{(2)}=\rho_r^{(2)}\partial_t^{(0)}+(2\psi_r\rho_r+ \psi_r^{(2)}\rho)\partial_t^{(1)}+ \psi_{r}^2\rho\partial_t^{(2)}&\\\nonumber
	&\partial_r^{(3)}=\rho_r^{(3)}\partial_t^{(0)}+(3\psi_r\rho_r^{(2)}+3\psi_r^{(2)}\rho_r+\psi_r^{(3)}\rho)\partial_t^{(1)}+(3\psi_r^2\rho_{r}+3\psi_r\psi_r^{(2)}\rho)\partial_t^{(2)}+\psi_{r}^3\rho \partial_t^{(3)}&\\\nonumber
	&\partial_r^{(4)}=\rho_r^{(4)}\partial_t^{(0)}+[4\psi_{r}\rho_r^{(3)}+6\psi_r^{(2)}\rho_r^{(2)}+4\psi_r^{(3)}\rho_r+\psi_r^{(4)}\rho]\partial_t^{(1)}\\\nonumber
	&\quad\quad+[6{\psi_{r}}^2\rho_{rr}+12\psi_r\psi_r^{(2)}\rho_{r}+(3{\psi_r^{(2)}}^2+4\psi_r\psi_r^{(3)})\rho]\partial_r^{(2)}+(4\psi_r^3\rho_{r}+6\psi_r^2\psi_r^{(2)}\rho)\partial_r^{(3)}+\psi_{r}^4\rho \partial_r^{(4)}.&
	\end{align}
	In particular, when $\psi(r)=-\ln r$, it follows 
	\begin{equation}\label{eq:derivativeslog}
	\psi_r(r)=-r^{-1}, \quad \psi^{(2)}_r(r)=r^{-2}, \quad \mbox{and} \quad \psi^{(3)}_r(r)=-2r^{-3},
	\end{equation}
	which turns \eqref{eq:generalizedchangeofvariables} into the classical logarithm cylindrical change of coordinates. 
	Also, we observe that the choice $\psi(r)=\ln r$ would lead to a change of sign.
	For the geometrical point of view, the contrary sign choice is more natural.
	
	In what follows, to distinct the second and fourth order cases, we fixed double indexes.
	Nevertheless, in the main text, we only consider one index, since we  always deals with the fourth order case, that is, $K_0,K_1,K_2,K_3,J_0,J_1,J_2,$ (or $\widetilde{K}_0(t),\widetilde{K}_1(t),\widetilde{K}_2(t),\widetilde{K}_3(t),\widetilde{J}_0(t),\widetilde{J}_1(t),\widetilde{J}_2(t)$ in the nonautonomous case). 
	The same holds for the Fowler rescaling exponent $\gamma_{m}(s):=\frac{m}{s-1}$ for $m\geqslant1$, where $m$ is dropped since we always consider $m=4$.
	
	\subsection{Second order case}
	Now remember the expression for the Laplacian in spherical (polar) coordinates,
	\begin{align*}
	\Delta_{\rm sph}&=r^{-2}N_{20}(n)\partial_r^{(0)}+r^{-1}N_{21}(n)\partial_r^{(1)}+N_{22}(n)\partial_r^{(2)}+r^{-2}M_{20}(n)\Delta_{\sigma},
	\end{align*}
	where the coefficients are given by
	\begin{align*}
	N_{20}(n)=0,\quad 
	N_{21}(n)=n-1, \quad 
	N_{22}(n)=1\quad \mbox{and} \quad M_{20}(n)=1.
	\end{align*}
	Using the change coordinates \eqref{eq:generalizedchangeofvariables}, we get
	\begin{align*}
	\Delta_{\rm cyl}&=K_{20}(\rho,\psi)\partial_t^{(0)}+K_{21}(\rho,\psi)\partial_t^{(1)}+ K_{22}(\rho,\psi)\partial_t^{(2)}+r^{-2}J_{20}\Delta_{\theta},
	\end{align*}
	where $K_{2\ell}(\rho,\psi)=\sum_{j=0}^2N_\ell(n)c_{j\ell}(\rho,\psi)$. More explicitly, we have
	\begin{align*}
	&K_{20}(\rho,\psi)=N_{22}(n)c_{20}(\rho,\psi)+N_{21}(n)c_{10}(\rho,\psi)+N_{20}(n)c_{00}(\rho,\psi)&\\
	&K_{21}(\rho,\psi)=N_{22}(n)c_{21}(\rho,\psi)+N_{21}(n)c_{11}(\rho,\psi)&\\
	&K_{22}(\rho,\psi)=N_{22}(n)c_{22}(\rho,\psi),&\\
	&J_{20}(\rho,\psi)=M_{20}(n),&
	\end{align*}
	where, by using \eqref{eq:decompositionderivative}, we can explicitly get the matrix of coefficients in the second order case,
	\begin{equation}\label{eq:secondordercoefficients}
	c_{j\ell}(\rho,\psi)=
	\left[{\begin{array}{ccc}
		\rho & 0 & 0  \\
		\rho & \psi_r\rho & 0 \\
		\rho_{rr} & 2\psi_r\rho_r +\psi_{rr}\rho  & \psi_{r}^2 \rho \\
		\end{array} } \right].
	\end{equation}
	The choice of $\rho(r)$ will depend on the natural scaling of the problem. Namely, for $s\in(1,2^{*}]$, we consider $\rho(r)=r^{-\gamma_2(s)}$, where $\gamma_2(s)=2/(s-1)$.
	Notice that $-\gamma_2(s)$ is the constant solution to te functional equation $\gamma_2(s)-2=(s-1)\gamma_2(s)$, which is obtained by a scaling analysis of \eqref{eq:secondorderode}.
	Whence,
	\begin{align*}
	&\rho_r(r)=-\gamma_2r^{-(\gamma_2+1)},\quad \mbox{and} \quad \rho^{(2)}_r(r)=\gamma_2(\gamma_2+1)r^{-(\gamma_2+2)}.&
	\end{align*}
	This turns the change of variables \eqref{eq:generalizedchangeofvariables} into the (second order) Emden-Fowler transformation from \cite{fowler,MR982351}, which, by substituting in \eqref{eq:secondordercoefficients}, provides \eqref{eq:secondorderode}, that is,
	\begin{align*}
	&K_{20}(n,s)={2}(s-1)^{-2}\left[n(s-1)-2s\right], \  K_{21}(n,s)=-{(s-1)^{-1}}\left[s(n-2)+n+2\right],\\
	&K_{22}(n,s)=1 \ \mbox{and} \ J_{20}(n,s)=1&
	\end{align*}
	Then, defining $K^*_{2j}:=K_{2j}(n,2^{*}-1)$, one finds
	\begin{equation*}
	K^*_{22}\equiv1, \quad K^*_{21}\equiv0, \quad 
	K^*_{20}=-\frac{(n-2)^2}{4}, \quad {\rm and} \quad J^*_{20}=1,
	\end{equation*}
	whereas, when $K_{2j,*}:=K_{2j}(n,2_{*})$, it holds
	\begin{equation*}
	K^*_{22}\equiv1, \quad K^*_{21}=n-2, \quad 
	K^*_{20}\equiv0, \quad {\rm and} \quad J^*_{20}=1,
	\end{equation*}
	Thus, since $K^*_{20}\equiv0$, when $s=2_{*}$, one needs to consider a more suitable change of coordinates. In this fashion, we take $\rho(r)=r^{2-n}\psi(r)^{\frac{2-n}{2}}$, then  
	\begin{align*}
	&\rho_r(r)=r^{-(2\gamma_2+1)}\psi^{-(\gamma_2+1)}\left(-\gamma_2\psi_r^{(1)}r-2\gamma_2\psi\right),\quad \mbox{and}&\\
	&\rho^{(2)}_r(r)=r^{-(2\gamma_2+2)}\psi^{-(\gamma_2+2)}\left\{\gamma_2(4\gamma_2+2)\psi^2+4\gamma_2\psi\psi_r^{(1)}r+\left[\gamma_2(\gamma_2+1){\psi_r^{(1)}}^2-\psi{\psi_r^{(2)}}\right]r^2\right\},&
	\end{align*}
	and, similarly to the other cases, leads to \eqref{eq:nonautonomoussecondorder}, that is,
	\begin{align*}
	\widetilde{K}_{20}(n,t)=\frac{n(n-2)}{4t^2}-\frac{(n-2)^2}{2t}, \ \widetilde{K}_{21}(n,t)=-\frac{(n-2)}{t}+(n-2), \
	\widetilde{K}_{22}(n,t)=1 \ \mbox{and} \ \widetilde{J}_{20}(n,t)=1,
	\end{align*}
	where each term was simplified by the factor $r^{-n}\psi(r)^{\frac{2-n}{2}}$.
	We should emphasize, that contrarily \eqref{eq:secondorderode}, this is a nonautonomous ODE; this makes the analysis harder.
	
	\subsection{Fourth order case}
	Analogously, we consider the bi-Laplacian in spherical (polar) coordinates,
	\begin{align*}
	\Delta^2_{\rm sph}&=r^{-4}N_{40}(n)\partial_r^{(0)}+r^{-3}N_{41}(n)\partial_r^{(1)} r^{-2}N_{42}(n)\partial_r^{(2)}+r^{-1}N_{43}(n)\partial_r^{(3)}+N_{44}(n)\partial^{(4)}_r&
	\\\nonumber
	&+r^{-4}M_{40}(n)\Delta_{\sigma}+r^{-3}M_{41}(n)\partial^{(1)}_r\Delta_{\sigma}+r^{-2}M_{42}(n)\partial^{(2)}_r\Delta_{\sigma}+r^{-4}M_{44}(n)\Delta_{\sigma}^2,&
	\end{align*}
	where the coefficients are given by
	\begin{align*}
	&N_{40}(n)=0,\quad 
	N_{41}(n)=2(n-1), \quad N_{42}(n)=(n-1)(n-3), \quad N_{43}(n)=-(n-1)(n-3), \quad
	N_{44}(n)=1,&\\
	&M_{40}(n)=-2(n-4),\quad 
	M_{41}(n)=2(n-3), \quad 
	M_{42}(n)=2, \quad \mbox{and} \quad
	M_{44}(n)=1.&
	\end{align*}
	Using the change coordinates \eqref{eq:generalizedchangeofvariables}, after a suitable rescaling, we get
	\begin{align*}
	\Delta^2_{\rm cyl}&=K_{40}(\rho,\psi)\partial_t^{(0)}+K_{41}(\rho,\psi)\partial_t^{(1)}+ K_{42}(\rho,\psi)\partial_t^{(2)}+K_{43}(\rho,\psi)\partial_t^{(3)}+K_{44}(\rho,\psi)\partial^{(4)}_t,&\\\nonumber
	&+J_{40}(\rho,\psi)\Delta_{\theta}+J_{41}(\rho,\psi)\partial^{(1)}_t\Delta_{\theta}+J_{42}(\rho,\psi)\partial^{(2)}_t\Delta_{\theta}+J_{44}(\rho,\psi)\Delta_{\theta}^2,&
	\end{align*}
	where $K_{4\ell}(\rho,\psi)=\sum_{j=0}^4N_\ell(n)c_{j\ell}(\rho,\psi)$. More explicitly, we find
	\begin{align*}
	&K_{40}(\rho,\psi)=N_{44}(n)c_{40}(\rho,\psi)+N_{43}(n)c_{30}(\rho,\psi)+N_{42}(n)c_{20}(\rho,\psi)+N_{41}(n)c_{10}(\rho,\psi)+N_{40}(n)c_{00}(\rho,\psi)&\\
	&K_{41}(\rho,\psi)=N_{44}(n)c_{41}(\rho,\psi)+N_{43}(n)c_{31}(\rho,\psi)+N_{42}(n)c_{21}(\rho,\psi)+N_{41}(n)c_{11}(\rho,\psi)&\\
	&K_{42}(\rho,\psi)=N_{44}(n)c_{42}(\rho,\psi)+N_{43}(n)c_{32}(\rho,\psi)+N_{42}(n)c_{22}(\rho,\psi)\\
	&K_{43}(\rho,\psi)=N_{44}(n)c_{43}(\rho,\psi)+N_{43}(n)c_{33}(\rho,\psi)&\\
	&K_{44}(\rho,\psi)=N_{44}(n)c_{44}(\rho,\psi),&\\
	&J_{40}(\rho,\psi)=M_{42}(n)c_{20}(\rho,\psi)+M_{41}(n)c_{10}(\rho,\psi)+M_{40}(n)c_{00}(\rho,\psi)&\\
	&J_{41}(\rho,\psi)=M_{42}(n)c_{21}(\rho,\psi)+M_{41}(n)c_{11}(\rho,\psi)&\\
	&J_{41}(\rho,\psi)=M_{42}(n)c_{22}(\rho,\psi)&\\
	&J_{44}(\rho,\psi)=M_{44}(n).&
	\end{align*}
	Hence, by following, \eqref{eq:decompositionderivative}, we find
	\begin{align}\label{eq:coefficients}
	&c_{00}=\rho&\\\nonumber
	&c_{10}=\rho_r, \quad c_{11}=\psi_r\rho &\\\nonumber
	&c_{20}=\rho_r^{(2)}, \quad c_{21}=2 \psi_r\rho_r+\psi_r^{(2)}\rho, \quad c_{22}=\psi_{r}^2\rho &\\\nonumber
	&c_{30}=\rho_r^{(3)}, \quad c_{31}=3\psi_r\rho_r^{(2)}+3\psi_r^{(2)}\rho_r+\psi_r^{(3)}\rho, \quad c_{32}=3\psi_r^2\rho_{r}+3\psi_r\psi_r^{(2)}\rho, \quad c_{33}=\psi_{r}^3\rho&\\\nonumber
	&c_{40}=\rho_r^{(4)}, \quad c_{41}=4\psi_{r}\rho_r^{(3)}+6\psi_r^{(2)}\rho_r^{(2)}+4\psi_r^{(3)}\rho_r+\psi_r^{(4)}\rho,&\\\nonumber
	&c_{42}=6\psi_r^2\rho_r^{(2)}+(3\psi_r^{(2)}+12\psi_r\psi_r^{(2)})\rho_{r}+(3\psi_r^2+4\psi_r\psi_r^{(3)})\rho, \quad c_{43}=4\psi_r^3\rho_{r}+6\psi_r^2\psi_r^{(2)}\rho,
	\quad c_{40}= \psi_{r}^4\rho.&
	\end{align}
	
	For the Bi-Laplacian, we consider a suitable scaling function given by $\rho(r)=r^{{-\gamma}_4(s)}$ with $\gamma_{4}(s)=4/(s-1)$. Hence, the change of variables \eqref{eq:generalizedchangeofvariables} becomes the (fourth order) Emden-Fowler transformation from Subsection~\ref{subsec:cylindrcialtransform}, that is,
	\begin{align}\label{eq:derivativescaling}
	&\rho_r(r)=-\gamma_4r^{-(\gamma_4+1)}, \quad \rho^{(2)}_r(r)=\gamma_4(\gamma_4+1)r^{-(\gamma_4+2)},&\\\nonumber
	&\rho^{(3)}_r(r)=-\gamma_4(\gamma_4+1)(\gamma_4+2)r^{-(\gamma_4+3)},\quad \mbox{and} \quad \rho^{(4)}_r(r)=\gamma_4(\gamma_4+1)(\gamma_4+2)(\gamma_4+3)r^{-(\gamma_4+4)}.&
	\end{align}
	Therefore, by substituting \eqref{eq:derivativescaling} and \eqref{eq:derivativeslog} into \eqref{eq:coefficients}, and simplifying by the common scaling factor $r^{\gamma_4-4}$, which is chosen so it satisfies  the functional equation $\gamma(s)-4=(s-1)\gamma(s)$, we arrive at 
	\begin{align*}
	&K_{40}(n,s)={8}(s-1)^{-4}\left[(n-2)(n-4)(s-1)^{3}+2\left(n^{2}-10 n+20\right)(s-1)^{2}
	-16(n-4)(s-1)+32\right],&\\\nonumber
	&K_{41}(n,s)=-{2}{(s-1)^{-3}}\left[(n-2)(n-4)(s-1)^{3}+4\left(n^{2}-10 n+20\right)(s-1)^{2}-48(n-4)(s-1)+128\right],&\\\nonumber
	&K_{42}(n,s)={(s-1)^{-2}}\left[\left(n^{2}-10 n+20\right)(s-1)^{2}-24(n-4)(s-1)+96\right],&\\\nonumber
	&K_{43}(n,s)={2}{(s-1)}^{-1}[(n-4)(s-1)-8],&\\\nonumber
	&K_{44}(n,s)=1,&\\\nonumber
	&J_{40}(n,s)={2}{(s-1)^{-2}}\left[(s+1)^2(s-1)^{2}-n(s-3)(s-1)\right],&\\\nonumber
	&J_{41}(n,s)={2}{(s-1)^{-1}}\left[(n-4)(s-1)+16\right].&,&\\\nonumber
	&J_{42}(n,s)=2,&\\\nonumber
	&J_{44}(n,s)=1.&
	\end{align*} 
	Now, considering $K^*_{4j}:=K_{4j}(n,2^{**}-1)$, one gets
	\begin{equation*}
	K^*_{41}\equiv K^*_{43}\equiv J^*_{41}\equiv0, \quad 
	K^*_{40}=\frac{n^2(n-4)^2}{16}, \quad K^*_{42}=-\frac{n^2-4n+8}{2}, \quad {\rm and} \quad J^*_{40}=-\frac{n(n-4)}{2},
	\end{equation*}
	and, for $K_{4j,*}:=K_{4j}(n,2_{**})$, it follows
	\begin{align*}
	&K_{40,*}\equiv0, \quad K_{41,*}=2(n-4)(n+2), \quad K_{42,*}=n^2-10n+20,&\\ 
	&K_{43,*}=J_{41,*}=2(n-4), \quad {\rm and} \quad J_{40,*}=-2(n-4).&
	\end{align*}  
	Again, the zeroth order coefficient vanishes thus, one needs to consider a more suitable change of coordinates given by the scaling factor $\rho(r)=r^{4-n}\psi(r)^{\frac{4-n}{4}}$, which as before provides the following nonautonomous coefficients,
	\begin{align*}
	&\nonumber\widetilde{K}_{40}(n,t)=\frac{(n-4)n(n+4)(n+8)}{256t^4}-\frac{(n-4)^2n(n+4)}{32t^3}+\frac{(n-4)n(n^2-10n+20)}{16t^2}+\frac{(n-4)(n-2)(n+4)}{t},&\\\nonumber &\widetilde{K}_{41}(n,t)=\frac{(n-4)n(n+4)}{16t^3}+\frac{3n(n-4)}{8t^2}+\frac{(n-4)(n^2-10n+20)}{2t}-2(n-4)(n-2),&\\\nonumber
	&\widetilde{K}_{42}(n,t)=\frac{3n(n-4)}{8t^2}-\frac{3(n-4)^2}{2t}+n^2-10n+20,&\\
	&\widetilde{K}_{43}(n,t)=\frac{n-4}{t}+2(n-4),&\\\nonumber
	&\widetilde{K}_{44}(n,t)=1,&\\\nonumber
	&\widetilde{J}_{40}(n,t)=\frac{n(n-4)}{8t^2}-\frac{(n-4)^2}{2t}-2(n-4),&\\\nonumber
	&\widetilde{J}_{41}(n,t)=-\frac{n-4}{t}+2(n-4),&\\\nonumber
	&\widetilde{J}_{42}(n,t)=2,&\\\nonumber
	&\widetilde{J}_{44}(n,t)=1.&
	\end{align*}
	where each term is simplified by the factor $r^{-n}\psi(r)^{\frac{4-n}{4}}$.
	
	\begin{remark}	
		It would also be interesting to run the same analysis for the higher order derivative case $m\geqslant6$.
		Indeed, computing the $m$th order matrices  $K_{j\ell}(n,s),J_{j\ell}(n,s)$ ($\widetilde{K}_{j\ell}(n,t)$,$\widetilde{J}_{j\ell}(n,t)$) and studying the sign of their components would allow us to provide the classification and the local behavior near the isolated singularity for the general (even order) polyharmonic system,
		\begin{equation}\label{higherordersubcriticalsystem}
		(-\Delta)^{k} u_{i}=c(n,s,k)|\mathcal{U}|^{s-1}u_{i} \quad {\rm in} \quad B_R^* \quad {\rm for} \quad i\in I.
		\end{equation}
		Here $s\in(1,2_k^{*}-1]$, where $2_k^*:=2n/(n-2k)$ is the (upper critical Sobolev exponent) with $n>2k$, $m:=2k$ and $c(n,s,k)$ is a normalizing constant with geometrical meaning. 
		In an upcoming work, we make a systematic study of this system.
	\end{remark}

	\section{Another moving spheres technique}\label{app:alternativamovingspheres}
	In \cite[Section~3.6]{arXiv:2002.12491}, we utilize a different moving spheres technique, which is based on \cite[Proposition~1.1]{MR2558186} and \cite[Theorem~1.3]{MR1611691}.
	Let us emphasize that these techniques could also be employed in this subcritical situation. 
	However, we choose to perform this sliding technique using its integral form, which works in a more general setting, namely, for higher order equations.
	
	\begin{proof}[Another proof for Lemma~\ref{lm:movingspheres}]
		Without loss of generality, we consider $z=0$; let us denote $\mu^*(0)=\mu^*$. By the definition of $\mu^*$, when $\mu^*<\infty$, it follows that for any $\mu\in(0,\mu^*]$ and $i\in I$, 
		\begin{equation}\label{eua}
		(u_i)_{\mu}\leqslant u_i \quad {\rm in} \quad \mathbb{R}^{n}\setminus B_{\mu}(0).
		\end{equation}
		Hence, there exist $i_0\in I$ and $(\mu_{k})_{k\in\mathbb{N}}$ in $(\mu^*,\infty)$ satisfying $\mu_{k}\rightarrow\mu^*$ and such that \eqref{eua} does not hold for $i=i_0$ and $\mu=\mu_{k}$. For $\mu>0$, let us define $\omega_{\mu}=(u_{i_0})-(u_{i_0})_{\mu}$. 
		
		\noindent {\bf Step 1:} $\omega_{\mu^*}$ is superharmonic. 
		
		\noindent In fact, using \eqref{eq:subcriticalsystem}and Lemma~\ref{prop:conformalinvariance}, we get
		\begin{equation*}
		\begin{cases}
		\Delta^{2}\omega_{\mu^*}(x)=c_{\mu^*}(x)\omega_{\mu^*}& {\rm in} \quad \mathbb{R}^n\setminus B_{\mu^*}(0)\\
		\Delta \omega_{\mu^*}(x)=\omega_{\mu^*}(x)=0& {\rm on} \quad \partial B_{\mu^*}(0),
		\end{cases}
		\end{equation*}
		where
		\begin{equation*}
		c_{\mu^*}=\frac{|{\mathcal{U}}|^{s-1}u_{i_0}-|\mathcal{U}_{\mu^*}|^{s-1}({u}_{i_0})_{\mu^*}}{{u}_{i_0}-({u}_{i_0})_{\mu^*}}>0 \quad {\rm in} \quad \mathbb{R}^n\setminus B_{\mu^*}(0).
		\end{equation*}
		Therefore, by Claim 1, we can use the strong maximum principle in \cite[Theorem~3.5]{MR1814364} to conclude
		\begin{equation*}
		\min_{\mathbb{R}^{n}\setminus B_{\mu^*}(0)}\omega_{\mu^*}=\min_{\partial B_{\mu}(0)}\omega_{\mu^*}.
		\end{equation*}
		
		\noindent {\bf Step 2:} $\omega_{\mu^*}\equiv0$.
		
		\noindent Supposing that $\omega_{\mu^*}$ is not equivalently zero in $\mathbb{R}^{n}\setminus B_{\mu}(0)$, by the Hopf Lemma \cite[Lemma~3.4]{MR1814364}, we have that ${\partial_\nu}\omega_{\mu^*}>0 $ in $\partial B_{\mu^*}(0)$.
		Moreover, by the continuity of $\nabla u_{i_0}$, one can find $r_0>\mu^*$ such that for any $\mu^*\in[\mu,r_0)$, we get
		\begin{equation}\label{spo}
		\omega_{\mu^*}>0 \quad {\rm in} \quad \bar{B}_{r_0}(0)\setminus B_{\mu}(0).
		\end{equation}
		On the other hand, one can prove that there exists $\varepsilon>0$ such that for any $\mu\in[\mu^*,\mu^*+\varepsilon)$ and $x\in\mathbb{R}^{n}\setminus B_{r_0}(0)$, it follows
		\begin{equation}\label{san}
		|\omega_{\mu^*}(x)-\omega_{\mu}(x)|=|(u_i)_{\mu}(x)-(u_{i_0})_{\mu^*}(x)|\leqslant\frac{1}{2}\left(\frac{r_0}{|x|}\right)^{n-4}\min_{\partial B_{r_0}(0)}\omega_{\mu^*}.
		\end{equation}
		Therefore, a combination of \eqref{spo}--\eqref{san} yields that $\omega_{\mu^*}\geqslant0$ in $\mathbb{R}^{n}\setminus B_{\mu}(0)$ for any $\mu\in[\mu^*,\mu^*+\varepsilon)$. This is a contradiction with the definition of $\mu^*$, thus $\omega_{\mu^*}\equiv0$ in $\mathbb{R}^{n}\setminus B_{\mu}(0)$. Moreover, let us define
		\begin{equation*}
		\omega_{\mu}(x)=-\left(\frac{\mu^*}{|x|}\right)^{n-4}\omega_{\mu^*}\left(\left(\frac{\mu^*}{|x|}\right)^2x\right).
		\end{equation*}
		Hence, it follows that $\omega_{\mu^*}\equiv0$ in $\mathbb{R}^{n}\setminus\{0\}$. Since $u_{i_0}$ cannot be identically zero without contradicting the definition of $\mu^*$, and using that nonnegative solutions are weakly positive, we find $u_{i_0}$ is nowhere vanishing. 
		Therefore, we find that $|\mathcal{U}_{\mu^*}|\equiv|\mathcal{U}|$ in $\mathbb{R}^{n}\setminus\{0\}$.
		
		By definition of $\mu^*(z)$, if $\bar{\mu}(z)=\infty$, we get that $(u_i)_{z,\mu}\leqslant u_i$ in $\mathbb{R}^n\setminus B_{\mu}(z)$, for any $\mu>0$ and $i\in I$. Moreover, assuming that $z=0$, by \eqref{eua}, we have 
		\begin{equation*}
		\mu^{n-4}\leqslant \liminf_{|x|\rightarrow\infty}|x|^{n-4}u_i(x),
		\end{equation*}
		which by passing to the limit as $\mu\rightarrow\infty$ provides that for any $i\in I$, either $u_i(0)=0$ or $|x|^{n-4}u_i(x)\rightarrow0$ as $|x|\rightarrow\infty$. Using that $u_i(0)=0$ for all $i\in I$, we conclude $u_i\equiv0$. Therefore, we may assume $|x|^{n-4}u_i(x)\rightarrow\infty$ as $|x|\rightarrow\infty$ for all $i\in I_{+}$.
		
		\noindent{\bf Step 3:} $\mu^*(x)=\infty$ for all $x\in\mathbb{R}^n$.
		
		\noindent Indeed, when $\mu^*(x)<\infty$ for some $y\in\mathbb{R}^n$, using (i), we obtain
		\begin{equation*}
		|z|^{n-4}|\mathcal{U}(z)|=|z|^{n-4}|\mathcal{U}_{x,\mu^*(x)}(z)|\rightarrow\mu^*(x)^{n-4}|\mathcal{U}(z)| \quad {\rm as} \quad |z|\rightarrow\infty,
		\end{equation*}
		which is a contradiction.
		Hence, the alternative proof of the lemma is finished.
	\end{proof}
	
	\begin{acknowledgement}
		This manuscript began to be written in the first-named author's visit to the Department of Mathematics at Princeton University, whose hospitality he gratefully acknowledges. 
	\end{acknowledgement}
	

\end{document}